\documentclass[11pt]{amsart}
\usepackage{amsfonts, mathabx}
\usepackage{stmaryrd} 

\usepackage[dvipsnames]{xcolor}
\colorlet{RED}{red}
\usepackage{mathrsfs}
\usepackage[alphabetic,initials]{amsrefs}
\usepackage{enumitem}

\usepackage{pdflscape}
\usepackage{chngcntr}
\usepackage{tikz-cd}

\input xy
\xyoption{all}
\CompileMatrices
\usepackage{sseq,amssymb}
\usepackage{wasysym} 
\usepackage{todonotes}

\usepackage{pdfsync}
\usepackage{tikz}
\usepackage{mathrsfs}

\usepackage{hyperref}
\hypersetup{
  colorlinks=true,
  linkcolor=blue,
  citecolor=Thistle,
  urlcolor=blue,
  filecolor=black, 
}
\usepackage[capitalise,nameinlink,noabbrev]{cleveref}
\usepackage{gjm-cyc} 

\usepackage{forest}
\forestset{
  DCR/.style={
    for tree={
      grow'=north,
      parent anchor=north,
      child anchor=south,
      align=center,
      l sep=10pt,
      s sep=7pt,
      inner sep=2pt,
      anchor=center
    }
  },
  op/.style={
    circle,
    draw,
    minimum size=3mm,
    inner sep=1pt
  },
}

\usepackage{etoolbox}

\preto\appendix{%
  \renewcommand{\thesection}{\Alph{section}}%

  \setcounter{equation}{0}%
}

\title[Cyclotomic extensions in stable homotopy theory]
 {Cyclotomic extensions in stable homotopy theory} \author{Douglas C. Ravenel}
\address{Department of Mathematics\\
University of Rochester\\
Rochester, NY 14627}

\email{dcravenel@gmail.com}
\date{\today}

\begin{document}

\begin{abstract}
  This expository paper is a companion to \cite{Rav:gjmcyc}, in which
  we discuss cyclotomic spectra.  Both papers are intended to shed
  light on the recent resolution of the telescope conjecture by Robert
  Burklund, Jeremy Hahn, Ishan Levy and Tomer Schlank (hereafter
  referred to as BHLS) in \cite{BHLS}.  Their proof involves both
  cyclotomic spectra, the subject of \cite{Rav:gjmcyc}, and cyclotomic
  extensions of spectra, the subject of this paper.  Higher cyclotomic
  extensions of commutative ring spectra are analogous to Galois
  extensions of $p$-adic number fields (or rings of integers thereof)
  obtained by adjoining roots of unity.
\end{abstract}

\maketitle

\begin{center}
\includegraphics[width=11cm]{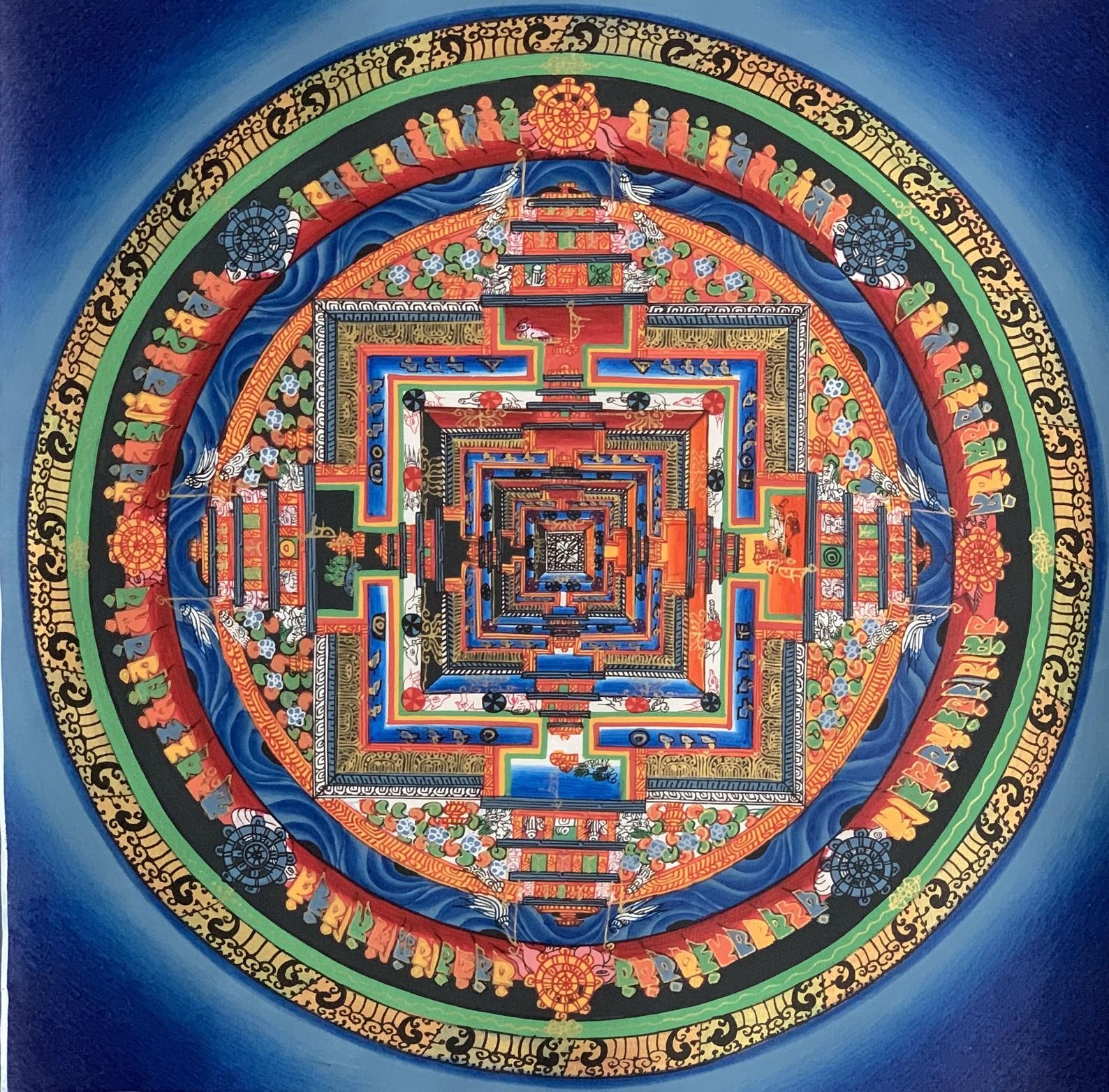}\\
Tibetan sand painting, anonymous.
\end{center}

\setcounter{tocdepth}{3}

\tableofcontents


\section{Introduction}\label{sec-ext-intro}

This expository paper is a companion to \cite{Rav:gjmcyc}, in which we
discuss cyclotomic spectra.  Both papers are intended to shed light on
the recent resolution of the telescope conjecture by Robert Burklund,
Jeremy Hahn, Ishan Levy and Tomer Schlank (hereafter referred to as
BHLS) in \cite{BHLS}.  Their proof involves both cyclotomic spectra,
the subject of \cite{Rav:gjmcyc}, and cyclotomic extensions of
spectra, the subject of this paper.

Cyclotomic spectra have an action of the circle group $S^{1}$,
which we denote by $\mT$, for torus.  The subspectra fixed by
finite subgroups of $\mT$ are required to have certain
properties. Such spectra come up in algebraic $K$-theory and its close
relatives topological Hochschild homology $\THH$ and topological
cyclic homology $\TopC$.

Higher cyclotomic extensions of commutative ring spectra are analogous
to Galois extensions of $p$-adic number fields (or rings of integers
thereof) obtained by adjoining roots of unity.  When we want to
distinguish between these two forms of cyclotomy, we will use the
terms {\bf smooth} for the spectra of \cite{Rav:gjmcyc} and {\bf
  discrete} for the extensions of this paper.

For the reader's convenience we note that
\begin{itemize}
\item [$\bullet$]   \cite[Theorem A]{BHLS} is our  \cref{thm-main-A}.

\item [$\bullet$]   \cite[Theorem B]{BHLS} is our  \cref{thm-thmB}.

\item [$\bullet$]   \cite[Theorem C]{BHLS} is our  \cref{thm-fpLQ}

\item [$\bullet$]   \cite[Theorems D and E]{BHLS} are not covered here.
\end{itemize}

\begin{remark}\label{rem-font}
  As in \cite{Rav:gjmcyc} we follow the font convention of
  \cite{BHLS} by denoting the algebraic $K$-theory of a ring spectrum
  $R$ by $\rK(R)$ (Quillen's $\rK$) to avoid confusion with its $nth$
  Morava $K$-theory $K (n)_{*} (R)$ (Jack's $K$).

  Here, as in \cite{Rav:gjmcyc} and \cite{Rav:Whatis}, we use
{\color{cyan} cyan} colored symbols to denote {\qcats} that are not
necessarily ordinary categories.
\end{remark}

In \cref{sec-players} we introduce various functors of interest
(\cref{sec-tel-loc}), describe the BHLS counterexample to the {\TC}
(\cref{sec-BHLS-counter}), review methods for showing that certain
telescopically localized coassembly maps are not equivalences
(\cref{sec-nontrivial-coassembly}), and discuss the {\LQ} property in
\cref{sec-LQ}. 

\cref{sec-Ambi} is about {\bf  semiadditivity}, which concerns the following
situation.  Suppose we have a group $G$ acting on a space, spectrum
or object in a suitable {\qcat} $\qmcC$.  This action is {\eqvt}
to a functor $F:\qBG\to \qmcC$, where $\qBG$ is the one object {\qcat}
associated with $G$.  Such a functor determines the choice of an
object $X$ in $\qmcC$ since $\qBG$ has a single object.  This functor
has a colimit $X_{hG}$, the homotopy orbit object, and a limit
$X^{hG}$, the homotopy fixed point object.

When $G$ is finite one can define a so called norm map
$\Nm_{X}:X_{hG}\to X^{hG}$.  It is known to be an equivalence when
$\qmcC=\qSp_{K(n)}$, the {\qcat} of $K(n)$-local spectra.  An {\qcat}
is said to be {\bf 1-semiadditive} (see \cref{def-semiadd}) if it has
a similar property for all functors to it from $\qBG$.  In \cite{HL}
Mike Hopkins and Jacob Lurie generalize from $BG$, a path connected
space for which $\pi_{1}$ is finite and all other homotopy groups
vanish, to an {\bf $m$-finite space} as in \cref{def-pi-finite},
meaning one with finitely many path components, each of which has
finite homotopy groups that vanish above dimension $m$.  They show
that for a $\qmcC$-valued functor from such a space one can still
define a norm map, and they say that $\qmcC$ is {\bf $m$-semiadditive}
if the norm is always an equivalence. $\qmcC$ is said to be {\bf
  $\infty$-semiadditive} if it is $m$-semiadditive for all $m$.

They prove that $\qSp_{K(n)}$ has this property.  In \cite[Theorem
A]{CSY22}, Schlank, Shachar Carmeli and Lior Yanovski show the same
for the (possibly) larger {\qcat} $\qSp_{T(n)}$. In \cite[Theorem
B]{CSY22} (our \cref{thm-CSY22-thm-B}) they show that $\qSp_{K (n)}$
and $\qSp_{T (n)}$ are respectively minimal and maximal among
1-semiadditive {\qcats}. In \cite[Theorem C]{CSY22} (our
\cref{thm-CSY22-thm-C}) they show that 1-semiadditivity is equivalent
to $\infty$-semiadditivity.

The original {\TC} asserted the equality $\qSp_{K (n)}$ and
$\qSp_{T (n)}$, both of which are $\infty$-semiadditive. This suggests
that further study of such {\qcats} could lead to insight about the
conjecture, which has indeed proven to be the case.

In \cref{sec-CCE}, which is about {\bf chromatic cyclotomic
  extensions}, we learn that between $\qSp_{K (n)}$ and $\qSp_{T (n)}$
there is an intermediate $\infty$-semiadditive {\qcat}
$\qSp_{\Cyc(n)}$, that of {\bf height $n$ cyclotomically complete
  spectra} as in \cref{def-cyc-comp}.

In order to define such extensions, we first review the cyclotomic
extensions of classical number theory in \cref{sec-classical}.  These
are Galois extensions of the field $\rats_{p} $ of $p$-adic numbers, and their
rings of integers, obtained by adjoining roots of unity.

In \cref{sec-Rognes} we review Rognes' Galois theory for commutative
ring spectra, the stable homotopy theoretic analog of classical Galois
theory.  A pivotal example of a faithful (as in
\cref{def-Rognes-3defs}) profinite Galois extension in this sense (see
\cref{prop-Rognes-5.4.9}) is the map $\mS _{K(n)} \to E_{n}(\barFFp)$
from the $K(n)$-local sphere spectrum to extended height $n$ Morava
$E$-theory.  Its Galois group is the extended Morava stabilizer group
$\GG_{n}$.  One might say that $E_{n}(\barFFp)$ is the algebraic
closure of $\mS _{K(n)}$.

It would be nice to have a similar algebraic closure of the telescopic
analog $\mS _{T(n)}$, but no telescopic analog of either $E_{n}(\barFFp)$ or
$\GG_{n}$ is known or even thought to exist.  On the other hand, by
\cite[Theorems A]{CSY24} (quoted here as \cref{thm-lifting-abelian})
we do have a telescopic analog of every {\em abelian} extension of
$\mS _{K(n)}$, including the maximal one.  These are treated in
\cref{sec-height-n}.

The abelianization of $\GG_{n}$ is known to be isomorphic to that of
the absolute Galois group of $\rats_{p}$, namely
$\Z^{\times }_{p}\times \widehat{\Z}$.  The second factor, the
profinite integers, is generated by a lift of the Frobenius element in
the absolute Galois group of $\FFp$.

It has a quotient isomorphic to $(\Z/p^{j})^{\times }$ obtained by
sending the Frobenius generator of $\widehat{Z}$ to the identity
element. This corresponds to the cyclotomic extension
$\rats_{p}\to \rats_{p}[\omega_{p^{j}}]$ obtained by adjoining a
primitive $p^{j}$th root of unity.

The analogous quotient to $\GG_{n}$ corresponds to faithful Galois
extensions
\begin{numequation}\label{eq-faithful}
\begin{split}
  \mS_{K(n)}\to \mS_{K(n)}[\omega^{(n)}_{p^{j}}]
  \qquad \aand
  \mS_{T(n)}\to \mS_{T(n)}[\omega^{(n)}_{p^{j}}].
\end{split}
\end{numequation}
 
\noindent These are the chromatic and telescopic analogs of the
classical cyclotomic extensions of \cref{sec-classical}.  Each is
obtained by adjoining a {\bf higher root of unity} as in
\cref{def-height-n-root}, which is as follows. For a commutative ring
spectrum $R$, such as $\mS_{K(n)}$ or $\mS_{T(n)}$, one has a
topological abelian group of units $R^{\times }$ underlain by the
subspace of $\Omega^{\infty}R$ comprising the path components indexed
by the invertible elements in the ring $\pi_{0}R$.  An ordinary
$p^{j}$th root of unity in $R$ is the image of a generator under a
monomorphism $\rC_{p^{j}}\to R^{\times }$. A {\bf height $n$ $p^{j}$th
  root of unity} in $R$ is a suitable map
$\rC_{p^{j}}\to \Omega^{n}R^{\times }$, which is adjoint to a map
$B^{n}\rC_{p^{j}}\to R^{\times }$.  The domain of the latter is the
 $n$-finite {\SESM} space $K(\Z/p^{j},n)$, hence the
relevance of the $n$-semiadditivity of \cref{sec-Ambi}.  Remarkably it
turns out that
\begin{displaymath}
  \mS_{T(n)}[\omega^{(n)}_{p^{j}}]
  \simeq \mS_{T(n)}\otimes K((\Z/p^{j})^{\times },n)_{+}
\end{displaymath}
 
\noindent and similarly for its $K(n)$-localization.

Taking the colimit as $j\to \infty$ in \cref{eq-faithful} yields
profinite Galois extensions as in \cref{defin-two-infinite},
\begin{displaymath}
  \mS_{K(n)}\to \mS_{K(n)}[\omega^{(n)}_{p^{\infty}}]=:R_{n}
  \qquad \aand
  \mS_{T(n)}\to \mS_{T(n)}[\omega^{(n)}_{p^{\infty}}]=:R^{\fin}_{n}.
\end{displaymath}
 
\noindent The former is known to be faithful but the latter is not.
In terms of Bousfield classes, as in \cite{AKB:BA} and
\cite[Definition 7.2.1]{Rav:NP}, we have
\begin{numequation}\label{eq-AKB-classes}
\begin{split}
  \langle \mS_{K(n)} \rangle
   = \langle R_{n} \rangle
   \leq  \langle R^{\fin}_{n} \rangle
   \leq \langle \mS_{T(n)} \rangle,
\end{split}
\end{numequation}
 
\noindent which should be compared with \cref{eq-K-cyc-T}.  The three
corresponding localization functors are denoted by $L_{n}$,
$L^{\Cyc}_{n}$ and $L^{\fin}_{n}$.  In \cref{def-cyc-comp} we say that
a $T(n)$-local spectrum is {\bf cyclotomically complete} if it is
$ R^{\fin}_{n}$\nobreakdash-local.  We do not know whether the first
inequality of \cref{eq-AKB-classes} is strict, but the counterexample
of \cite{BHLS} shows that the second one is.  Cyclotomic completion
figures prominently in their proof. 

The subject of \cref{sec-cyclo-redshift} is the charmingly named
chromatic cyclotomic redshift, ``the final key idea in giving a
counter-example to the telescope conjecture'' \cite[\S6.3]{BHLS}. The
term ``redshift'' refers to the increase of chromatic height caused by
algebraic $K$-theory or some related functor. Conversely ``blueshift''
refers to the reduction of chromatic height by the Tate construction,
which is the starting point for the semiadditivity of \cref{sec-Ambi}.
\cref{thm-redshift-B} says that for a $T(n)$-local commutative ring
spectrum $R$,
\begin{displaymath}
\rK_{T(n+1)}( R[\omega^{(n)}_{p^\infty}] )
 \simeq \rK_{T(n+1)} (R)[ \omega^{(n+1)}_{p^\infty}].
\end{displaymath}
 
In \cref{sec-unipotent} we discuss locally unipotent $\Z$-actions as
in \cref{def-loc-uni}, which are relevant to the main counterexample
of \cite{BHLS}.  The local unipotence condition on a group action can
be thought of as ``nearly trivial.''

Suppose we have a spectrum $X$ acted on by a group $G$, to which
we apply a functor $F$.
There is a {\bf  coassembly map} as in
\cite[Definition 5.6]{Rav:gjmcyc},
\begin{numequation}\label{eq-coassembly}
\begin{split}
  \epsilon:F(X^{hG})\to F(X)^{hG},
\end{split}
\end{numequation}
 
\noindent which is an equivalence if $F$ preserve
limits. \cite[Theorem 3.22]{BHLS}, quoted here as \cref{thm-failure},
says that if $R$ is a $T(n)$-local $\EE_{1}$-ring spectrum for
$n\geq 1$ for which $\rK_{T (n+1)} (R)$ is nontrivial, then the map
\begin{displaymath}
\epsilon:\rK_{T (n+1)}(R^{B\Z})\to \rK_{T (n+1)}(R)^{B\Z},
\end{displaymath}
 
\noindent the coassembly map for the trivial $\Z$-action, is {\em not}
an equivalence.

In the main counterexample of \cite{BHLS}, the ring of interest is
\begin{displaymath}
R=L_{T(n)}\BPn
\end{displaymath}
 
\noindent
equipped with a locally unipotent $\Z$-action by
certain Adams operations described in \cref{sec-Adams-BPn}.  In order
to apply \cref{thm-failure} we need to show that this nontrivial group
action is functorially related to one that is trivial.  We have two
tools at our disposal:
\begin{itemize}
\item [$\bullet$]  replace the group $\Z$ by a subgroup $p^{k}\Z$ and
\item [$\bullet$]  smash with a suitable finite spectrum. 
\end{itemize}

\noindent \cref{thm-fpLQ} says that for $k\gg 0$ we can smash with a
finite spectrum of type $n+2$ and ``trivialize'' the $\Z$-action on
$\THH(R)$ in the sense that its homotopy fixed point set behaves like
that of the trivial action.  \cref{thm-tel-asym-const} and
\cref{thm-thmB} are similar statements for $\TopC(R)$ and
$\TopC_{T(n+1)}(\BPn)$.

In \cref{sec-main-counter} we sketch the proof that the map of
\cref{eq-coassembly} is an equivalence for $F=\rK_{K(n+1)}$ but not
for $F=\rK_{T(n+1)}$, thus showing that the two functors are
different.

We need \cref{thm-thmB} as well as \cref{thm-tel-asym-const} because
\cref{thm-purity} and \cref{cor-BHLS-6.3} pertain only to {\em
  connective} ring spectra.  The latter says that for a
$T(n+1)$-acyclic, connective $\mathbb{E}_1$-algebra $R$ with a
$\Z$-action, such as $\BPn$, there is a diagram
\begin{displaymath}
\xymatrix
@C=10mm
@R=6mm
{
\rK_{T(n+1)}(L_{T(n)} R^{h(p^{k}\Z)}) 
 \ar[r]^(.5){\epsilon }
  &
\rK_{T(n+1)}(L_{T(n)} R)^{h(p^{k}\Z)} \\
\rK_{T(n+1)}( R^{h(p^{k}\Z)}) \ar[r]^(.5){\epsilon }
\ar[d]_{\simeq} \ar[u]^{\simeq}   
   &
\rK_{T(n+1)}(R)^{h(p^{k}\Z)}
   \ar[d]_{\simeq} \ar[u]^{\simeq}   \\
\TopC_{T(n+1)}(R^{h(p^{k}\Z)}) \ar[r]^(.5){\epsilon }
   &
\TopC_{T(n+1)}(R)^{h(p^{k}\Z)},
}
\end{displaymath}
 
\noindent This means we can replace the functor $\rK_{T(n+1)}L_{T(n)}$
with the more accessible $\TopC_{T(n+1)}$.

\cref{thm-BHLS-6.25} is the specialization of the above to the case
where $R=\BPn$ for $n\geq 1$ with its Adams operations.  In that case it can be
shown the the coassembly map is height $n+1$ cyclotomic completion.
Hence the coassembly map on
\begin{displaymath}
\rK_{\Cyc(n+1)}(BP\langle n \rangle^{h(p^{k}\Z)})
\end{displaymath}
 
\noindent
is an equivalence, while the middle coassembly map above is not.  Thus
the functors $L_{T(n_{1})}$ and $L_{\Cyc(n+1)}$, so the {\TC} is false.

There are two appendices, each added to clarify a needed but
complicated definition.  \cref{sec-qcat-defs} covers various {\qcatal}
notions including that of a hypersheaf (\cref{defin-hypersheaf}),
which is needed in \cref{prop-BMCSY-5.11} and \cref{thm-redshift-B}.
\cref{sec-BV-tensor} covers operads in order to describe the
multiplicative structure preserved by the Adams operations on $\BPn$
discussed in \cref{sec-Adams-BPn}.

\bigskip

It is a pleasure to acknowledge helpful conversations with Tomer
Schlank, Ishan Levy, Jeremy Hahn, Robert Burklund, Hari Rau-Murthy,
Siddharth Gurumurthy, John Rognes, Mike Mandell, Inbar Klang, Liam
Keenan, and Mike Hopkins.

\section{The players}\label{sec-players}
\subsection{The chromatic landscape}\label{sec-tel-loc}

\begin{defin}\label{def-MR99} 
{\phantom{hello}}

\begin{enumerate}[label={(\roman*)},itemindent=0em]
\item \label{def-MR99ii} \cite[Definition 1.5.3.]{Rav:NP} A $p$-local
  finite spectrum $X$ {\bf has type $n$} if\linebreak
  $K (n)_{*}X\neq 0$ and $K (m)_{*}X=0$ for all $m<n$.  We will
  sometimes denote such a spectrum by $F(n)$.

\item \label{def-MR99i} \cite[\S3]{MahoRezk99} A $p$-complete bounded
  below spectrum $Y$ {\bf has {\fp}-type $n$} if it has finite type
  and for each finite spectrum $F(n+1)$ of type $n+1$ as above,
  $F(n+1)\otimes Y$ is $\pi$-finite, meaning that is has only finitely
  many nontrivial homotopy groups, each of which is finite.
\end{enumerate}

In both cases we say the spectrum has {\bf (chromatic) height $n$}.
\end{defin}

It is known \cite[Theorem 2.11]{Rav:Loc} that for finite spectra $X$,
$K (m)_{*}X=0$ implies $K (m-1)_{*}X=0$, but this is far from true for
infinite CW-spectra.  For example there is a infinite spectrum
$\BPn$ for which $K (m)_{*}\BPn$ is trivial for $m>n$ but nontrivial
for $m\leq n$, and this spectrum has {\fp}-type $n$.  An infinite
spectrum need not have an {\fp}-type, but a finite one always has a
type.

The {\TC} is about the relation between Bousfield localizations
(originally defined by Pete Bousfield in \cite{AKB:Spectra}) with
respect to $K (n)$, the $n$th Morava $K$-theory, and ``the'' telescope
$T (n)$. The quotation marks refer to the following consequences of
the periodicity theorem of  Jeff Smith and Hopkins \cite{HS}:

\begin{itemize}

\item [$\bullet$] For each prime $p$ and positive integer $n$ (the
   chromatic height) there are $p$-local finite spectra of type
  $n$ as in \cref{def-MR99}\cref{def-MR99ii}.

\item [$\bullet$] Each type $n$ finite spectrum $F(n)$ admits a map
$v:F(n) \to \Sigma^{-d}F(n)$ for some $d>0$ inducing an isomorphism
in $K (n)$ homology.  The cofiber of $v$ has type $n+1$, so it is
possible to construct finite spectra of all types by induction on $n$.
\end{itemize}

We denote the filtered colimit obtained by iterating $v$, the {\bf
  height $n$ telescope}, by $T (n)$. For a given prime $p$ and height
$n$, neither $F(n)$ nor $v$ is unique, although for a given finite
$F(n)$, the homotopy type of its telescope $T (n)$ is known to be
independent of the choice of $v$.  Better still, the Bousfield
localization functor $L_{T (n)}$ associated with $T (n)$ is known to
be independent of the choice of $F(n)$ as well. Hence it is customary
to omit $F(n)$ and $v$ from the notation.
 
The original conjecture of \cite{Rav:Loc} was that the functors
$L_{T (n)}$ and $L_{K (n)}$ are the same, which was then known to be
the case for $n=0$ and $n=1$.  A few years later (1989) it became
apparent that the statement was likely to be false for $n>1$.  The
author made several unsuccessful attempts to disprove it, one of which
is the subject of \cite{MRS4}.  None of them were anything like the
work of BHLS.

There is an equality of Bousfield classes \cite[Definition 7.2.1 and
Theorem 7.3.2 (d)]{Rav:NP},
\begin{numequation}\label{eq-En-Bousfield-class}
\begin{split}
\langle E_{n} \rangle 
 = \langle E (n) \rangle 
 = \bigoplus_{0\leq m\leq n}\langle K (m) \rangle.
\end{split}
\end{numequation}%

\begin{defin}\label{def-Ln-fin}
{\bf The functors $L_{n}$ and $L_{n}^{\fin}$.}  The former is
Bousfield localization with respect to $E (n)$, periodic
Johnson-Wilson theory, or {\eqt}ly with respect to $E_{n}$, Morava
$E$-theory. We denote the {\qcat} of $E(n)$-local spectra by $L_{n}\qSp$.

$L_{n}^{\fin}$ is Bousfield localization with respect to 
\begin{displaymath}
\bigoplus_{0\leq m\leq n}T (m).
\end{displaymath}

\noindent This functor is independent of the choice of the $T(m)$s. We
denote  the corresponding category of local
spectra by $L^{\fin}_{n}\qSp$.
\end{defin} 

The notation $L_{n}^{\fin}$ is used because it is known that the fiber
of the map $X\to L_{n}^{\fin}X$ is the colimit of all {\em finite} $E
(n)$-acyclic spectra mapping to $X$.  See Haynes Miller's
\cite{Mil:Fin} and \cite[\S2]{BMCSY23} for more discussion.

Both $L_{n}$ and $L_{n}^{\fin}$ are known to be {\bf smashing}
\cite[Theorem 7.5.6]{Rav:NP}, meaning that for any spectrum $X$,
\begin{displaymath}
L_{n}X\simeq X\wedge L_{n}\mS
\qquad \aand 
L_{n}^{\fin}X\simeq X\wedge L_{n}^{\fin}\mS.
\end{displaymath}

In \cref{def-cyc-comp} we will define the functors $L_{n}^{\Cyc}$
(height $n$ cyclotomic completion) and $L_{\Cyc(n)}$ that interpolate
respectively between $L_{n}^{\fin}$ and $L_{n}$, and between
$L_{T(n)}$ and $L_{K(n)}$.

As in \cite{Rav:gjmcyc} we will use the following notation.

\begin{defin}\label{def-KT-KK}
{\bf Telescopic, cyclotomically complete and chromatic localizations
    of $K$-theory and $\TopC$.}  For an $\EE_{1}$-ring spectrum $R$
  and $n\geq 0$,
\begin{align*}
\rK_{T (n)} (R)
 & := L_{T (n)}\rK(R),  &
\TopC_{T (n)} (R)
 & := L_{T (n)}\TopC (R),  \\
\rK_{\Cyc(n)} (R)
 & := L_{\Cyc(n)}\rK(R),  &
\TopC_{\Cyc(n)} (R)
 & := L_{\Cyc(n)}\TopC (R),  \\
 &\mbox{where $L_{\Cyc(n)}$
   is as in \cref{def-cyc-comp} below,}\hspace{-10cm}\\ 
\rK_{K (n)} (R)
 & := L_{K (n)}\rK(R) , &
\aand \TopC_{K (n)} (R)
 & := L_{K (n)}\TopC (R). 
\end{align*}
\end{defin}

The notation $\rK_{T (n)}$ is used in \cite{BMCSY23}, but not in most
other papers in this area.

\subsection{The Burklund-Hahn-Levy-Schlank counterexample}
\label{sec-BHLS-counter}

The authors of \cite{BHLS} show that for each $n\geq 1$ and each prime
$p$, there is a $p$-local $\EE_{1}$-ring spectrum $R$ of chromatic
height $n$ such that $\TopC_{K (n+1)} (R)$ and $\TopC_{T (n+1)} (R)$
are distinct.  More precisely, for each prime $p$ and each integer
$n>0$, they consider a form of the Johnson-Wilson spectrum $\BPn$
(which has fp-type $n$), originally defined by Dave Johnson and Steve
Wilson in \cite{JW2} and revisited half a century later by Dylan
Wilson (no relation to Steve) and Hahn in \cite{Hahn-Wilson}.  In
\cite[\S 5]{BHLS} the authors define an action of the additive group of
integers $\Z$, and hence of its subgroups $p^{k}\Z$ for $k>0$, via
Adams operations, which we will discuss in \cref{sec-Adams-BPn}.  They
prove that the $T(n+1)$-local $K$-theoretic coassembly map of
\cite[Definition 5.6]{Rav:gjmcyc}

\begin{numequation}\label{eq-loc-coassembly}
\begin{split}
\epsilon:\rK_{T (n+1)}(\BPn^{h(p^{k}\Z)})
         \to  \rK_{T (n+1)}(\BPn)^{h(p^{k}\Z)}
\end{split}
\end{numequation}%

\noindent is not an equivalence, but becomes one after $K(n +
1)$-localization.  In other words, the algebraic $K$-theory functor on
ring spectra with $\Z$-action does does not commute with passage to
homotopy fixed points, even $T (n+1)$-locally, but its does so commute
$K (n+1)$-locally.  Thus, the telescope conjecture fails at heights
greater than 1.  In their words,

\begin{quote}
We do this by looking at the coassembly map from two highly divergent
perspectives, which are connected via trace theorems:
\begin{enumerate}
\item From the perspective of locally unipotent $\Z$-actions
  on ring spectra, the results of \cite[\S4]{BHLS} tell us that the
  coassembly map cannot be an isomorphism. [In
  \cref{sec-nontrivial-coassembly} we will say that it is {\bf
    nontrivial} when it is not an isomorphism.]

  \item From the perspective of cyclotomic redshift of \cite{BMCSY23},
the map 
\[ L_{T(n)}\BPn ^{h(p^k \Z)}
\;\longrightarrow\; L_{T(n)}\BPn
\] 
\noindent splits after base change to the maximal abelian extension of
the $K(n)$-local sphere [see \cref{defin-two-infinite}], and therefore
the coassembly map is a $K(n+1)$-local isomorphism.
\end{enumerate}
\end{quote}

\begin{remark}\label{remark-why-BPn} {\bf Why $\BPn$?}  While $\BPn$ has less
  structure than its relatives $E_{n}$ and $E_{n}(\barFFp)$,
  it has the advantage of being connective. (The homotopy groups of
  these three spectra are given in \cref{eq-coefficients}.) We need
  connectivity to apply \cite[Corollary 3.21]{BHLS} (quoted as
  \cite[Proposition 6.6]{Rav:gjmcyc}) to show that a certain
  $T(n+1)$-local coassembly map is not an equivalence. 
\end{remark}

In that statement the $\Z$-action on the connective ring spectrum is
assumed to be trivial, but the action here is nontrivial.  We only
know that it is locally unipotent as in \cref{def-loc-uni}.  Such
actions are sketched in \cref{sec-unipotent}.  In \cref{thm-thmB}, we
will see that our action on $\BPn$ becomes trivial after smashing with
a suitable finite complex.

\subsection{Nontrivial coassembly maps}
\label{sec-nontrivial-coassembly}
The relevant results of \cite[\S4]{BHLS} are reported in
\cite[\S6]{Rav:gjmcyc}.  They have to do with a
connective ring spectrum $R$ on which $\Z$ acts trivially, leading to
a similar actions on $\THH(R)$ and $\TopC(R) $.  It turns out that the
coassembly maps $\epsilon$ in these cases are far from being
isomorphisms.

\cite[Proposition 3.19]{BHLS}, quoted as
\cite[Proposition 6.6]{Rav:gjmcyc}, identifies the fiber of
coassembly in the $\THH$ case when $R=\mS _{p}$, the $p$-local sphere
spectrum.  It also guarantees the nontriviality of the fiber in the
$\TopC$ case for general $R$ provided that $\TopC (R) $ is nontrivial.

The example of interest is a {\em nontrivial} action of $\Z$ (via
Adams operations) on the connective ring spectrum $R=\BPn$.  This
action is locally unipotent (\cref{def-loc-uni}) after $p$-adic
completion, and such actions are discussed here in
\cref{sec-unipotent}.  \cref{thm-fpLQ} and \cref{thm-tel-asym-const}
say that if such a spectrum has certain properties (that $\BPn$ is
known to have), then smashing it with a type $n+2$ finite complex
renders the $Z$-action trivial on $\THH (R)$ and on $\TopC(R)$
respectively.  Hence we know that the $T(n+1)$-local coassembly maps
are nontrivial.

BHLS take their discrete cyclotomic extensions and apply functors such
as algebraic $K$-theory, $\TopC$ and $\THH$, which are defined in
terms of smooth cyclotomy.  Their cyclotomic completion
(\cref{def-cyc-comp}) has to do with discrete cyclotomy, while
cyclotomic boundedness (in relation to the Antieau-Nikolaus
$t$-structure of \cite{AN21}) has to do with smooth cyclotomy.  Most
of their proof takes place in the {\qcat} $\qCycSp$ of smoothly
cyclotomic spectra as in
\cite[Definition 5.24]{Rav:gjmcyc}.

\cite{Rav:gjmcyc} ends by quoting \cite[Theorem 3.22 for
$X=\mS$]{BHLS}, as \cite[Theorem 6.7]{Rav:gjmcyc}, which
says that a map similar to \cref{eq-loc-coassembly} is not an
equivalence.  In it $\BPn$ is replaced by a $T (n)$-local ring
spectrum $R$ for which $\rK_{T (n+1)} (R)$ (see \cref{def-KT-KK}) is
nontrivial, equipped with a {\em trivial} action of $\Z$.

The object of this paper is twofold:
\begin{enumerate}[label={(\roman*)},itemindent=0em]
\item To outline the proof
that the chromatic analog of \cref{eq-loc-coassembly}, the coassembly
map
\begin{displaymath}
\rK_{K (n+1)}(\BPn^{h(p^{k}\Z)}) \to  \rK_{K (n+1)}(\BPn)^{h(p^{k}\Z)},
\end{displaymath}
 
\noindent {\em is an equivalence} for a certain nontrivial action of
$\Z$ by Adams operations.  We will see in \cref{eq-K-cyc-T} that there
is an intermediate functor between $K(n)$ and $T(n)$ localizations,
height $n$ cyclotomic completion $L_{\Cyc(n)}$ as in
\cref{def-cyc-comp}.  It turns out that the map of
\cref{eq-loc-coassembly} with $T(n+1)$-localization replaced by height
$n+1$ cyclotomic completion is also an equivalence. See \cref{thm-BHLS-6.25}.

\item To show that said group action (which is locally unipotent as in
  \cref{def-loc-uni}), when restricted to the subgroup $p^{k}\Z$ for
  $k\gg 0$, is close enough to trivial to show that
  \cref{eq-loc-coassembly} {\em is not an equivalence.}
 A $\TopC $ version of this is
  given in \cref{thm-thmB}, and \cref{thm-BHLS-6.25} relates it to
  algebraic $K$-theory.

  Then we can use the fact that when $\Z$ acts trivially on a ring
  spectrum $R$ for which $\rK_{T (n+1)} (R)$ (see \cref{def-KT-KK}) is
  nontrivial, the coassembly map
\begin{displaymath}
\epsilon :\rK_{T (n+1)} (R^{B\Z}) \to \rK_{T (n+1)} (R)^{B\Z},
\end{displaymath}

\noindent is not  an equivalence by \cite[Theorem 3.22 for
$X=\mS$]{BHLS}, which is quoted as
\cite[Theorem 6.7]{Rav:gjmcyc}.
\end{enumerate}

\subsection{The {\LQ} property}\label{sec-LQ}

\begin{defin}\label{def-LQ}
\cite[Definition 4.2]{BHLS} 
Let $n\geq -1$. We say that an $\EE_{1}$-ring $R$ has the {\bf height
$n$ {\LQ} property} if $\THH(R)$ is bounded below and, for each finite
spectrum $F$ of type $n+2$, the  spectrum $F \otimes
\THH(R)$ is bounded above in the Postnikov sense.  
\end{defin}

The bounded above condition means that $F \otimes \THH(R)$ is a
colimit of spectra with finite Postnikov systems (its $m$-connected
covers), so all of its Morava $K$-theories vanish.  Since $K(m)_{*}F=0$
iff $m\leq n+1$, we must have $K(m)_{*}\THH(R)=0$ for $m\geq n+2$.

This condition on $F \otimes \THH(R)$ is also known to be {\eqt}
to \mbox{$F \otimes \THH(R)$}  being bounded above in the
Antieau-Nikolaus $t$-structure of \cite{AN21}, which is described in
\cite[\S5.11]{Rav:gjmcyc}.
Recall that the
cyclotomic spectrum $\THH (H \FFp )$, the subject of 
\cite[\S5.9]{Rav:gjmcyc},
is bounded above in the Antieau-Nikolaus $t$-structure, but its
underlying spectrum is not bounded above in the (classical) Postnikov
$t$-structure.

\begin{theorem}\label{thm-LQ-TC}
{\bf The  {\LQ} property and localized $\TopC $.}
\cite[Theorem 3.4.1]{Hahn-Wilson} For $R$ as in \cref{def-LQ} with
finite type, $\TopC(R) $ has $\fp$-type at most $n+1$. In particular,
the map
\[
\mathrm{TC}(R) \longrightarrow \mathrm{TC}_{T(n+1)}(R)
\]
is truncated, meaning that the homotopy groups  of its fiber are bounded above. 
\end{theorem}

\section{Semiadditivity}\label{sec-Ambi}

Semiadditivity (along with ambidexterity, which we do not need here)
is first introduced in the unpublished masterpiece \cite{HL} of
Hopkins and Lurie.  Much of the subsequent work in this area is due to
Carmeli, Schlank and Yanovski.  Their three papers,
\cite{CSY22}, \cite{CSY21} and \cite{CSY24}, appeared as preprints in
the stated order, even though \cite{CSY21} was published before
\cite{CSY22}.  They also wrote \cite{BCSY24} with Toby Barthel and
\cite{BMCSY23} with Shay Ben-Moshe.

In \cite{CSY21} they say

\begin{quote}
The localizations $L_{K (n)}$ and $L_{T (n)}$ are known to possess
several rather special and remarkable properties. Among them are the
vanishing of the Tate construction for finite group actions
\cite{CM-short, Hovey-Sadofsky, Greenlees-Sadofsky, Kuhn-Tate}; see
\cref{thm-blueshift}. In \cite{HL}, Hopkins and Lurie reinterpret this
Tate vanishing property as $1$-semiadditivity (see
\cref{def-semiadd}), and vastly generalize it by showing that the
$\infty$-categories $\mathrm{\qSp}_{K(n)}$ are $\infty$-semiadditive;
see \cref{thm-HL002}\cref{thm-HL002i}. In turn, this is exploited to
obtain new structural results for $\mathrm{\qSp}_{K(n)}$.

In \cite[Theorem~B, our \cref{thm-CSY22-thm-B}]{CSY22}  the authors
extended the results of \cite{HL} by classifying all the higher
semiadditive localizations of $\mathrm{\qSp}$ with respect to homotopy
rings. First, for all such localizations, $1$-semiadditivity was shown
to be equivalent to $\infty$-semiadditivity. Second, the telescopic
localizations $L_{T(n)}$, for various primes $p$ and heights $n$, were
shown to be precisely the maximal examples of such localizations
(while the $L_{K(n)}$ are the minimal).

Concisely put, among localizations of spectra with respect to homotopy
rings, the higher semiadditive property singles out precisely the
monochromatic localizations, which are parameterized by the chromatic
height.
\end{quote}

I must say that I had no idea that $K (n)$-localization had such
properties when I wrote \cite{Rav:Loc}!



\bigskip

The starting point of this theory is the following result of Mark
Hovey and Hal Sadofsky, which is related to \cite[Conjecture
12]{MR}. See \cref{rem-height}.  It is stated in terms of the Tate
construction $(-)^{tG}$ for a finite group $G$, which was originally
defined by John Greenlees and Peter May in
\cite[Introduction]{GreeMay-Tate} (quoted as
\cite[(4.26)]{Rav:gjmcyc}) to be the cofiber of a map from
the homotopy orbit spectrum to the homotopy fixed point spectrum.

\begin{thm}\label{thm-blueshift}
  {\bf Algebraic blueshift.} \cite[Theorem 1.1]{Hovey-Sadofsky}
  For a finite spectrum $X$ with trivial action by a finite group $G$,
  the Bousfield class of $(L_{n}X)^{tG}$ is that of $L_{n-1}X$.
\end{thm}

We now know that the finiteness of $X$ and the triviality of the
$G$-action are unnecessary hypotheses.

The term {\bf blueshift} refers to the lowering of chromatic height,
and hence of ``wavelength'' or periodicity. The term {\bf redshift}
refers to the raising of chromatic height, which we will see in
\cref{sec-cyclo-redshift}.

If a finite spectrum $X$ has type $n$, meaning that $L_{n-1}X$ is
contractible but $L_{n}X$ is not, \cref{thm-blueshift} says that the
Tate spectrum of $L_{n}X=L_{K (n)}X$ is contractible, making the norm
map $\Nm_{G}:X_{hG}\to X^{hG}$ of
\cite[(4.28)]{Rav:gjmcyc}
(of which the cofiber is the Tate spectrum) an equivalence.  To put it
another way, for a type $n$ finite spectrum $X$ with trivial
$G$-action, the norm map $\Nm_{G}$, which appears in Greenlees-May's
``Tate diagram'' of \cite[(4.35)]{Rav:gjmcyc}, induces
an isomorphism in $K (n)_{*} (-)$. A related result of Greenlees and
Sadofsky \cite[Theorem 1.1]{Greenlees-Sadofsky}, from the same year as
the Hovey-Sadofsky result, says that the Tate spectrum for $K (n)$
itself (with trivial $G$-action) is contractible.

If we assume instead that $X$ is $K (n)$-local, then so is the
homotopy fixed point spectrum $X^{hG}$, although the homotopy orbit
spectrum $X_{hG}$ need not be. \cite[Theorem 0.0.1]{HL} says that for
{\em any action} of $G$ on $X$, the norm map exhibits $X^{hG}$ as a
$K (n)$-localization of $X_{hG}$.  Their goal is to generalize this
result in a way that involves the following.

\begin{defin}\label{def-pi-finite}
  {\bf $\pi $-finiteness, truncation and connectivity of spaces.}
  A space (or Kan complex) $X$ is {\bf $\pi $-finite} if it has
  finitely many path connected components, each of its homotopy groups
  $\pi_{i} (X,x)$ for any base point (or vertex) $x$ is finite, and
  for each $x$ they vanish for $i\gg 0$.  It is {\bf $m$-finite} if
  these groups vanish for $i>m$.  

  It is {\bf $m$-truncated} for $m\geq 0$ if its homotopy groups
  vanish above dimension $m$. They need not be finite in low
  dimensions.  It is {\bf $(-1)$-truncated} if it is empty or
  contractible, and {\bf $(-2)$-truncated} if it is contractible.

It is {\bf $m$-connective} for $m\geq 0$ if its homotopy groups vanish below
dimension $m$. 

For a prime $p$, a {\bf $p$-space} is one whose homotopy groups are
all finite $p$-groups.

An {\bf $m$-finite (respectively $m$-connective or $\pi $-finite) map}
$q : A \to B$ is one for which $q^{-1} (b)$ is $m$-finite
$m$-connective or($\pi $-finite) for all $b\in B$.

An {\bf $m$-finite limit or colimit} is one that is indexed by an
$m$-finite space.
\end{defin}

For an {\qcatal} generalization of the above see \cref{defin-truncated-object}.

The word {\bf truncation} here refers to the process of killing
homotopy groups above dimension $m$ by attaching $k$-cells for
$k\geq m+2$.  This process is also known as {\bf  passing to the $m$th
Postnikov section}.

The definition of $m$-finiteness above is such that the space of paths
between any pair of points in an $m$-finite space is $(m-1)$-finite,
even when $m\leq 0$.  

A map $q:A\to B$ is $(-2)$-finite if it is an equivalence, it is
$(-1)$-finite if the preimage of each component of $B$ is either empty
or is mapped {\eqt}ly to it,  it is $0$-finite if the preimage of each
component of $B$ is {\eqt} to a finite covering of it, and so on.

As before let $\qSp_{K (n)}$ be the {\qcat} of $K (n)$-local spectra.
Then a $K (n)$-local spectrum $X$ with $G$ action is a functor
$\rho :{\color{cyan}BG}\to \qSp_{K (n)}$ for ${\color{cyan}BG}$ the
one object {\qcat} associated with $G$ as in
\cite[Definition 5.1]{Rav:gjmcyc}.  Its limit and colimit are
$X^{hG}$ and $L_{K (n)}X_{hG}$.  The norm map of
\cite[(4.28)]{Rav:gjmcyc} extends uniquely to
$L_{K(n)}X_{hG}$ (since the target is $K(n)$-local), and we can ask if
this extension is an equivalence.

The space $BG$ for a finite group $G$ has 
\begin{displaymath}
\pi_{i}BG
=\mycases{    
G      &\mbox{for }i=1\\
0      &\mbox{otherwise.}
}\end{displaymath}

\noindent {\em The key insight of \cite{HL} is that similar statements
  hold not just for functors ${\color{cyan}BG}\to \qSp_{K (n)}$, but
  for functors ${\color{cyan}X}\to \qSp_{K (n)}$ where $X$ 
  is $\pi$-finite as in \cref{def-pi-finite}.}  The case
$X=B^{m}\rC_{p^{j}}$ for $m,j>0$ is of particular interest.

\begin{theorem}\label{thm-HL002} 
{\bf The Hopkins-Lurie norm.}  
Let $X$ be a $\pi$-finite space or Kan complex.

\begin{enumerate}[label={(\roman*)},itemindent=0em]

\item \label{thm-HL002i}\cite[Theorem 0.0.2]{HL} For any functor $\rho
:{\color{cyan}X} \to \qSp_{K (n)}$, there is a
canonical equivalence
\begin{displaymath}
\xymatrix
@R=4mm
@C=10mm
{
{\Nm_{X} :\colim{\color{cyan}X}\rho }
    \ar[r]^(.62){\simeq }
  &{\lim{\color{cyan}X}\rho.}
}
\end{displaymath}

\item \label{thm-HL002ii}\cite[Theorem A]{CSY22}
The same holds for any functor ${\color{cyan}X} \to \qSp_{T (n)}$.
\end{enumerate}
\end{theorem}

This norm map is studied more generally by Lurie in
\cite[6.1.6]{Lurie:HA}.  For a suitable {\qcat} $\qmcC$ and a map of
Kan complexes $f:X\to Y$, one has a functor
$f^{*}:\qmcC^{Y}\to \qmcC^{X}$ (sometimes called a pullback functor)
with right and left adjoints $f_{*}, f_{!}:\qmcC^{X}\to \qmcC^{Y}$
(as in \cref{defin-space-indexed}), sometimes called the pushforward
and extraordinary pushforward.  When $Y=\bdelt^{0}$,
$f^{*}:\qmcC \to \qmcC^{X}$ is the diagonal functor, and its adjoints
are the homotopy colimit and homotopy limit functors.  When in
addition $X=BG$ for a finite group $G$, they are the homotopy orbit
and homotopy fixed point functors.

The definition of this map requires some ``intricate categorical
constructions'' given in \cite[\S4]{HL}. An easier way to do it is
given by Yonatan Harpaz in \cite{Harpaz}.  He defines an {\qcat}
$\qSpan_{m}$ (\cref{def-Harpaz-cat}) of spans of $m$-finite spaces and
shows that it has a certain universal property with respect to
$m$-semiadditive {\qcats}, as in \cref{def-pi-finite}.

Here is Harpaz' definition.  In \cite[Theorem 4.1]{Harpaz} he shows
that this category has a universal property that is convenient for the
study of $m$-semiadditive {\qcats} as in \cref{def-semiadd}.

\begin{defin}\label{def-Harpaz-cat}
{\bf The $m$th Harpaz {\qcat}.} \cite[Definition 2.12]{Harpaz} 
The symmetric monoidal {\qcat} $(\qSpan_{m}, \times ,\pt)$ of
spans of $m$-finite spaces has $m$-finite spaces as in
\cref{def-pi-finite} as objects. A morphism from $X$ to $Y$ is a
diagram (called a {\bf span}) $X\leftarrow W\rightarrow Y$ (as is a
morphism from $Y$ to $X$) in which $W$ is also $m$-finite.  Its
composite with $Y\leftarrow W'\rightarrow Z$ is
$X\leftarrow W''\rightarrow Z$, which is derived from the diagram
\begin{displaymath}
\xymatrix
@R=2mm
@C=3mm
{  &  &{W''}\ar[dr]^(.5){}\ar[dl]^(.5){}\\
   &{W}\ar[dr]^(.5){}\ar[dl]^(.5){}
      &   &{W'}\ar[dr]^(.5){}\ar[dl]^(.5){}\\
{X}&  &{Y}&  &{Z}
}
\end{displaymath}

\noindent in which the square is a pullback diagram.  
\end{defin}

\begin{defin}\label{def-semiadd}
An {\qcat} $\qmcC$ is {\bf $m$-semiadditive} if for every
functor\linebreak $\rho :{\color{cyan}X} \to \qmcC$ from an $m$-finite
space $X$ as in \cref{def-pi-finite}, there is an equivalence
$\Nm_{X}$ as in \cref{thm-HL002}.  It is {\bf $\infty $-semiadditive}
if it is $m$-semiadditive for all $m>0$.
\end{defin}

Infinite semiadditivity for $\qmcC$ does not imply that there is an
equivalence $\Nm_{X}$ for a functor $\rho :{\color{cyan}X} \to \qmcC$
for {\em any} space $X$, but only for spaces $X$ with finite
Postnikov towers and finite homotopy groups.

\begin{ex}
{\bf The meaning of $m$-semiadditivity for small $m$.}
\begin{itemize}

\item [$\bullet$] $(-2)$-{\SAy} means that the limit and colimit of a
functor from the trivial category (the one with a single object and a
single morphism) are {\eqt}.  Such a functor is simply the choice of
an object in $\qmcC$, with both the limit and colimit being the object
itself.  Hence {\em every {\qcat} is $(-2)$-{\SA}. }

\item [$\bullet$] $(-1)$-{\SAy} means that in addition the limit and
colimit are {\eqt} for any functor from the empty category.  There is
only one such functor, and its limit and colimit are terminal and
initial objects in $\qmcC$. Hence {\em $\qmcC$ is $(-1)$-semiadditive iff
it is pointed.}

\item [$\bullet$] $0$-{\SAy} means in addition that the limit and
colimit are {\eqt} for any functor from a finite discrete
category. This means that finite products and coproducts are the same,
{\em which is the case in any stable {\qcat}}.

\item [$\bullet$] $1$-{\SAy} means that {\em for any finite group $G$},
the Greenlees-May norm map $\Nm_{G}^{X}:X_{hG}\to X^{hG}$ of
\cite[(4.28)]{Rav:gjmcyc} is an equivalence for any
$X$-valued functor from ${\color{cyan}BG}$, which defines a
$G$-action on $X$.  Here ${\color{cyan}BG}$ denotes the nerve of the
one object category $\mathcal{B}G$ associated with $G$.  This Kan
complex has a single vertex, so a functor from the corresponding
${\color{cyan}BG}$ (meaning a map of simplicial sets) identifies a
single object $X$ in $\qmcC$ and defines a $G$-action on it. Hence the
Tate object $X^{tG}$ (the cofiber of $\Nm_{G}^{X}$) is contractible.

\item [$\bullet$] Higher {\SAy} means that one can replace
${\color{cyan}BG}$ for a finite group $G$ as above by an $m$-finite space
and still have contractible Tate objects.
\end{itemize}
\end{ex}

Thus \cref{thm-HL002} says that $\qSp_{K (n)}$ and $\qSp_{T (n)}$ are
$\infty $-semiadditive.  It turns out that certain other localizations of
$\qSp$ also have this property.

\begin{theorem}\label{thm-CSY22-thm-B}
  \cite[Theorem B]{CSY22}
  Let $R$ be a non-zero $p$-local homotopy
  ring spectrum.  The $\infty$-category $\qSp_R$ is $1$-semiadditive
  if and only if there exists a (necessarily unique) integer $n \ge 0$
  such that
  \begin{numequation}\label{eq-KRT}
\begin{split}
  \qSp_{K(n)} \subseteq \qSp_R \subseteq \qSp_{T(n)} \, .
\end{split}
\end{numequation}
\end{theorem}

As noted in \cite[paragraph after Theorem B and proof of Theorem
5.4.7]{CSY22}, the Nilpotence Theorem of \cite[Theorem 1]{DHS} (also
treated in \cite[]{Rav:NP}) implies that $\qSp_R$ is $1$-semiadditive
if and only if $R \otimes H\mathbb{F}_p = 0$ and there is exactly one
integer $n \geq 0$ for which $R \otimes K(n) \neq 0$. Namely, $\qSp_R$
is $1$-semiadditive if and only if $R$ is supported at a unique
(finite) chromatic height.  In \cref{eq-Rfn} we will see an example of
such an $R$, namely a certain infinite Galois extension of
$\mS_{T(n)}$.

This means that $\qSp_{K (n)}$ and $\qSp_{T (n)}$ are respectively
minimal and maximal among 1-semiadditive {\qcats} of $R$-local
spectra.  {\em Now that we know these two are distinct for $n\geq 2$,
  we would like to know the structure of the poset of such {\qcats}
  $\qSp_{R}$.}  

\begin{theorem}\label{thm-CSY22-thm-C}
\cite[Theorem C]{CSY22} Let $R \in \qSp$ be a homotopy ring
spectrum. The $\infty$-category $\qSp_R$ of $R$-local spectra is
$1$-semiadditive if and only if it is $\infty$-semiadditive.
\end{theorem}

There is no requirement above that $R$ be local at any prime. When it
is, we can say much more.

\begin{theorem}\label{thm-CSY22-thm-D}
\cite[Theorem D]{CSY22} 

Let $R$ be a non-zero $p$-local homotopy ring spectrum. The following
are equivalent:
\begin{enumerate}[label={(\roman*)},itemindent=0em]
    \item $R \otimes H\mathbb{F}_p = \pt$ and there is exactly one
integer $n \geq 0$ for which $R \otimes K(n) \neq \pt$.

    \item There exists a (necessarily unique) integer $n \geq 0$ such
that \cref{eq-KRT} holds.

    \item Either $\qSp_R = \qSp_{H\mathbb{Q}}$, or the functor
$\Omega^{\infty } \colon \qSp_R \to \qmcS_{*}$ admits a retract,
meaning a Bousfield-Kuhn functor as in \cite{AKB:Unique} and
\cite{Kuhn-Omega}.

    \item $\qSp_R$ is $1$-semiadditive.

    \item $\qSp_R$ is $\infty$-semiadditive.
\end{enumerate}
\end{theorem}

\begin{theorem}\label{thm-CSY22E}
\cite[Theorem E]{CSY22} Let $R$ be a non-zero $p$-local homotopy ring
spectrum, and let $n\geq 0$. Then the following are equivalent:
\begin{enumerate}[label={(\roman*)},itemindent=0em]

\item \label{thm-CSY22Ei} $R \otimes K (m)\simeq \pt$ for all $m>n$.

\item \label{thm-CSY22Eii} $R\otimes F(n+1) \simeq \pt$ for a finite
$p$-local spectrum $F(n+1)$ of type $n+1$ as in
\cref{def-MR99}\cref{def-MR99ii}.

\item \label{thm-CSY22Eiii} $R\otimes \Sigma^{\infty }A \simeq \pt$ for
every $n$-connected $\pi $-finite space $A$.
\end{enumerate}
\end{theorem}

Not all nonzero $p$-local homotopy ring spectra have the properties of
\cref{thm-CSY22E}, $H\mathbb{F}_p$ itself being an obvious exception.
Such a spectrum with vanishing mod $p$ and rational homology cannot be
connective.

Returning to $\pi $-finite spaces, each one has a finite Postnikov
tower. This means we need to understand (among others) the cases
$X=BG$ for a finite group $G$, and
\begin{displaymath}
X=B^{m}\rC_{p} = K (\Z/p,m).
\end{displaymath}

\noindent We know that $K (n)_{*}BG$ has finite rank \cite{Rav:KNBG},
and a formula for its Euler characteristic is given in \cite{HKR}.  It
is concentrated in even dimensions when $G$ is abelian, but there are
counterexamples for nonabelian $G$ due to Igor Kriz and Kevin Lee
\cite{KrizLee}.

The Morava $K$-theory of each {\SESM} space with finite abelian
homotopy group is computed in \cite{RW:CF}; see \cref{sec-SESM-surprise}.
In each case it is concentrated in even dimensions for $n>0$.  We also
know by \cite{HRW} that the Morava $K$-theory of a finite Postnikov
tower is that of the corresponding product of {\SESM} spaces.  In
other words, the $k$-invariants of such a tower are invisible to
Morava $K$-theory.

In any case, Hopkins and Lurie have formal methods to reduce showing
their norm map is an equivalence in the case where $X=B^{m}\rC_{p}$
and the functor $\rho $ of \cref{def-semiadd} is constant with value
$K (n)$.  This amounts to showing that the map induces an isomorphism
$K (n)_{i}X\to K (n)^{-i}X$ for each $i$.

\begin{remark}\label{rem-height}
{\bf Three notions of height.}  A finite $p$-local spectrum $X$ of
type $n$ has $K (m)_{*}X=0$ iff $m<n$ since it is known that the
nontriviality of $K (m)_{*}X$ implies that of $K (m+1)_{*}X$ for any
finite $X$.  The situation for the ring spectra of \cref{thm-CSY22E}
is in a sense the opposite: they are $K (m)$-acyclic for all $m>n$
instead of for $m<n$.  The authors of \cite{CSY22} observe

\begin{quote}
We obtain an equivalence of three different notions of {\bf height $\leq
n$} for a homotopy ring:
\begin{itemize}
\item []
\begin{enumerate}[label={(\roman*)},itemindent=0em]
\item\label{rem-heighti}
the {\bf algebraic} one using Morava $K$-theories,
\item\label{rem-heightii}  the {\bf geometric} one using finite complexes, and
\item\label{rem-heightiii} the {\bf categorical} one using
$\pi$-finite spaces.
\end{enumerate}
\end{itemize}

\noindent The categorical height of a spectrum (i.e. the minimal $d$
for which condition \cref{thm-CSY22E}\cref{thm-CSY22Eiii} holds) was
considered, using different terminology, by Bousfield in
\cite{AKB:HEHLS}. The most prominent example of such $R$ is $K (n)$,
which by \cite{RW:CF}, has categorical height $n$. Bousfield's work
also implies that for all $n\geq 0$, the spectrum $T (n)$ has some
finite categorical height, but determining its precise value had been
an open question. This can be now settled using \cref{thm-CSY22E}; as
the algebraic and geometric heights of $T (n)$ are known to be equal
to $n$, the categorical height must be $n$ as well.
\end{quote}

The blueshift of \cref{thm-blueshift} is algebraic in the sense of
\cref{rem-heighti} above.  One could also speak of geometric and
categorical blueshift.  The former is the subject of \cite[Conjecture
12]{MR}.
\end{remark}


\section{Chromatic cyclotomic extensions}\label{sec-CCE}

A theorem of Ethan Devinatz and Hopkins \cite[Theorem 1]{DH:HtyFixed}
says that $\mS_{K (n)}:=L_{K (n)}\mS$, the localization of
the sphere spectrum at height $n$ Morava $K$-theory, is the homotopy
fixed point set $E_{n} (\barFFp )^{h\GG_{n}}$.  Here
$E_{n} (\barFFp )$ is the Lubin-Tate spectrum constructed by Paul
Goerss, Hopkins and Haynes Miller, (see \cite{Goerss-Hopkins} and
\cite{Rezk}) and $\GG_{n}$ is the extended height $n$ Morava
stabilizer group of \cite{Morava}.

For future reference recall that 
\begin{numequation}\label{eq-coefficients}
\begin{split}
\pi_{*}BP 
 & = \Zloc [v_{1}, v_{2},\dotsc ]\qquad \mbox{with }|v_{i}|= 2 (p^{i}-1) \\
\pi_{*}\BPn
 & = \Zloc [v_{1}, v_{2},\dotsc, v_{n} ]\\
\pi_{*}E (n)
 & = v_{n}^{-1}\pi_{*}\BPn\\
\pi_{*}E_{n}
 & = W (\FF_{p^{n}})\pow{u_{1},\dotsc ,u_{n-1}}[u^{\pm 1}]\\
  &\qquad \mbox{with $|u|=2$, $|u_{i}|=0$, $v_{n} =u^{p^{n}-1}$ 
and $v_{i}=u^{p^{i}-1}u_{i}$}; \\
 & \qquad \mbox{see \cref{rem-periodicity}.}\\
\pi_{*}E_{n}(\barFFp )
 & = W (\barFFp )\pow{u_{1},\dotsc ,u_{n-1}}[u^{\pm 1}]
\end{split}
\end{numequation}%

Using ideas of John Rognes \cite{Rognes:Galois}, and Andy Baker and
Birgit Richter \cite{Baker-Richter}, the Morava spectrum
$E_{n} (\barFFp )$ can be viewed as a Galois extension of
$\mS_{K (n)}$ with Galois group $\GG_{n}$, the subject of
\cref{prop-Rognes-5.4.9}.  This means that for each normal subgroup
$N\subseteq \GG_{n}$ there is a Galois extension of $L_{K (n)}\mS$
with Galois group $G=\GG_{n}/N$.  We are interested in the cases where
$G$ is abelian.

We will begin with a review of abelian extensions of the $p$-adic
numbers in \cref{sec-classical}.  In \cref{sec-Rognes} we will review
Rognes' Galois theory for commutative ring spectra, the subject of
\cite{Rognes:Galois}.  In \cref{sec-height-n} we will discuss {\em
  higher} cyclotomic extensions.  \cref{sec-SESM-surprise} and
\cref{sec-cyclo-redshift} will cover the role of {\SESM} spaces and
cyclotomic redshift respectively.

\subsection{Some classical number theory}\label{sec-classical} 
Recall the classical description of the abelian extensions of the
$p$-adic numbers $\rats_{p}$.  A concise summary of the relevant
material on local fields (meaning finite extensions of the $p$-adic
numbers $\rats_{p}$) can be found in Xavier Caruso's
\cite[\S1.1]{Caruso19}.  The definitive introduction is Jean-Pierre
Serre's \cite{Serre-Local}.

\subsubsection{Ramified cyclotomic extensions}\label{sec-ramified-cyc-ext}
We know that every finite abelian extension is contained in some
cyclotomic extension, that is the extension obtained by adjoining a
primitive $m$th root of unity for some integer $m>2$.  The most
interesting case for us is $m=p^{j}$ for some $j>0$, and we denote
such a primitive root of unity by $\omega_{p^{j}}$.  This extension
has degree $(p-1)p^{j-1}$ and Galois group $(\Z/p^{j})^{\times }$.
Consider the diagram
\begin{numequation}\label{eq-cyclotomic-tower}
\begin{split}
\xymatrix
@R=1mm
@C=5mm
{
  & & &{\omega_{p^{j}}}\ar@{|->}[r]^(.5){}
        &{\omega_{p^{j+1}}^{p}}&{}
          &{}\\
{\rats_{p}}\ar[r]^(.5){}
  &{\rats_{p}[\omega_{p}]}\ar[r]^(.5){}
    &{\dotsb }\ar[r]^(.5){}
      &{\rats_{p}[\omega_{p^{j}}]}\ar[r]^(.5){}
        &{\rats_{p}[\omega_{p^{j+1}}]}\ar[r]^(.5){}
          &{\dotsb }\ar[r]^(.5){}
            &{\rats_{p}[\omega_{p^{\infty }}]}\\
{}\\
{\Z_{p}}\ar[r]^(.5){}\ar[uu]^(.5){}
  &{\Z_{p}[\omega_{p}]}\ar[r]^(.5){}\ar[uu]^(.5){}
    &{\dotsb }\ar[r]^(.5){}
      &{\Z_{p}[\omega_{p^{j}}]}\ar[r]^(.5){}\ar[uu]^(.5){}
        &{\Z_{p}[\omega_{p^{j+1}}]}\ar[r]^(.5){}\ar[uu]^(.5){}
          &{\dotsb }\ar[r]^(.5){}
            &{Z_{p}[\omega_{p^{\infty }}]}\ar[uu]^(.5){}\\
{}\\
{\Z_{p}}\ar[r]^(.5){}\ar[uu]^(.5){}
  &{\Z_{p}[\rC_{p}]}\ar[r]^(.5){}\ar[uu]^(.5){}
    &{\dotsb }\ar[r]^(.5){}
      &{\Z_{p}[\rC_{p^{j}}]}\ar[r]^(.5){}\ar[uu]^(.5){}
        &{\Z_{p}[\rC_{p^{j+1}}]}\ar[r]^(.5){}\ar[uu]^(.5){}
          &{\dotsb }\ar[r]^(.5){}
            &{Z_{p}[\rC_{p^{\infty }}]}\ar[uu]^(.5){}
}
\end{split}
\end{numequation}%

\noindent in which the top row is the {\bf ramified cyclotomic tower}
of field extensions, the second row shows their rings of integers, and
the third is the evident diagram of $p$-adic group rings with a
generator of the cyclic group $\rC_{p^{j}}$ of order $p^{j}$ mapping
to $\omega_{p^{j}}$.  We denote the colimits of the three rows by
$\rats_{p}[\omega_{p^{\infty }}]$, $\Z_{p}[\omega_{p^{\infty }}]$ and
$\Z_{p}[\rC_{p^{\infty }}]$, where $\rC_{p^{\infty }}\subseteq \mT$ is
the Pr\"uffer group of \cite[Definition 5.23 (ii)]{Rav:gjmcyc}.  The
Galois group of $\rats_{p}[\omega_{p^{\infty }}]$ over $\rats_{p}$ is
$\Z_{p}^{\times }$, the multiplicative group of $p$-adic units.

There is a split short exact sequence
\begin{numequation}\label{eq-SES}
\begin{split}
\xymatrix
@R=4mm
@C=10mm
{
{1}\ar[r]^(.5){}
  &{\Z_{p}}\ar[r]^(.5){e}
    &{\Z_{p}^{\times }}\ar[r]^(.5){}
      &{T_{p}}\ar[r]^(.5){}
        &{1,}
}
\end{split}
\end{numequation}
 
\noindent where $T_{p}\subset \Z_{p}^{\times } $ is the torsion
subgroup, which consists of roots of unity, namely
$\mu_{2}\cong \left\{\pm 1 \right\} $ for $p=2$, and
$\mu_{p-1}\cong \rC_{p-1} $, the group of $(p-1)$th roots of unity,
for $p$ odd.  The map $e$ is the $p$-adically convergent exponential
function given by
\begin{displaymath}
e(x)=\mycases{    
  \exp(4x)=
\displaystyle{\sum _{k\geq 0}\frac{(4x)^{k}}{k!}=1+4x+8x^{2}+\cdots }
  &\mbox{for }p=2\\
  \exp(px)=
  \displaystyle{\sum _{k\geq 0}\frac{(px)^{k}}{k!}=1+px+
    \frac{p^{2}x^{2}}{2}+\cdots }
  &\mbox{for }p>2.
}
\end{displaymath}
 
\noindent Note that the power series $\exp(2x)$ does not converge
2-adically.

For future reference we note that $\Z_{p}^{\times }$ has
a dense finitely generated subgroup
\begin{numequation}\label{eq-dense-discrete}
\begin{split}
  \Z_{p}^{\times, \fg}
  &:= \mycases{
\langle e(1) ,\,-1 \rangle
  &\mbox{for }p=2\\
\langle e(1),\, \zeta_{p-1} \rangle
       &\mbox{for }p>2
         }\\
  &\qquad \mbox{for a primitive $(p-1)$th root of unity $\zeta_{p-1}$}
  \\ 
  &\phantom{:} \cong\Z _{p}\times T_p.
\end{split}
\end{numequation}

\noindent 
We will see this group in \cref{prop-BCSY24-6.19}.

Let 
\begin{displaymath}
\varpi_{j}:= \omega_{p^{j}}-1 \in \Z_{p}[\omega_{p^{j}}].
\end{displaymath}

\noindent Then we find that $\varpi _{j}^{(p-1)p^{j-1}}$ is a unit multiple
of $p$, so
\begin{displaymath}
(p)= (\varpi_{j})^{(p-1)p^{j-1}}.
\end{displaymath}

\noindent The maximal ideal $(p)\subseteq \Z_{p}$ becomes a power of
the maximal ideal $(\varpi_{j})$ when we pass to
$\Z_{p}[\omega_{p^{j}}]$. This is called {\bf ramification}. The
$p$-adic integers $\Z_{p}$ and the $p$-adic numbers $\rats_{p}$ have
discrete valuations $||-||$ in which\linebreak
$||p^{i}||=i$.  They extend to valuations on $\Z_{p}[\omega_{p^{j}}]$
and $\rats_{p}[\omega_{p^{j}}]$ in which
$||\varpi_{j}||=1/ (p-1)p^{j-1}$.

The lower vertical maps in \cref{eq-cyclotomic-tower} are induced by
certain group monomorphisms
$\rC_{p^{j}}\to \Z_{p}[\omega_{p^{j}}]^{\times }$.  Such a map, along
with a choice of generator of $\rC_{p^{j}}$, determines a $p^{j}$th
root of unity in $\Z_{p}[\omega_{p^{j}}]$.  These maps are known to
split after inverting $p$, as explained in \cite[page~3514]{CSY24}.  For
any ring $R$ in which $p$ is invertible, $R[\rC_{p^{j}}]$ is
isomorphic as a ring to $R[\rC_{p^{j-1}}]\times R[\omega_{p^{j}}]$.

\subsubsection{Unramified extensions}\label{sec-unram}
For each $d>1$, the $p$-adic numbers also has an
{\bf unramified} (meaning that the ideal $(p)$ remains maximal)
abelian extension whose ring of integers is the Witt ring $\mW
(\FF_{p^{d}})$. Its Galois group is $\rC_{d}$. It is also the
cyclotomic extension $\rats_{p}[\omega_{p^{d}-1}]$ obtained by
adjoining $(p^{d}-1)$th roots of unity.  The degree of the
corresponding extension of $\rats $ is the Euler totient $\varphi
(p^{d}-1)$, which larger than $d$.  The cyclotomic polynomial
$\Phi_{p^{d}-1} (x)$ is irreducible over $\rats $ but has an
irreducible factor of degree $d$ over $\rats_{p}$.  For example when
$(p,d)= (3,2)$, we have
\begin{displaymath}
\Phi_{8} (x)=x^{4}+1
 =  \left(x^2-\sqrt{-2} x-1\right) \left(x^2+ \sqrt{-2} x-1\right),
\end{displaymath}

\noindent and the reader can verify that $\sqrt{-2}$ can be regarded a
3-adic integer, as can the square root of any integer congruent to 1
mod 3.

The algebraic closure $\barFFp$, which is the colimit of the finite
fields $\FF_{p^{d}}$, has Galois group $\widehat{\Z}$, the profinite
integers.  We know that the Galois group of the maximal abelian
extension of $\rats_{p}$ is $\Z_{p}^{\times }\times \widehat{\Z}$.

\subsection{Rognes' Galois theory for commutative ring spectra}
\label{sec-Rognes}

Rognes' definitions are stated in terms of spectra localized at a
homology theory $E$ such as Morava $K$-theory.  His spectra are
$\mS$-algebras in the sense of \cite{EKMM} but we are treating them in
terms of orthogonal spectra.  We are not aware of a published
definition in the language of {\qcats}, but one is hinted at
by Akhil Mathew, Niko Naumann and  Justin Noel in \cite{MNN17}.

\begin{defin}\label{def-Rognes-Galois}
{\bf Galois extensions of  $E$-local commutative ring spectra.}
\cite[Definition 4.1.3]{Rognes:Galois} 

\begin{itemize}
\item [$\bullet$] Let $A\to B$ be a map of $E$-local commutative ring
spectra, making $B$ a commutative $A$-algebra, and let $G$ be an
$E$-locally stably dualizable group (for example a finite group)
acting continuously on $B$ from the left through commutative
$A$-algebra maps.

\item [$\bullet$] 
Let $i \colon A \to B^{hG}$
be the map to the homotopy fixed point spectrum\linebreak  $B^{hG} =
F(EG_+, B)^G$ that is right adjoint to the composite
map
\begin{displaymath}
A \wedge EG_+ \to A \to B,
\end{displaymath}

\noindent that collapses the contractible free $G$-space $EG$ to a
point.

\item [$\bullet$] Let $h \colon B \wedge_A B \to F(G_+, B)$ be the
canonical map to the product (cotensor) $F(G_+, B)$ that is right
adjoint to the composite map
\[
B \wedge_A B \wedge G_+ \to B \wedge_A B \to B,
\]
induced by the action $B \wedge G_+ = G_+ \wedge B \to B$ of $G$ on $B$,
followed by the $A$-algebra multiplication $B \wedge_A B \to B$ in $B$.
\end{itemize}

A map of $E$-local commutative ring spectra $A \to B$ is an {\bf
$E$-local $G$-Galois extension} if $i$ and $h$ are both $E$-local
equivalences.  Such an extension is {\bf finite}, {\bf abelian} or
{\bf finite abelian} if $G$ is.
\end{defin}

\begin{defin}\label{def-Rognes-3defs}
{\bf Properties of Galois extensions.}

\begin{enumerate}[label={(\roman*)},itemindent=0em]

\item \label{def-Rognes=3defsi} \cite[Definition 4.3.1]{Rognes:Galois}
A map of commutative ring spectra $A\to B$ is {\bf faithful} if for
every $A$-module $N$, the condition $N \wedge_A B \simeq *$ implies
that $N \simeq *$.

\item \label{def-Rognes=3defsii} \cite[Definition
9.1.1]{Rognes:Galois} The map $A \to B$ is {\bf separable} if the
multiplication map $\mu \colon B \wedge_A B \to B$ admits
a bimodule section up to homotopy.

\item \label{def-Rognes=3defsiii} \cite[Definition
10.2.1]{Rognes:Galois} A commutative ring spectrum $B$ is {\bf
connected}, in the sense of algebraic geometry, if its space of
idempotents $\mathcal{E}(B)$ is weakly equivalent to the two-point
space $\{0,1\}$.
\end{enumerate}
\end{defin}

Then his main theorem is 

\begin{theorem}\label{thm-Rognes-main}
{\bf The stable homotopy theoretic Galois correspondence.}  
\cite[Theorem 1.2]{Rognes:Galois} 

Let $A \to B$ be a faithful $E$-local $G$-Galois extension of
commutative ring spectra.
\begin{enumerate}[label={(\roman*)},itemindent=0em]
  \item\label{thm-Rognes-maini}
For each subgroup $H \subset G$, the map $C = B^{hH}
\to B$ is a faithful $E$-local $H$-Galois extension, and
the map $A \to C$ is separable.

  \item\label{thm-Rognes-mainii} For each normal subgroup $H \subset
G$, the map $A \to C = B^{hH}$ is a faithful $E$-local $G/H$-Galois
extension.
\end{enumerate}

\noindent If $B$ is connected, then
\begin{enumerate}[label={(\roman*)},itemindent=0em]
  \item[(iii)]\label{thm-Rognes-mainiii}
The Galois group
$G$ is weakly equivalent to the mapping space\linebreak  $\mathrm{CAlg}_A(B,
B)$ of commutative $A$-algebra self-maps of $B$.

  \item[(ov)]\label{thm-Rognes-mainiv}
For each factorization $A \to C
\to B $ of the $G$-Galois extension, with $A \to C$
separable and $C \to B$ faithful, there exists a subgroup $H \subset
G$ such that $C \simeq B^{hH}$ as an $A$-algebra over $B$.
\end{enumerate}
\end{theorem}

In other words, for a faithful $E$-local $G$-Galois extension $A \to
B$ with $B$ connected, there is a bijective contravariant Galois
correspondence\linebreak $H \longleftrightarrow C \simeq B^{hH}$
between the subgroups $H \subset G$ and the weak equivalence classes
of separable $A$-algebras $C$ mapping faithfully to $B$.

The inverse correspondence takes $C$ to $H = \pi_0 \mathrm{CAlg}_C(B, B)$,
the group of connected components of the mapping space of commutative
$C$-algebra self-maps of $B$.

\begin{prop}\label{prop-Rognes-5.4.9}
  {\bf A faithful profinite Galois extension.}
  \cite[Proposition 5.4.9]{Rognes:Galois} The map
  $\mS _{K(n)} \to E_{n}(\barFFp)$ is a $K(n)$\nobreakdash-local
  pro-Galois extension with Galois group $\mathbb{G}_n$, and both the
  extension and the associated $\GG _{n}$–action are faithful. We have
  a short exact sequence
  \begin{numequation}\label{eq-SGZ}
\begin{split}
 1 \longrightarrow \mS_{n}
  \longrightarrow \GG_{n}
  \longrightarrow \mathrm{Gal}(\barFFp/\mathbb{F}_p)
  \longrightarrow 1,
\end{split}
\end{numequation}
where $\mS_{n}$ has a pro-$p$ open subgroup of index $p^{n}-1$.
\end{prop}

Recall that $\mathrm{Gal}(\barFFp/\mathbb{F}_p)$ and $\mS_{n}$ are
isomorphic respectively to the profinite integers $\widehat{\Z}$ and
 the group of units $\mathcal{O}_D^{\times}$ in the ring of integers
$\mathcal{O}_D$ of the division algebra $D$ over $\mathbb{Q}_p$ with
Hasse invariant $1/n$.

\subsection{Higher cyclotomic extensions and cyclotomic
  completion}\label{sec-height-n}

We do not have an analog of \cref{prop-Rognes-5.4.9} for
$\mS _{T(n)}$. However by a theorem of Carmeli, Schlank and Yanovski
we do know that for each {\em abelian} Galois extension of
$\mS_{K (n)}$, there is an analogous extension of $\mS_{T (n)}$.

We also know that the abelianization of the absolute Galois group of
$\rats_{p}$, which controls the abelian extensions of that field
reviewed in \cref{sec-classical}, is isomorphic to
$\Z_{p}^{\times }\times \widehat{\Z}$.  The same is true of the
abelianization $\GG_{n}^{\rm ab}$ of the extended Morava stabilizer
group $\GG_{n}$, while the abelianization of $\mS_{n} $ is
$\Z_{p}^{\times }$.

In the former case the projections onto $\Z_{p}^{\times }$ and
$\widehat{\Z}$ correspond respectively to the maximal totally ramified
and unramified abelian extensions of $\rats_{p}$.

In the latter case this means that the abelian extensions of
$\mS_{K (n)}$ can also be thought of as ``chromatic cyclotomic
extensions,'' hence the title of \cite{CSY24}.  We will be interested
in the chromatic analogs of the ramified extensions of
\cref{eq-cyclotomic-tower} rather than the unramified ones of
\cref{sec-unram}.

\begin{theorem}\label{thm-lifting-abelian}
{\bf Lifting abelian extensions of $\mS_{K(n)}$ to $\mS_{T(n)}$.}
\cite[Theorems A and 5.31]{CSY24}
For every $K(n)$-local finite
abelian extension\linebreak  $\mS_{K(n)} \to R$ in \(
\mathrm{\qSp}_{K(n)} \), there exists a $T(n)$-local finite abelian
extension $\mS_{T(n)} \to R^{\fin}$ in \(
\mathrm{\qSp}_{T(n)} \) (having the same Galois group), such that
\begin{displaymath}
L_{K(n)} R^{\fin} \simeq R.
\end{displaymath}

\noindent When $G=(\Z/p^{j})^{\times }$ for $j>0$, we denote this extension
by $\SS_{T (n)}[\omega^{(n)}_{p^{j}}]$.  For
$X \in \qSp_{T(n)} $ we define
\begin{numequation}\label{eq-Xnj}
\begin{split}
X[\omega^{(n)}_{p^{j}}]:=X\otimes \SS_{T (n)}[\omega^{(n)}_{p^{j}}].
\end{split}
\end{numequation}
\end{theorem}

\begin{defin}\label{defin-two-infinite}
{\bf Two infinite abelian extensions of $\SS_{K(n)} $ and $\SS_{T(n)} $.}
Let
\begin{align*}
  R_{n}
  =\SS_{K (n)}[\omega^{(n)}_{p^{\infty}}]
 & :=\colim{j}\SS_{K (n)}[\omega^{(n)}_{p^{j}}],\\
\SS_{K(n)}^{\ab}   
 & := \mS \WW(\barFFp)\otimes \SS_{J (n)}[\omega^{(n)}_{p^{\infty}}],\\
  R^{\fin}_{n}
  =\SS_{T (n)}[\omega^{(n)}_{p^{\infty}}]
 & :=\colim{j}\SS_{T (n)}[\omega^{(n)}_{p^{j}}]\\
  \aand
\SS_{T(n)}^{\ab}   
 & := \mS \WW(\barFFp)\otimes \SS_{T (n)}[\omega^{(n)}_{p^{\infty}}],
\end{align*}

\noindent where $\mS \WW(\barFFp)$ denotes the evident Moore spectrum.
$\SS_{K(n)}^{\ab}$ and $\SS_{T(n)}^{\ab}$ are the {\bf maximal abelian
  extensions} of $\SS_{K (n)}$ and $\SS_{T (n)}$.
\end{defin}

\begin{remark}\label{rem-max-ab-faithful}
  $\mS_{K(n)}[\omega^{(n)}_{p^\infty}]$ and $\SS_{K(n)}^{\ab}$ are
  known to be a faithful Galois extensions of $\mS _{K(n)}$.
\end{remark}

\begin{ex}\label{ex-Westerland}
{\bf Westerland's spectrum.}
The spectrum $\mS_{K(n)}[\omega^{(n)}_{p^\infty}]$ for $n>0$ is the
ring $R_n$ studied by Craig Westerland in \cite[Theorem
1.2]{Westerland}. Namely, it is the (continuous) homotopy fixed points
of $E_n$ for the action  the kernel $S\GG_{n}^{\pm }$ of
\[
  {\det}  _{\pm } : \GG_{n} \longrightarrow \Z_p^\times,
\]

\noindent whence it is a $\Z_p^\times$–Galois extension of
$\mS_{K(n)}$.  This map is the extension of the determinant
homomorphism on the subgroup $\mS_{n}$ (as in \cref{eq-SGZ}) to
$\GG_{n}$ given by sending the Frobenius generator $F$ of
$\mathrm{Gal}(\barFFp/\mathbb{F}_p)$ to $(-1)^{n-1}$.  This means the
homomorphism is trivial on the subgroup
$\mathrm{Gal}(\barFFp/\mathbb{F}_{p^{n}})$ topologically generated by
$F^{n}$.

He describes $R_{n}$ as a close relative of
$L_{K(n)}\Sigma^{\infty}K(\Z_{p},n+1)_{+}$.  Note that the fiber
sequence
\begin{displaymath}
  K(\rats_{p},n )
  \to K(\rats_{p}/Z_{p},n )\to K(\Z_{p},n+1)
\end{displaymath}
 
\noindent leads to an equivalence
\begin{numequation}\label{eq-K(n)ZC}
\begin{split}
  L_{K(n)}\Sigma^{\infty}K(\rC _{p^{\infty},n})_{+}
       \to L_{K(n)}\Sigma^{\infty}K(\Z_{p},n+1)_{+}.
\end{split}
\end{numequation}

\noindent since  $K(\rats_{p},n )$ is $K(n)$-acyclic for $n>0$.

We will say more about {\SESM} spaces in \cref{sec-SESM-surprise}.
\end{ex}

For each finite $j$, the map 
\begin{displaymath}
\SS_{T (n)}\to 
\SS_{T (n)}[\omega^{(n)}_{p^{j}}]^{h(\Z/p^{j})^{\times}}
\end{displaymath}

\noindent is an equivalence, meaning that the corresponding Galois
extension is faithful.  This {\em does not} imply that the Galois extension
\begin{numequation}\label{eq-Rfn}
\begin{split}
\SS_{T (n)}\to 
\SS_{T (n)}[\omega^{(n)}_{p^{\infty }}]=:R^{\fin}_{n}
\end{split}
\end{numequation}

\noindent is faithful.

\begin{defin}\label{def-cyc-comp}
  {\bf Height $n$ cyclotomic completion.}
  We denote localization with respect to $R^{\fin}_{n}$ by
  $L_{n}^{\Cyc}$ and the corresponding category by $L^{\Cyc}_{n}\qSp$.
  A $T (n)$-local spectrum is {\bf cyclotomically complete} if it is
  $ R^{\fin}_{n}$\nobreakdash-local (see \cref{eq-Rfn}).

 The
analog of $L_{K(n)}$ is
\begin{displaymath}
L_{\Cyc(n)}:=L_{T(n)}L_{n}^{\Cyc},
\end{displaymath}

\noindent and we denote the corresponding category by $\qSp_{\Cyc(n)}$. 
\end{defin}

$\qSp_{\Cyc(n)}$ is denoted by $\widehat{\qSp}_{T (n)}$ in
\cite{BCSY24} and by $(\qSp_{T(n)})^{\wedge }_{{\rm cyc}}$ in
\cite[page~4]{BHLS}.

\begin{prop}\label{prop-BCSY24-6.19} {\bf The height $n$
    cyclotomically complete sphere.} \cite[Proposition~6.19]{BCSY24}
  The restriction of the functor $L^{\Cyc}_{n}$ to $\qSp_{T(n)}$ is
  smashing with the unit object 
  \begin{displaymath}
    (R^{\fin}_{n})^{h\Z_{p}^{\times, \fg}}
  \in \CAlg(\qSp_{\Cyc(n)})
  \qquad \mbox{for $\Z_{p}^{\times, \fg}$ as in \cref{eq-dense-discrete}} .
\end{displaymath}
\end{prop}

We do not know if the functor $L^{\Cyc}_{n}$ is smashing in the larger
category $\qSp$ like $L_{n}$ and $L^{\fin}_{n}$ are known to be.

\begin{remark}\label{remark-why-fg-subgroup} {\bf Why the finitely generated
    subgroup?}
  Note that the subgroup
  $\Z_{p}^{\times, \fg} \subseteq \Z^{\times }_{p}$ is not closed.
  Devinatz-Hopkins theory \cite[Theorem 1]{DH:HtyFixed} concerns
  closed subgroups of $\GG_{n}$ and the category of $K(n)$-local
  spectra, but we are now in the category of $T(n)$-local spectra.  In
  that world we need to replace $\Z^{\times }_{p}$ by a dense finitely
  generated subgroup in the description of the unit object.
\end{remark}

It is stated without explicit proof in \cite[page 94]{BCSY24} that
$K (n)$-local spectra are height $n$ cyclotomically complete.  By
\cref{thm-CSY22-thm-B} (originally \cite[Theorem B]{CSY22}) this is
the case because $K(m)_{*}R^{\fin}_{n}=0$ iff $m>n$ and
$R^{\fin}_{n}\otimes H\mathbb{F}_p = 0$.

Thus, using the notation of \cref{def-Ln-fin,def-cyc-comp}, we have
categorical inclusions and natural transformations of restricted
functors (see \cref{eq-AKB-classes})
\begin{numequation}\label{eq-K-cyc-T}
  \begin{split}
    \xymatrix
@R=1mm
@C=8mm    
{
{L_{n}}    
  &{L_{n}^{\Cyc}}\ar@{=>}[l]^(.5){}
    &{L_{n}^{\fin}}\ar@{=>}[l]^(.5){}\\
{L_{n}\qSp}\ar@{^{(}->}[r]^(.5){}
  &{L^{\Cyc}_{n}\qSp}\ar@{^{(}->}[r]^(.5){}
    &{L^{\fin}_{n}\qSp}\\
    {}\\
{\qSp_{K (n)}}\ar@{^{(}->}[r]^(.5){}\ar@{^{(}->}[uu]^(.5){}
  &{\qSp_{\Cyc(n)}}\ar@{^{(}->}[r]^(.5){}\ar@{^{(}->}[uu]^(.5){}
    &{\qSp_{T (n)}}\ar@{^{(}->}[uu]^(.5){}\\
{L_{K(n)}}   
  &{L_{\Cyc(n)}}\ar@{=>}[l]^(.5){}
    &{L_{T(n)}.}\ar@{=>}[l]^(.5){}
      }
  \end{split}
\end{numequation}%

\noindent For $n=1$, the three categories in the second row are the
same (as are the ones in the third row) since the {\TC} is known to be
true for $n=0$ and $n=1$.  For $n\geq 2$, $T (n)$-local spectra are
now known not to be cyclotomically complete by \cite[Theorem A]{BHLS},
quoted here as \cref{thm-main-A}.

\begin{thm}[{\cite[Proposition~4.17]{BMCSY23} and
    \cite[Theorem 6.18]{BHLS}}]\label{thm:6.18}
  For $n \ge 0$, let $R \in \CAlg(\qSp_{T(n)})$.  The map
\[
  \rK_{T(n+1)} (R)
    \;\longrightarrow\;
    \rK_{T(n+1)} \!\left( \mS_{T(n)}[\omega^{(n)}_{p^\infty}]^{hT_p}
       \otimes R \right)^{h\Z},
\]

\noindent for $T_{p}, \Z\subseteq \Z_{p}^{\times }$ as in
\cref{eq-dense-discrete}, exhibits the target as the height $n+1$
cyclotomic completion (as in \cref{def-cyc-comp}) of the source, where
the tensor products are taken in $\CAlg(\qSp_{T(n)})$.
\end{thm}

\begin{corollary}\label{cor-6.19}
  \cite[Corollary 6.19]{BHLS} 
For $n \ge 0$, the map
\[
  \mS_{T(n)} \longrightarrow \mS_{T(n)}[\omega^{(n)}_{p^\infty}]
\]
is a $\rK_{\Cyc(n+1)}$-cover (for $\rK_{\Cyc(n+1)}$ as in
\cref{def-KT-KK}) in the sense of \cref{def-F-descent}.
\end{corollary}

\subsection{The surprising appearance of {\SESM} spaces}
\label{sec-SESM-surprise}

\begin{defin}\label{defin-space-indexed}
{\bf Space indexed products and coproducts.}
\cite[\S1.4(3)]{CSY21} Given an {\qcat} $\qmcC$ and a space $A$, the
map $A\to \pt$ induces a functor
$A^{*}:\qmcC \to \qmcC^{A}:=\qFun (\qA, \qmcC)$ sending an object $X$
in $\qmcC$ to the constant $X$-valued functor on $\qA$, the {\qcat}
associated with $A$ as in
\cite[Definition 5.1(ii)]{Rav:gjmcyc}.
For suitable $\qmcC$, $A^{*}$ has left and right adjoints $A_{!}$ and
$A_{*}$, the colimit and limit of the functor $\qA\to \qmcC$.  We
define
\begin{numequation}\label{eq-X[A]}
\begin{split}
X[A]:=A_{!}A^{*}X
\qquad \aand 
X^{A}:=A_{*}A^{*}X,
\end{split}
\end{numequation}

\noindent which are respectively covariant and contravariant in $A$.
They come equipped with a counit or fold map $\nabla:X[A]\to X$ and a
unit or diagonal map $\Delta :X\to X^{A}$.
\end{defin}

\begin{ex}\label{ex-finite-coprod-prod}
  When $A$ is a finite set with $k$ elements, then $X[A]$ and $X^{A}$
  are the $k$-fold coproduct and product of $X$.  When $A=BG$ for a
  topological group $G$, then $X[A]$ and $X^{A}$ are the homotopy orbit
  space $X_{hG}$ and the homotopy fixed point set $X^{hG}$.
\end{ex}

It turns out that in the $K (n)$-local and $T (n)$-local worlds for
$n>0$, the height $n$ analogs of $\Z_{p}[\rC_{p^{j}}]$ in
\cref{eq-cyclotomic-tower} are the spectra
\begin{align*}
\SS_{K (n)} [B^{n}\rC_{p^{j}}]
 & = L_{K (n)}\Sigma^{\infty}B^{n}\rC_{p^{j}}\\
\aand
 \SS_{T (n)} [B^{n}\rC_{p^{j}}]
 & = L_{T (n)}\Sigma^{\infty}B^{n}\rC_{p^{j}}.
\end{align*}
The analogs of $\Z_{p}[\omega_{p^{j}}]$ are summands of these.

Here the space $B^{n}\rC_{p^{j}}$ is the $n$th iterated classifying
  space of the cyclic group of order $p^{j}$, more traditionally known
  as the {\SESM} space $K (\Z/p^{j}, n)$, first introduced in
  \cite{SESM2}.  We saw the colimit as $j\to \infty$ in
  \cref{ex-Westerland}.

To see why these spaces arise, suppose we replace the ring
$\Z_{p}[\omega_{p^{\infty }}]$ in \cref{eq-cyclotomic-tower} by a
commutative ring $R$ in a symmetric monoidal {\qcat} $\qmcC$, for
example an $\EE_{\infty }$-ring spectrum. Its multiplicative group of
units $R^{\times }$ can have nontrivial iterated loop spaces, unlike the
totally disconnected space $\Z_{p}[\omega_{p^{\infty }}]^{\times }$.

\begin{defin}\label{def-height-n-root}
  \cite[Definition 4.2]{CSY24}
A {\bf $p^{j}$th root of unity of height $n$},
$\omega^{(n)}_{p^{j}}\in R$, is the image of a generator of
$\rC_{p^{j}}$ under a group homomorphism $\rC_{p^{j}}\to \Omega
^{n}R^{\times }$.
\end{defin}

Such a map is adjoint to a map $B^{n}\rC_{p^{j}}\to R^{\times }$.
{\em This notion of height differs from that in the theory of formal
  group laws.}  See \cref{rem-height}.

Higher roots of unity are discussed in \cite[\S4.1]{CSY24}.  They are
also used in \cite{BCSY24} to construct higher height analogues of the
discrete Fourier transform and Kummer theory. In \cite[Theorems~4.26
and 4.28]{BMCSY23} respectively the authors show that these constructions
play nicely with the functor \( \rK_{T(n+1)} \).

The Morava $K$-theories of the spaces $B^{n}\rC_{p^{j}}$, for all
heights (in the formal group law sense) and all values of $n$ and $j$,
were computed long ago by Steve Wilson and the author in \cite{RW:CF}
for odd primes, and by Johnson and Wilson in \cite[Appendix]{JW:EPG}
for $p=2$.  We also did it for the $K(n)$-equivalent spaces
$B^{n+1}\Z$ and $B^{n}\rC_{p^{\infty}}$; see \cref{eq-K(n)ZC}.  The
alternative computation of Hopkins and Jacob Lurie in \cite[\S2]{HL},
which works for all primes, is explained by Sanath Devalapurkar in
\cite{Devalapurkar18}.  An interesting arithmetic interpretation of
our result is given by Victor Buchstaber and Andrey Lazarev in
\cite{Buchstaber-Lazerev}.

In particular we know that $K (n)_{*}B^{n}\rC_{p^{j}}$ has rank
$p^{j}$ (as a module over $K (n)_{*}$) and is concentrated in
dimensions divisible by
\begin{displaymath}
|v_{n}|/ (p-1)=2 (1+p+\dotsb +p^{n-1}).
\end{displaymath}

\noindent This means that in the $K (n)$-local world,
$\Sigma^{\infty }_{+}B^{n}\rC_{p^{j}}$ ``looks like'' a CW-complex
with $p^{j}$ cells evenly distributed in such dimensions, with an
extra one in dimension 0.  This generalizes that fact (the case $n=0$)
that $\Sigma^{\infty }_{+}\rC_{p^{j}}$ (and hence its rationalization)
is a wedge of $p^{j}$ copies of the (rationalized) sphere spectrum.

{\em When we wrote \cite{RW:CF}, we had no idea that these spaces
would play such a role in chromatic homotopy theory!}

\begin{remark}\label{rem-periodicity}
{\bf Periodicity in Morava $K$-theory.}  In the above paragraph we are
using the original definition of $K (n)$, for which
\begin{displaymath}
\pi_{*}K (n)=\FFp [v_{n}^{\pm 1}]
\qquad \mbox{with }|v_{n}|=2 (p^{n}-1). 
\end{displaymath}

\noindent This spectrum is $|v_{n}|$-periodic.  One can formally
adjoin a $(p^{n}-1)$th root $u$ of $v_{n}$ (with $|u|=2$) to obtain a
2-periodic spectrum $\widehat{K} (n)$. It follows that 
\begin{align*}
\widehat{K} (n)_{0} (X)
 & = \bigoplus_{0\leq i<p^{n}-1}u^{-i}K (n)_{2i}X  \\
\aand 
\widehat{K} (n)_{1} (X)
 & = \bigoplus_{0\leq i<p^{n}-1}u^{-i}K (n)_{2i+1}X. 
\end{align*}

\noindent The functors $L_{K (n)}$ and $L_{\widehat{K} (n)}$ are the
same, as are the {\qcats} $\qSp_{K (n)}$ and $\qSp_{\widehat{K} (n)}$.

In the rest of the paper $K (n)$ will abusively denote the 2-periodic
spectrum $\widehat{K} (n)$.
\end{remark}

The equality of the ranks of $T (n)_{*}B^{m}\rC_{p^{j}}$ and $K
(n)_{*}B^{m}\rC_{p^{j}}$ is the subject of \cite[Lemma 5.1.7]{CSY22}.

\subsection{Chromatic cyclotomic redshift}\label{sec-cyclo-redshift}

Here we recall some results of the insightful work of Ben-Moshe,
Carmeli, Schlank and Yanovski, \cite{BMCSY23}.

We first look at {\qsheaves} (\cref{defin-sheaves}) on the
categories $\Fin _{\Z_{p}}$ and $\Fin _{\Z^{\times }_{p}}$ (see
\cref{defin-TG}) of finite sets acted on by the groups of $p$-adic
integers and $p$-adic units.  The general theory of such sheaves for
profinite groups $G$ is reviewed in \cref{sec-finite-profinite}.

\begin{defin}\label{def-cyclo-sheaf}\cite[Definition 5.10]{BMCSY23}
  For $X \in \qSp_{T(n)} $, the {\bf height $n$ cyclotomic sheaf} is
  the contravariant functor on $\Fin _{\Z^{\times }_{p}}$ given by
  \begin{displaymath}
    (\Z/p^{j})^{\times }\mapsto X[\omega^{(n)}_{p^{j}}]
\end{displaymath}
 
\noindent as in
  \cref{eq-Xnj}.  Its stalk (see \cref{eq-stalk}) is
  $X[\omega^{(n)}_{p^{\infty}}]$.
\end{defin}

\begin{prop}\label{prop-BMCSY-5.11}
  \cite[Proposition 5.11]{BMCSY23}
The cyclotomic sheaf $\mS_T^{(n)}[\omega_{p^{(-)}}^{(n)}]$ is a continuous
$\Z_p^{\times}$--Galois extension, and there is a symmetric monoidal
equivalence
\[
  \qSp_{T(n)}
    \;\xrightarrow{\;\sim\;}\;
  \Mod_{\mS_T^{(n)}[\omega_{p^{(-)}}^{(n)}]}
    \bigl(\Shv(\Fin_{\Z _{p}};\,\qSp_{T(n)})\bigr),
  \qquad
  X \longmapsto X[\omega_{p^{(-)}}^{(n)}].
\]

\noindent Moreover, $X \in \qSp_{T(n)}$ is cyclotomically complete
if and only if the height $n$ cyclotomic sheaf $X[\omega_{p^{(-)}}^{(n)}]$ is
a hypersheaf as in \cref{defin-hypersheaf}.
\end{prop}

\begin{defin}\label{defin-Lnf-local} \cite[Definition 2.10]{BMCSY23} 
   An {\bf  $L_{n}^{\fin}$-local {\qcat}}
is an idempotent complete stable {\qcat} in which all mapping
spectra are $L_{n}^{\fin}$-local as in \cref{def-Ln-fin}.
\end{defin}

The first theorem involves such an {\qcat} acted on by a $\pi $-finite
$p$-group $G$ as in \cref{def-pi-finite}; an ordinary finite $p$-group
is a special case.  This means there is a functor ${\qBG}\to {\qCatw}$
(the latter {\qcat} is the subject of \cite[Chapter 3]{Lurie:HTT})
sending the single object of the domain to an {\qcat} $\qmcC$, for
which the limit is ${\qmcC}^{hG}$ and the colimit is ${\qmcC}_{hG}$.

\begin{theorem}\label{thm-BMCSY-ThmA}
  {\bf Higher descent.} \cite[Theorem A]{BMCSY23}
  Let $\qmcC$ be a $L_{n}^{\fin}$-local {\qcat} (and hence a
  simplicial set) acted on by a $\pi $-finite $p$-group $G$ as in
  \cref{def-pi-finite}, and let $\rK_{T (n+1)}$ be the functor of
  \cref{def-KT-KK}.  Then the coassembly map
\begin{displaymath}
\epsilon :\rK_{T (n+1)} ({\qmcC}^{hG})\to \rK_{T (n+1)} ({\qmcC})^{hG}
\end{displaymath}

\noindent and the assembly map
\begin{displaymath}
\eta :\rK_{T (n+1)} ({\qmcC})_{hG}\to \rK_{T (n+1)} ({\qmcC}_{hG})
\end{displaymath}

\noindent of \cite[Definition 5.6]{Rav:gjmcyc} are
isomorphisms.
\end{theorem}

\begin{cor}\label{cor-}\cite[Corollary
  1.4]{BMCSY23}
Let $R \to S$ be a $T(n)$-local $G$-Galois extension
where $G$ is an $n$-finite $p$-group. Then
\[
\rK_{T(n+1)} (R) \longrightarrow \rK_{T(n+1)} (S)
\]
is a $T(n+1)$-local $G$-Galois extension.
\end{cor}

A similar theorem and corollary for a discrete finite $p$-group $G$
are proved by Mathew, Naumann, Noel and Dustin Clausen as
\cite[Theorem C and Corollary 4.1]{CMNN2}.

The following says that when $X$ is a $T(n)$-local ring, the functor
$\rK _{T(n+1)}$ converts the height $n$ cyclotomic sheaf of
\cref{def-cyclo-sheaf} on $R$ to the height $n+1$ one on
$\rK _{T(n+1)}R$.  For $n=0$ this is proved by Clausen, Mathew and
Bhargav Bhatt in \cite{BCM20}.

\begin{theorem}\label{thm-redshift-B}
{\bf Cyclotomic redshift and hyperdescent.} 
\cite[Theorems B and C]{BMCSY23} Let \( R \) be a \( T(n) \)-local
ring spectrum.  Then there is a \(\Z_p^\times \)-equivariant
isomorphism:
\[
\rK_{T(n+1)}( R[\omega^{(n)}_{p^\infty}] )
 \simeq \rK_{T(n+1)} (R)[ \omega^{(n+1)}_{p^\infty}],
\]

\noindent and the sheaf on $\Fin_{\Z^{\times }_{p}}$
given by
\begin{displaymath}
(\Z/p^{j})^{\times }\mapsto \rK_{T(n+1)} (R)[ \omega^{(n+1)}_{p^{j}}]
\end{displaymath}
 
\noindent is a hypersheaf as in \cref{defin-hypersheaf}.
\end{theorem}

\section{Locally unipotent $\Z$-actions}\label{sec-unipotent}

\begin{defin}\label{def-loc-uni}
{\bf Local unipotence.}  For an action of $\Z$ on a topological abelian group
$A$, let $\Psi :A\to A$ be the automorphism induced by a generator of
$\Z$. Then the action is {\bf unipotent} if $\Psi -1$ is a
topologically nilpotent endomorphism of $A$.

An action of $\Z$ on a spectrum $R$ is {\bf
locally unipotent} if the induced action on each homotopy group is
unipotent.  We denote by $\qSp^{B\Z,u}$ the {\qcat} of spectra $X$
equipped with an action of $\Z$ that is locally unipotent on
$\pi_{*}X$.
\end{defin}

If $A$ is $p$-adically complete, topological nilpotence means that
$\Psi -1$ is nilpotent on $A/p^{i}$ for each $i$, which is implied by
nilpotence on $A/p$.

The word ``local'' above refers to the fact that there need not be
a power of the endomorphism that kills all homotopy groups, but each
such group is killed by some power of it.

As explained in the introduction, the counterexample to the {\TC}
involves a nontrivial but locally unipotent action of $\Z$ on \linebreak 
$R=L_{T (n)}\BPn$ that is induced by one on $\BPn$ itself.

\subsection{Trivializing a locally unipotent $\Z$-action}
\label{sec-triv-loc-unip}

\begin{theorem}\label{thm-failure}
{\bf The $T (n+1)$-local $K$-theory coassembly map for the trivial
$\Z$-action.}  \cite[Theorem 3.22]{BHLS} Let $R$ be a $T
(n)$-local $\EE_{1}$-ring spectrum for $n\geq 1$ and let $X$ be a
spectrum.  If $L_{T (n+1)} (X\otimes \rK(R))$ (see \cref{def-KT-KK}) is
nontrivial, then the coassembly map of
(\cite[Definition 5.6]{Rav:gjmcyc}) for the trivial action
of $\Z$ on $R$,
\begin{displaymath}
  \epsilon :L_{T (n+1)} (X\otimes \rK(R^{B\Z}))
   \to L_{T (n+1)} (X\otimes \rK(R))^{B\Z},
\end{displaymath}

\noindent is {\em not} an equivalence.  In particular the map
\begin{displaymath}
  \epsilon :\rK_{T (n+1)} (R^{B\Z})
   \to \rK_{T (n+1)} (R)^{B\Z},
\end{displaymath}

\noindent is not an equivalence when $\rK_{T (n+1)} (R)$ is
nontrivial. \end{theorem}

\begin{theorem}\label{thm-fpLQ}
  {\bf  Asymptotic constancy for $\THH$.}
  \cite[Theorem C]{BHLS} Let $\EE\AA_{2}$ be the operad of
  \cref{ex-E1A2}, and suppose that $R$ is an $\EE\AA_{2}$-ring
  spectrum (such as $\BPn$) that is connective, $p$-complete and of
  {\fp}-type $n$ as in \cref{def-MR99}.  Suppose also that it has a
  locally unipotent action of $\Z$.

  If $R$ has the height $n$ {\LQ} property of \cref{def-LQ}, then for
  $k\gg0$ so does $R^{h(p^{k}\Z)}$, and for a finite spectrum $F$ of
  type $n+2$, there is a commutative diagram of cyclotomic spectra
\begin{displaymath}
\xymatrix
@R=6mm
@C=10mm
{
{F\otimes \THH (R^{h(p^{k}\Z)})}
  \ar[r]^(.5){\epsilon }\ar[d]_(.4){\cong }
  &{F\otimes \THH (R)^{h(p^{k}\Z)}}\ar[d]^(.4){\cong }\\
 {F\otimes \THH (R^{B(p^{k}\Z)})}
  \ar[r]^(.5){\epsilon }
  &{F\otimes \THH (R)^{B(p^{k}\Z)}}
}
\end{displaymath}

\noindent where the horizontal maps are the coassembly maps of
\cite[Definition 5.6]{Rav:gjmcyc}. 
\end{theorem}

Hence tensoring with $F$ ``trivializes'' the $\Z$-action on $\THH(R)$
in the sense its homotopy fixed point set behaves like that of the
trivial action.  This implies a similar statement for $\TopC(R)$.

Following \cite[Notations and Conventions]{BHLS}, we denote the
{\qcat} of dualizable spectra by $\qSp^{\Diamond}$ (which includes all
finite complexes) and that of associative algebras in a symmetric
monoidal {\qcat} $\qmcC$ by $\Alg (\qmcC)$.

\begin{defin}\label{def-UAlg}\cite[Definition 4.10]{BHLS}
  We let $\UAlg (\qSp)$ be the presentably symmetric monoidal {\qcat}
  defined by the pullback square
\begin{displaymath}
\xymatrix
@R=8mm
@C=8mm
{
{\UAlg (\qSp)}\ar[r]^(.5){}\ar[d]^(.5){}
              \ar@{}[dr]|-(.15){\pb}
  &{\Alg(\qSp)^{B\Z,u} \times \Alg(\qSp^\Diamond) }
         \ar[d]^(.5){}
    &{(R,V)}\ar@{|->}[d]^(.5){}\\
{\Alg (\qSp)}\ar[r]_(.45){\triv}
  &{\Alg(\qSp)^{B\Z,u}.}
    &{R\otimes V}
}
\end{displaymath}
\end{defin}

Thus an object in this {\qcat} is a pair $(R,V)$, where $R$ is a ring
spectrum with a unipotent $\Z$-action and $V$ is a dualizable ring
spectrum on which $\Z$ acts trivially, such that the diagonal action
on $R\otimes V$ is trivial.

For the next result we need some notation.  Let $C^{0}(\Z_{p})$ denote
the ring of continuous (meaning locally constant) $\FF_{p}$-valued
functions on the $p$-adic integers.  Any $a \in \Z_{p}$ can be written
uniquely as
\begin{displaymath}
  a=\sum_{k\geq 0}a_{k}p^{k}\qquad \mbox{with }a_{k}^{p}=a_{k}. 
\end{displaymath}
 
\noindent Hence each coefficient $a_{k}$ is either zero or a $(p-1)$th
root of unity.  Thus its mod $p$ reduction, which we alos denote (abusively)
by $a_{k}$, is a continuous $\FF_{p}$-valued function.  It turns out
that the ring of all such functions is
\begin{displaymath}
  C^{0}(\Z_{p})= \FF_{p}[a_{k}:k\geq 0]/(a_{k}^{p}-a_{k}).
\end{displaymath}
 
\noindent See \cite[Example 6.5]{Rav:gjmcyc} for more discussion.

Let $\mW (C^{0}(\Z_{p}))$ denote the commutative ring spectrum of
spherical Witt vectors as in \cite[Example 5.2.7]{Lurie:Ell2}. It is a
countable coproduct of $p$-adic sphere spectra.

The proof of the following occupies seven pages of \cite[\S4.2]{BHLS}
and makes use of the {\em Dehn twist} of \cite[\S4.2.3]{BHLS}.

\begin{theorem}\label{thm-BHLS-4.11}
  \cite[Theorem 4.11]{BHLS}
There is a natural transformation $\theta :M_{1}\implies M_{2}$
(denoted by $\eta $ in \cite[Theorem 4.11]{BHLS}) of lax symmetric
monoidal functors
\begin{displaymath}
M_{1},M_{2}:\UAlg (\qSp) \to \qSp^{\Delta^1}
\end{displaymath}

\noindent given by
\begin{align*}
M_{1} (R,V)
 & := \mW(C^{0}(\Z_p)) \otimes V \otimes 
          \mathrm{res}_\varphi (\THH(R)^{h\Z})  \\
\aand 
M_{2} (R,V)
 & := V \otimes \mathrm{res}_\varphi \THH(R^{h\Z}),
\end{align*}

\noindent 
such that:
\begin{enumerate}

\item
$\theta$ becomes an isomorphism after composing with pullback along \linebreak 
$i_0 \colon \Delta^{0} \to \Delta^1$.

\item $\theta$ becomes an isomorphism upon restricting to the full
  subcategory of those $(R,V)$ for which the Tate coassembly map
\begin{displaymath}
\xymatrix
@R=4mm
@C=10mm
{
  {\mW(C^{0}(\Z_p))\otimes \bigl((V \otimes \THH(R))^{tC_p}\bigr)}
  \ar[d]^(.5){}\\
  {\bigl(\mW(C^{0}(\Z_p))\otimes V \otimes \THH(R)\bigr)^{tC_p}}
}
\end{displaymath}
 
\noindent is an isomorphism.
\end{enumerate}

Here the abusively denoted composite functor
$\mathrm{res}_\varphi := \mathrm{res}_\varphi\,\res_{\pentagon} $ as
in \cite[Definition 5.40]{Rav:gjmcyc}, sends a $p$-cyclotomic
spectrum $X$ to its Frobenius map $\varphi_{p}:X\to X^{t\rC_{p}}$,
which is an object in the morphism category $\qSp^{\Delta^1}$, which
is tensored over $\qSp$.

\end{theorem}

The first property means that the two morphisms $M_{1} (R,V)$ and
$M_{2} (R,V)$ have the same domain.  

The proof of the next result occupies two pages of  \cite[\S4.2]{BHLS}.

\begin{theorem}\label{thm-asymptotic-boundedness-THH}
{\bf Asymptotic boundedness for $\THH$.} \cite[Theorem 4.30]{BHLS}
Let 
\[
  R \in \mathrm{Alg}_{\EE\AA_{2}}(\mathrm{\qSp}^{B\Z,u})
\]
be connective of fp-type \(n \ge -1\), and let \(F\) be a finite
spectrum of type \(n+2\). Suppose that \(F \otimes \mathrm{THH}(R)\)
is bounded in the range \([c,b]\) in the sense of the Antieau-Nikolaus
$t$-structure of \cite{AN21}, which is described in
\cite[\S5.11]{Rav:gjmcyc}.

 

Then for $k\gg 0$,
the spectrum \(F \otimes \mathrm{THH}(R^{h (p^k
 \Z)})\) is bounded in the range \([c-1,\, b+3]\), and there is
an isomorphism of \(W_k\)-modules in cyclotomic spectra
\[
  F \otimes \mathrm{THH}(R^{h(p^k \Z)})
    \;\simeq\;
  F \otimes W_k \otimes \mathrm{THH}(R),
\]

\noindent where $W_{k}:=\THH (\mS^{B (p^{k}\Z)})$.
\end{theorem}

The spectrum $W_{k}$ above is described by Cary Malkiewich in
\cite[Corollary 1.3]{malkiewic} (quoted as
\cite[Theorem 6.2]{Rav:gjmcyc}) as a cyclotomic
spectrum with a single underlying cell in dimension $-1$ and countably
many in dimension 0.

\begin{cor}\label{thm-tel-asym-const}
{\bf  Asymptotic constancy for $\TopC $.} \cite[Corollary 4.33]{BHLS} 
Let 
\[
  R \in \mathrm{Alg}_{\EE\AA_{2}}
            (\mathrm{\qSp}^{B\Z,u})
\]
be connective, of fp-type \(n \ge 0\), and satisfying the height \(n\)
{\LQ} property of \cref{def-LQ}.  Fix a finite spectrum $F$
of type \(n+1\) with a \(v_{n+1}\)-self map \(v\).


Then for $k\gg 0$, 
there is a commutative diagram of
\(\Z[v]\)-modules as below, where the horizontal maps are the
coassembly maps:
\begin{displaymath}
\xymatrix
@R=6mm
@C=10mm
{
 {\pi_{*}(F \otimes  TC(R^{h (p^k \Z)}))}
     \ar[r]^(.5){\epsilon }\ar[d]_(.4){\cong }
  &{\pi_{*}(F \otimes  TC(R)^{h (p^k \Z)}) }
                 \ar[d]^(.4){\cong }\\
  {\pi_{*}(F \otimes  TC(R^{B(p^{k}\Z)})) }
     \ar[r]^(.5){\epsilon }
   &{\pi_{*}(F \otimes  TC(R)^{B(p^{k}\Z)}).}
}
\end{displaymath}
\end{cor}

In the above the vertical isomorphisms are not induced by isomorphisms
of spectra.

The following is needed in the proof of \cref{thm-BHLS-6.25}.

\begin{lem}\label{lem-BHLS-6.15}
  {\bf The coassembly map and the {\Cech} nerve.}
  \cite[Lemma 6.15]{BHLS} Let
  $R \in \CAlg\!\bigl(\qSp^{B\Z, u}_{T(n)}\bigr)$ for $n \ge 1$.
  There is a commuting triangle, natural in $R$,
\[
\begin{tikzcd}
  \rK_{T(n+1)}(R^{h\Z}) \arrow[rr, "\epsilon"] \arrow[dr]
  & & \rK_{T(n+1)}(R)^{h\Z} \\
  & \displaystyle\lim{\Delta} \rK_{T(n+1)}
  \!\bigl(R \otimes_{R^{h\Z}} R^{h\Z^{\bullet+1}}\bigr)
       \arrow[ur, "\simeq "'] &
\end{tikzcd}
\]

\noindent identifying the coassembly map $\epsilon$ with the {\Cech}
nerve (\cref{def-Cech-nerve}) of the map
$i_{R}:R^{h\Z} \to R$.  In particular, the coassembly map for
$R$ is an isomorphism if and only if $i_{R}$ satisfies
$\rK_{T(n+1)}$-descent as in \cref{def-F-descent}.

\end{lem}

\subsection{Adams operations on  $BP\langle n \rangle$}\label{sec-Adams-BPn}

The main counterexample of \cite{BHLS} (see \cref{thm-main-A})
involves an action of the integers on the Johnson-Wilson spectrum
$BP\langle n \rangle$ via Adams operations.  The construction of such
operations is the subject of their \S5.

In \cite[Theorem C]{BHLS} (our \cref{thm-fpLQ}) the hypothesis on $R$
is that it is an algebra over the operad
$\EE\AA_{2}:=\EE_{1}\otimes_{\BV}\AA_{2}$ of \cref{ex-E1A2}. Here
$(-\otimes_{\BV}-)$ denotes the tensor product of operads defined by
Michael Boardman and Rainer Vogt in \cite{Boardman-Vogt} and discussed
in \cref{sec-BV}.  The operads $\AA_{2}$ and $\EE_{1}$ are defined in
\cref{def-Stash-op,def-little-cubes} respectively.  Such a structure
is intermediate between $\EE_{1}$ and $\EE_{2}$.

They choose $\EE\AA_{2}$ because it is the strongest structure they
can establish for the Adams operations $\Psi^{\ell }$ on $\BPn$; see
\cite[Theorem 5.4 and Remark 5.5]{BHLS}.  Even though $\BPn$ itself is
known to be an $\EE_{3}$-algebra over the $\EE_{\infty}$-ring
$MU_{(p)}$, they can only show that their Adams operation on it, and
hence its homotopy fixed point set, has the weaker structure.  Recall
that Vogt, Morten Brun and Zbigniew Fiedorowicz\cite{BFV} show that
for an $\EE_{m}$-ring spectrum $R$ with $m\geq 2$, $\THH (R)$ is an
$\EE_{m-1}$-ring spectrum.  Apparently if $R$ is an
$\EE\AA_{2}$-algebra, then $\THH (R)$ is an $\AA_{2}$-algebra, but we
know of no published proof of this.

The Adams operation we want is denoted by $\Psi^{\ell }$, where
$\ell=3$ when the implicit prime $p$ is 2, and $\ell =2$ when $p$ is
odd. For any $p$-local unit $k$ one has an infinite loop map
$\Psi^{k}:BU_{(p)}\to BU_{(p)}$ which induces multiplication by
$k^{i}$ in $\pi_{2i}$ for each $i>0$.  This map can be Thomified to an
$\EE_{\infty}$-map $\Psi^{k}:MU_{(p)}\to MU_{(p)}$ with similar
properties.
There is a similarly named $\EE_{\infty}^{}$-endomorphism of the
Morava spectrum $E_{n}$ that is related to the formal group law
endomorphism $[k](x)$.  More details can be found in \cite[\S5]{BHLS}.

We lose the $\EE_{\infty}$-structure when we pass to $\BPn$, which is only
known to be an $\EE_{3}$-$MU_{(p)}$-algebra.  There the induced Adams
operation is only known to to preserve the still weaker structure of an
$\EE\AA_{2}$-$MU_{(p)}$-algebra.

The proof of the following runs for twelve pages and makes use of
factorization homology in \cite[Proposition 5.11]{BHLS}.

\begin{theorem}\label{thm-Theorem-5.4}
{\bf  The curative effect of $T(n)$-localization on $\BPn$.}
\cite[Theorem 5.4]{BHLS} The $\EE\AA_{2}$--$MU_{(p)}$-algebra
underlying the $\EE_3$--$MU_{(p)}$-algebra $\BPn$ admits a lift to an
object
\[
  \BPn^{\Psi} \in 
  \Alg_{\EE\AA_{2}}\!\bigl(\Mod(\qSp^{B\Z}; MU^{\Psi}_{(p)})\bigr)
\]
such that:
\begin{enumerate}[label={(\roman*)},itemindent=0em]

\item There is a map
\[
  \iota \colon \BPn^{\Psi} \longrightarrow E^{\Psi}_n
\]
in 
\(
  \Alg_{\EE_1}\!\bigl(\Mod(\qSp^{B\Z}; MU^{\Psi}_{(p)})\bigr).
\)

\item There is an identification
  \[
    \xymatrix
@R=6mm
@C=25mm
{
  {L_{T(n)} \BPn^{\Psi}}\ar[r]^(.6){L_{T(n)}(\iota)}
                        \ar[d]_(.4){\cong }
    &{E_{n}^{\Psi}}\ar@{=}[d]^(.5){}\\
{(E^{\Psi}_{n})^{h(\mu_{p^{n}-1}\rtimes \widehat{\Z})}}
                  \ar[r]^(.5){}
    &{E_{n}^{\Psi}}
}
\]
in $\Alg_{\EE_1}(\qSp^{B\Z})$.  The two spectra other than
$L_{T(n)} \BPn^{\Psi}$ are $\EE_{\infty}$.  The subgroup
$\mu_{p^{n}-1} \rtimes \widehat{\Z} \subseteq \GG_{n}$ fits into
a diagram with exact rows
\begin{numequation}\label{eq-map-exact}
  \begin{split}
\xymatrix
@R=6mm
@C=10mm
{
{\mu_{p^{n}-1}} \ar[r]^(.5){}\ar@{=}[d]^(.5){}
    &{\mu_{p^{n}-1} \times n\widehat{\Z}}
                \ar[r]^(.5){}\ar[d]^(.5){}
    &{n\widehat{\Z}} \ar@{^{(}->}[d]^(.4){}
    \\
{\mu_{p^{n}-1}} \ar[r]^(.5){}\ar[d]^(.5){}
    &{\mu_{p^{n}-1} \rtimes \widehat{\Z}}
      \ar[r]^(.5){}\ar[d]^(.5){}
    &{\widehat{\Z}} \ar[d]^(.4){\cong}
    \\
{\mathcal{O}^{\times}_D} \ar[r]^(.5){}
   &{\GG_{n}} \ar[r]^(.5){}
      &\mathrm{Gal}(\mathbb{F}_p),
} 
  \end{split}
\end{numequation}

\noindent where the bottom row is the {\SES} of \cref{eq-SGZ}.  In the
middle row the Frobenius generator of $\widehat{\Z}$ acts on
$\mu_{p^{n}-1}$ by raising each root of unity to its $p$th power.

\item The underlying $\Z$-action on $\BPn$ is locally unipotent
in $p$-complete spectra after $p$-completion.
\end{enumerate}
\end{theorem}

Local unipotence is implied by the fact that $\ell ^{p-1}-1$ is
divisible by $p$. The generator of $\widehat{\Z}$ acts on
$\pi_{*}E_{n}$ via the lifting of the Frobenius automorphism of
$\barFFp $ to its Witt ring.  It follows that
\begin{displaymath}
E_{n}\simeq E_{n}(\barFFp) ^{h(n\widehat{\Z})};
\end{displaymath}
 
\noindent the action of $n\widehat{\Z}$ on the scalars in
$\WW(\barFFp)$ fixes $\WW(\FF_{p^{n}})$.  Meanwhile the
action of $\mu_{p^{n}-1}$ on $\pi_{*}E_{n}$ is such that
\begin{align*}
  \pi_{*}\left(E_{n}^{h\mu_{p^{n}-1}} \right) 
  & =  W (\FF_{p^{n}})\pow{u^{p-1}u_{1},\dotsc ,u^{p^{n}-1-1}u_{n-1}}
  [u^{\pm (p^{n}-1)}]\\
  & =  W (\FF_{p^{n}})\pow{v_{1},\dotsc ,v_{n-1}}  [v_{n}^{\pm 1}]\\
\mbox{and}\quad  
  \pi_{*}\left(E_{n}(\barFFp )
            ^{h( \mu_{p^{n}-1} \rtimes \widehat{\Z})} \right)
  & = \Z_{p}\pow{v_{1},\dotsc ,v_{n-1}}  [v_{n}^{\pm 1}]\\
  & = \pi_{*}L_{T(n)}\BPn.
\end{align*}

\noindent {\em This means the $\EE\AA_{2}$-structure of the group
  action on $\BPn$ becomes an $\EE_{\infty}$-structure upon passage to
  $L_{T(n)}\BPn$.}

The proof of the following occupies most of \cite[\S4]{BHLS}.

\begin{theorem}\label{thm-thmB}
  {\bf Telescopic $\TopC$ asymptotic constancy for $\BPn$.}\linebreak
  \cite[Theorem B]{BHLS} 
Fix a  telescope $T(n + 1)$ of a type $n+1$ $p$-local finite
spectrum $F(n+1)$. Then for all $k\gg 0$, there is a commuting square
\begin{displaymath}
\xymatrix
@R=6mm
@C=10mm
{
{T(n+1)_{*}\TopC (\BPn^{h (p^{k}\Z)})}
     \ar[r]^(.5){\epsilon}\ar[d]_(.45){\cong  }
  &{T(n+1)_{*}\TopC (\BPn)^{h (p^{k}\Z)}}
     \ar[d]^(.45){\cong }\\
{T(n+1)_{*}\TopC (\BPn^{B(p^{k}\Z)})}
     \ar[r]^(.5){\epsilon}
  &{T(n+1)_{*}\TopC (\BPn)^{B(p^{k}\Z)}),}
}
\end{displaymath}

\noindent where the horizontal maps are $\TopC$ coassembly maps.
\end{theorem}

\section{The main counterexample}\label{sec-main-counter}

As in \cite[\S6.1]{BHLS}, we start with

\begin{theorem}\label{thm-purity}
{\bf Telescopic $K$-theory for connective ring spectra.}
  \cite[Purity Theorem and Corollary 4.30]{LMMT} and \cite[Corollary
  4.11]{CMNN1}.  Let $R$ be a connective $\mathbb{E}_1$-algebra. For
  $n \geq 1$, the $(T(n)\oplus T(n+1))$-localization map and the
  cyclotomic trace induce isomorphisms
\[
\rK_{T(n+1)}\bigl(L_{T(n)\oplus T(n+1)} R\bigr)
\;\xleftarrow[]{\;\cong\;}\;
\rK_{T(n+1)}(R)
\;\xrightarrow[]{\;\cong\;}\;
\mathrm{TC}_{T(n+1)}(R).
\]
\end{theorem}

With additional hypotheses we have

\begin{cor}\label{cor-BHLS-6.3}
{\bf Connective rings with $\Z$-action.} \cite[Corollary 6.3]{BHLS} 
For $n \ge 1$, let $R$ be a $T(n+1)$-acyclic, connective
$\mathbb{E}_1$-algebra with a $\Z$-action.  The coassembly
map, the $T(n)$-localization map, and the cyclotomic trace fit into a
commuting diagram
\[
\xymatrix
@C=10mm
@R=6mm
{
\rK_{T(n+1)}(L_{T(n)} R^{h\Z}) 
 \ar[r]^(.5){\epsilon }
  &
\rK_{T(n+1)}(L_{T(n)} R)^{h\Z} \\
\rK_{T(n+1)}( R^{h\Z}) \ar[r]^(.5){\epsilon }
\ar[d]_{\simeq} \ar[u]^{\simeq}   
   &
\rK_{T(n+1)}(R)^{h\Z}
   \ar[d]_{\simeq} \ar[u]^{\simeq}   \\
\TopC_{T(n+1)}(R^{h\Z}) \ar[r]^(.5){\epsilon }
   &
\TopC_{T(n+1)}(R)^{h\Z},
}
\]
\end{cor}

The ring of interest is $\BPn$ with the $\Z$-action given by the
Adams operations of \cref{sec-Adams-BPn}. While its known
multiplicative structure ($\EE\AA_{2}$) is relatively weak, 
\cref{thm-Theorem-5.4} says that $T(n)$-localization elevates it to
$\EE_{\infty}$.

\begin{theorem}\label{thm-main-A}
{\bf The main counterexample.} \cite[Theorem A]{BHLS}
Let \( p \) be any prime and \( n  \geq 1 \). Then, for all \( k \geq 0 \),
\[
\rK_{T (n+1)}\left( BP\langle n \rangle^{h(p^{k}\Z)} \right)
\]
is  not \( K(n+1) \)-local. In particular,
\[
\mathrm{\qSp}_{K(n+1)} \neq \mathrm{\qSp}_{T(n+1)}.
\]
\end{theorem}

On \cite[page~5]{BHLS} they say (using our notation)
\begin{quote}
  Using cyclotomic redshift [see \cref{sec-cyclo-redshift}], and the
  fact that the $(p^{k}\Z_{p})$-pro-Galois extension
\[
L_{T(n)}{\BPn}^{h(p^{k}\Z)}
   \;\longrightarrow\;
L_{T(n)}{\BPn}
\]

\noindent is closely related to a [higher] cyclotomic extension [see
\cref{sec-height-n}], we deduce that there is an equivalence
\[
\rK_{T(n+1)}\!\left({\BPn}\right)^{h(p^{k}\Z)}
   \;\simeq\;
\rK_{\Cyc(n+1)}\!\left({\BPn}^{h(p^{k}\Z)}\right)
\]
for each $k \ge 0$. Thus, in order to prove \cref{thm-main-A}, it
suffices to show that the coassembly map
\[
\rK_{T(n+1)}\!\left({\BPn}^{h(p^{k}\Z)}\right)
   \;\longrightarrow\;
\rK_{T(n+1)}\!\left({\BPn}\right)^{h(p^{k}\Z)}
\]
is not an equivalence.
\end{quote}

\begin{lem}\label{lem-}
{\bf A cyclotomic completeness criterion.} \cite[Lemma 6.22]{BHLS}
For $n \geq 1$, if $R$ is a $T(n)$-local commutative ring spectrum
and there is a map in $\mathrm{CAlg}(\qSp_{T(n)})$
\[
  S[\omega^{(n)}_{p^\infty}]^{hT_p}
  \longrightarrow L_{T(n)}\bigl(W(\barFFp) \otimes R\bigr)
  \qquad \mbox{ for $T_{p}$ as in \cref{eq-dense-discrete},} 
\]

\noindent then $\rK_{T(n+1)}(R)$ is height $n+1$ cyclotomically complete.
\end{lem}

The proof of the following requires a full page and several lemmas in
\cite{BHLS}.  It uses the machinery sketched in \cref{sec-height-n}.
Recall that $\rK_{\Cyc(n+1)}$ (height $n+1$ cyclotomically complete algebraic
$K$-theory) is defined in \cref{def-KT-KK}.

\begin{prop}\label{prop-BHLS6.24}
  {\bf $\rK_{\Cyc(n+1)}$-descent and cyclotomic completeness.}
  \cite[Prop. 6.24]{BHLS} For $n \geq 1$, let $R = L_{T(n)} \BPn$
  with the $\EE_{\infty}$-$\Z$-action of
  \cref{thm-Theorem-5.4}.  Then for each $k \geq 0$, the map
\[
  f_k \colon R^{h(p^{k} \Z)} \longrightarrow R
\]

\noindent is a $\rK_{\Cyc(n+1)}$-cover as in \cref{def-F-descent} (also see
\cref{ex-cyclo-comp}), and $\rK_{T(n+1)}(R)$ is height $n+1$
cyclotomically complete as in \cref{def-cyc-comp}.
\end{prop}

\cref{thm-main-A} for a given value of $k$ implies the same for
smaller values of $k$, so it suffices to prove it for $k\gg 0$.  It is
implied by the following, which is related to \cref{cor-BHLS-6.3} and
relies on \cref{lem-BHLS-6.15}.

\begin{theorem}\label{thm-BHLS-6.25}
{\bf Coassembly as cyclotomic completion.}
\cite[Theorem 6.25]{BHLS}
Let $BP\langle n \rangle$ be as in Theorem 5.4. 
For every prime $p$, height $n \geq 1$, and $k \geq 0$, there is a diagram
\[
\xymatrix
@C=10mm
@R=6mm
{
\rK_{T(n+1)}(L_{T(n)} BP\langle n \rangle^{h(p^{k}\Z)}) 
 \ar[r]^(.5){\epsilon }
  &
\rK_{T(n+1)}(L_{T(n)} BP\langle n \rangle)^{h(p^{k}\Z)} \\
\rK_{T(n+1)}( BP\langle n \rangle^{h(p^{k}\Z)}) \ar[r]^(.5){\epsilon }
\ar[d]_{\simeq} \ar[u]^{\simeq}   
   &
\rK_{T(n+1)}(BP\langle n \rangle)^{h(p^{k}\Z)}
   \ar[d]_{\simeq} \ar[u]^{\simeq}   \\
\TopC_{T(n+1)}(BP\langle n \rangle^{h(p^{k}\Z)}) \ar[r]^(.5){\epsilon }
   &
\TopC_{T(n+1)}(BP\langle n \rangle)^{h(p^{k}\Z)},
}
\]

\noindent where the horizontal maps are coassembly maps as in
\cite[Definition 5.6]{Rav:gjmcyc}. These maps are not
isomorphisms, but rather exhibit the target as the cyclotomic
completion of the source.  Hence the coassembly map on
\begin{displaymath}
\rK_{\Cyc(n+1)}(BP\langle n \rangle^{h(p^{k}\Z)})
\end{displaymath}
 
\noindent
is an equivalence, while the middle coassembly map above is not.  In
particular, this gives a counterexample to the height $n+1$ telescope
conjecture.
\end{theorem}

\appendix

\section{Some {\qcatal} constructions}
  \label[appendix]{sec-qcat-defs}
  
  \subsection{Miscellaneous definitions}\label[appendix]{sec-basic-defs}

  We need the following definitions for future reference.  We suggest
  skipping this subsection initially and referring to specific items
  in it as needed.


\subsubsection{Augmented simplicial objects}\label{sec-aug-simp-obj}

  Recall that the {\bf  simplicial category} $\bdelt$ is that of finite
ordered sets and order preserving maps.  It has an object
$[n]=\left\{0, 1, \dots, n \right\} $ for each integer $n\geq 0$.
More details can be found in
\cite[\S2.4]{Rav:gjmcyc} and in \cite[Chapter
I]{Goerss-Jardine}.  Here we consider the {\bf  augmented simplicial
  category} $\bdelt_{+}$, which is obtained from $\bdelt$ by adding an
additional object, the empty set, denoted by $[-1]$. It is an initial
object, while $[0]$ is a terminal object.  There is a unique morphism
$[-1]\to [n]$ for each $n\geq 0$, but no morphisms going the other
way.  Hence the category $\bdelt_{+}$ is not pointed.

\begin{defin}\label{def-6.1.2.2}
  \cite[Definition 6.1.2.2]{Lurie:HTT} A {\bf simplicial object of an
    {\qcat} $\qmcC$} is a functor
\[
V_{\bullet} \colon \mathrm{N}(\bdelt)^{\op} \longrightarrow \qmcC.
\]
An {\bf augmented simplicial object of $\qmcC$} is an extension of $V$
to a map
\[
V^{+}_{\bullet} \colon \mathrm{N}(\bdelt_{+})^{\op} \longrightarrow \qmcC.
\]

We let
\[
\qmcC_{\bdelt} := \operatorname{Fun}\!\bigl(\mathrm{N}(\bdelt)^{\op},\,
\qmcC\bigr)
\]
denote the {\qcat} of simplicial objects of $\qmcC$.  
Similarly, we refer to
\[
  \qmcC_{\bdelt_{+}}
  := \operatorname{Fun}\!\bigl(\mathrm{N}(\bdelt_{+})^{\op},\,\qmcC\bigr)
\]
as the {\qcat} of augmented simplicial objects of $\qmcC$.

If $V_{\bullet}$ (or $V^{+}_{\bullet}$) is an (augmented) simplicial
object of $\qmcC$ and $n \ge 0$ (respectively $n \ge -1$), we write
\[
V_{n}  := V([n]) \in \qmcC.
\]

We denote by $d_{0}$ the map $V_{0}\to V_{-1}$ induced by the 
map  $[-1]\to [0]$.
\end{defin}

The simplicial sets $\Delta^{n}$ (the standard $n$-simplex) and
$\Lambda^{n}_{i}$ (its $i$th horn, meaning its boundary with the
interior of $i$th $(n-1)$-face removed) mentioned below are defined in
\cite[Definition 2.21]{Rav:gjmcyc} and on
\cite[page~6]{Goerss-Jardine}.

\begin{defin}\label{defin-fibrations}
  {\bf Various fibrations.}  \cite[Definition 2.0.0.3]{Lurie:HTT}
A morphism $f : X \to S$ of simplicial sets ({\eg} a functor of {\qcats}) is
\begin{itemize}
  \item a {\bf Kan fibration} if $f$ has the right lifting property with 
  respect to all horn inclusions $\Lambda^{n}_{i} \hookrightarrow \Delta^{n}$ 
  for $0 \le i \leq  n$;

  \item a {\bf left fibration} if the same holds 
  for $0 \le i < n$;

  \item a {\bf right fibration} if it holds
  for $0 < i \le n$;

\item an {\bf inner fibration} (or {\bf Joyal fibration}) if it holds 
  for $0 < i < n$.
\end{itemize}
\end{defin}  

We will use left fibrations to define left exact functors in
\cref{defin-left-exact}.

\subsubsection{Limits and colimits}\label{sec-lim-colim}

Recall that limits and colimits in {\qcats} of spaces or spectra
coincide with homotopy limits and colimits, as defined by Bousfield
and Dan Kan in \cite{BK1}, in the corresponding ordinary categories
with homotopy structure.

Limits and colimits in an {\qcat} are defined by Joyal in
\cite[Definition 4.5]{Joyal:QCKC}, which is quoted by Lurie as
\cite[Definition 1.2.13.4]{Lurie:HTT} and again as
\cite[Definition 5.5]{Rav:gjmcyc}.

\begin{defin}\label{def-cone}
  \cite[Dual of Remark 1.2.13.5]{Lurie:HTT}
Let $\qmcC$ be an $\infty$-category and let 
$p \colon K \to \qmcC$ be a diagram indexed by a simplicial set $K$.
An extension 
\[
  \bar{p} \colon K^{\triangleleft}
  \longrightarrow \qmcC
\]
is called a {\bf limit diagram} if the cone point 
$\bar{p}(-\infty)$ is a limit of $p$.

Here $K^{\triangleleft}$ is the {\bf left cone of $K$}, the simplicial
set obtained from $K$ by adding a vertex, denoted by $-\infty$, the cone point,
with an edge pointing to each vertex in $K$, a 2-simplex for each edge
in $K$ and so on.  See \cite[\S1.2.8 and Notation 1.2.8.4]{Lurie:HTT}
for more details.

Equivalently, $\bar{p}$ is a limit diagram if for every object 
$X \in \qmcC$, the induced map of mapping spaces
\[
\Map_{\qmcC}\!\bigl(X,\, \bar{p}(-\infty)\bigr)
\;\longrightarrow\;
\lim{k \in K} \Map_{\qmcC}\!\bigl(X,\, p(k)\bigr)
\]
is an equivalence.

There is a similarly defined right cone $K^{\triangleright}$ that is
relevant to colimits.
\end{defin}

We will see the right cone in \cref{theorem-topos-property}.

\subsubsection{Truncation and connectivity}\label{sec-trunc-conn}

\begin{defin}\label{defin-truncated-object}
  {\bf Truncated and connective objects and morphisms in an {\qcat}.}
  An object $D$ in an {\qcat} ${\qmcC}$ is {\bf $m$-truncated
    ($m$-connective)} of for each object $C$ in ${\qmcC}$ the space of
  maps $\Map_{\qmcC} (C, D)$ is $m$-truncated ($m$-connective). In
  particular it is {\bf discrete} if it is 0-truncated, {\bf
    connected} if it is 1-connective, and a final object if it is
  $(-2)$-truncated.  It is $(-1)$-truncated if each such mapping space
  is either empty or contractible.

  A morphism \(f : C \to D\) in \({\qmcC}\) is {\bf $m$-truncated
    ($m$-connective)} if, for every object \(E \in {\qmcC}\),
  composition with \(f\) induces an $m$-truncated ($m$-connective) map
  of spaces
\[
f_{*}:\Map_{{\qmcC}}(E, C) \longrightarrow \Map_{\qmcC}(E, D)
\]
which in the latter case is an effective epimorphism as in
\cref{def-effective-epi}.  It is {\bf $\infty$-connective} if each
such $f_{*}$ is a weak equivalence.

The full sub-categories $\tau_{\leq m}{\qmcC}$ and
$\tau_{\geq m}{\qmcC}$ of ${\qmcC}$ are those of $m$-truncated and
$m$-connective objects.  Hence we can regard $\tau_{\leq m}$ and
$\tau_{\geq m}$ as endofunctors in the {\qcat} of {\qcats}.
\end{defin}

\begin{defin}\label{def-truncated morphism}
  \cite[Def.~5.5.6.8]{Lurie:HTT}
Let $f : X \to Y$ be a map of Kan complexes.  
It  is {\bf $m$-truncated} if every homotopy fiber of $f$
(taken over any base point of $Y$) is $m$-truncated.

More generally, let $\qmcC$ be an $\infty$-category and let 
$f : C \to D$ be a morphism in $\qmcC$.  
It is {\bf $m$-truncated} if for every object 
$E \in \qmcC$, the induced map
\[
\Map_{\qmcC}(E, C) \longrightarrow \Map_{\qmcC}(E, D)
\]
is a $m$-truncated map of spaces.
\end{defin}

The following is immediate.

\begin{prop}\label{prop-same}
A morphism $f:C\to D$ in an {\qcat} $\qmcC$ is $m$-truncated iff the
same is true of $f$ as an object in $\qmcC_{/D}$.
\end{prop}

\begin{remark}\label{remark-m-category}
  An {\bf $m$-category}, defined formally in \cite[Definition
  2.3.4.1]{Lurie:HTT}, is an {\qcat} in which each object is
  $m$-truncated. It can also be defined by induction on $m$ as a
  category enriched over $(m-1)$-categories.

  Thus a $(-1)$-category has mapping spaces that are contractible,
  making all objects equivalent and the structure uninteresting.  A
  $0$-category has mapping spaces that are contractible or empty,
  making it equivalent to a set of equivalence classes of objects, so
  a $0$-category is essentially a set.  A $1$-category has mapping
  spaces that are equivalent to sets, making it equivalent to an
  ordinary category.
\end{remark}

\subsubsection{Accessible and presentable {\qcats}}\label{sec-access-presentable}

\begin{defin}\label{defin-ind-completion}
  \cite[Definition 5.3.5.1]{Lurie:HTT} Let ${\qmcC}$ be a small
  {\qcat} and let $\kappa$ be a regular cardinal. We let
  $\operatorname{Ind}_{\kappa}({\qmcC})$, the {\bf $\kappa$-filtered
    colimit category of $\qmcC$}, denote the full subcategory of
  $\mathcal{P}({\qmcC})$ spanned by those functors
  $f : {\qmcC}^{\op} \to {\qmcS}$ which classify right
  fibrations $\widetilde{{\qmcC}} \to {\qmcC}$, where the {\qcat}
  $\widetilde{{\qmcC}}$ has $\kappa$-filtered colimits. In the case
  where $\kappa = \omega$, we simply write
  $\operatorname{Ind}({\qmcC})$ for
  $\operatorname{Ind}_{\kappa}({\qmcC})$.
\end{defin}

Accessibility is the subject of \cite[\S5.4.2]{Lurie:HTT}. 

\begin{defin}\label{defin-kappa-accessible}
\cite[Definition 5.4.2.1]{Lurie:HTT} 
Let $\kappa$ be a regular cardinal. An {\qcat} ${\qmcC}$ is 
{\bf  $\kappa$-accessible} if there exists a small {\qcat} 
${\qmcC}^{0}$ and an equivalence
\[
  \operatorname{Ind}_{\kappa}({\qmcC}^{0}) \;\longrightarrow\; {\qmcC}.
\]
We say that ${\qmcC}$ is {\bf accessible} if it is $\kappa$-accessible for 
some regular cardinal $\kappa$.
\end{defin}

The presheaf category $\mcP (\qmcC)$ of \cref{defin-presheaves} is
accessible because ${\qmcS}$ is.

\begin{prop}\label{prop-accessible-qcat-props}
{\bf Properties of accessible {\qcats}.} \cite[Proposition 5.4.2.2]{Lurie:HTT} 
Let ${\qmcC}$ be an {\qcat} and $\kappa$ a regular cardinal.  
The following conditions are equivalent:

\begin{enumerate}[label={(\roman*)},itemindent=0em]
  \item The {\qcat} ${\qmcC}$ is $\kappa$-accessible.

  \item The {\qcat} ${\qmcC}$ is locally small and admits 
  $\kappa$-filtered colimits, the full subcategory 
  ${\qmcC}^{\kappa} \subseteq {\qmcC}$ of $\kappa$-compact objects is 
  essentially small, and ${\qmcC}^{\kappa}$ generates ${\qmcC}$ under 
  small $\kappa$-filtered colimits.

  \item The {\qcat} ${\qmcC}$ admits small $\kappa$-filtered 
  colimits and contains an essentially small full subcategory 
  ${\qmcC}'' \subseteq {\qmcC}$ consisting of $\kappa$-compact objects 
  which generates ${\qmcC}$ under small $\kappa$-filtered colimits.
\end{enumerate}
\end{prop}

Presentable {\qcats} are the subject of \cite[\S5.5]{Lurie:HTT}. 

\begin{defin}\label{def-presentable}
  \cite[Definition 5.5.0.1]{Lurie:HTT} An $\infty$-category $\qmcC$ is
  {\bf presentable} if it is accessible and admits small colimits.  We
  denote by $\qPrL$ (see \cite[Definition 5.5.3.1]{Lurie:HTT}) the
  {\qcat} of presentable {\qcats} and colimit preserving functors. The
  {\qcat} of such stable {\qcats} is denoted by $\qPrLst$.
\end{defin}

\begin{defin}\label{defin-compact-objects}
  {\bf Continuous functors, compact objects and and compact fibrations.}
  \cite[Definition 5.3.4.5]{Lurie:HTT} Let \({\qmcC}\) be an
  \(\infty\)-category which admits small \(\kappa\)-filtered colimits.
  A functor \(f : {\qmcC} \to {\qmcD}\) is {\bf \(\kappa\)-continuous}
  if it preserves \(\kappa\)-filtered colimits.

  For an object $C$ in \({\qmcC}\), let
  \(j_C : {\qmcC} \to \mathcal{S}\) denote the functor corepresented
  by \(C\).  If \({\qmcC}\) admits \(\kappa\)-filtered colimits, then
  \(C\) is {\bf \(\kappa\)-compact} if \(j_C\) is
  \(\kappa\)-continuous. We will say that \(C\) is {\bf compact} if it
  is \(\omega\)-compact (and \({\qmcC}\) admits filtered colimits).

A left fibration (\cref{defin-fibrations}) 
\({\qmcC}' \to {\qmcC}\) is {\bf   \(\kappa\)-compact} if it is
classified by a \(\kappa\)-continuous functor
\({\qmcC} \to \mathcal{S}\).
\end{defin}

\begin{defin}\label{defin-accessible-functor}
{\bf Accessible functors.}  \cite[Definition 5.4.2.5]{Lurie:HTT}
If ${\qmcC}$ is an accessible {\qcat}, then a functor
$F : {\qmcC} \to {\qmcC}'$ is {\bf accessible} if it is
$\kappa$-continuous for some regular cardinal $\kappa$ (and therefore
for all regular cardinals $\tau \ge \kappa$).
\end{defin}

\begin{defin}\label{defin-left-exact}
{\bf Left exact functors.}  \cite[Definition 5.3.2.1]{Lurie:HTT}
Let\linebreak $F : {\qmcC} \to {\qmcD}$ be a functor between {\qcats}
and let $\kappa$ be a regular cardinal. We say that $F$ is {\bf
  $\kappa$-left exact} if, for any left fibration
${\qmcD}' \to {\qmcD}$ where ${\qmcD}'$ is $\kappa$-filtered, the
{\qcat}
\[
{\qmcC}' = {\qmcC} \times_{{\qmcD}} {\qmcD}'
\]
is also $\kappa$-filtered. We say that $F$ is {\bf  left exact} if it is 
$\omega$-left exact.
\end{defin}

\subsubsection{The coskeleton functor}\label{sec-coskeleton}
Next we define the coskeleton functor, which figures in the
definitions of hypercoverings (\cref{def-HTT-6.5.3.2}) and
hypersheaves (\cref{defin-hypersheaf}), which are needed for
\cref{def-cyclo-sheaf}.  Before doing so, here is an observation.

For CW-complexes $X$ and $Y$, let $X^{n}$ denote the $n$-skeleton of $X$ and
$P^{n}Y$ the $n$th Postnikov section of $Y$.  Then we know that
\begin{numequation}\label{eq-skel-Post}
\begin{split}
\Map_{}(X^{n}, Y)\simeq \Map_{}(X,P^{n}Y),
\end{split}
\end{numequation}
 
\noindent so we have a pair of adjoint endofunctors in the {\qcat} of
Kan complexes $\qmcS$.

  \begin{defin}\label{def-coskeleton}
    \cite[Notation 6.5.3.1]{Lurie:HTT} For each $n \ge 0$, let
    $\bdelt^{\le n}$ denote the full subcategory of
    $\bdelt$ spanned by the set of objects $\{[0],\dots,[n]\}$, and we
    denote the inclusion functor by $i_{n}$.  

    Similarly let $\bdelt^{\le n}_{+}\subseteq \bdelt _{+}$ for
    $n\geq -1$ denote the full subcategory of $\bdelt$ spanned by the
    set of objects $\{[-1],\dots,[n]\}$.

    For a presentable $\infty$-category ${\qmcC}$ as in 
    \cref{def-presentable}, let
    \begin{align*}
  {\qmcC}_{\bdelt}
      & := \Fun\!\left( N(\bdelt)^{\op},\, {\qmcC} \right),&
  {\qmcC}_{\bdelt_{+}}
      & := \Fun\!\left( N(\bdelt_{+})^{\op},\, {\qmcC} \right),\\
  {\qmcC}_{\bdelt^{\leq n}}
      & := \Fun\!\left( N(\bdelt^{\le n})^{\op},\, {\qmcC} \right),&
\aand            
  {\qmcC}_{\bdelt_{+}^{\leq n}}
      & := \Fun\!\left( N(\bdelt_{+}^{\le n})^{\op},\, {\qmcC} \right).
\end{align*}

Let $\bdelt^{\leq -1} $ be the empty category, making
${\qmcC}_{\bdelt^{\leq -1}}$ the trivial {\qcat} whose single object
is the empty diagram.  We define $\bdelt^{\leq -1}_{+} $ to be the
trivial category with object the empty set, making 
${\qmcC}_{\bdelt^{\leq -1}_{+}}$ the discrete {\qcat} associated with
$\qmcC$.

The restriction functor
\[
  \skel_n:=(N(i_{n})^{\op})^{*} : {\qmcC}_{\bdelt}
  \longrightarrow
   {\qmcC}_{\bdelt^{\leq n}}
\]
has left and right adjoints $i^{*}_{n}$ and $i^{!}_{n}$ given by
left and right Kan extensions along the inclusion
\[
N(i_{n})^{\op }:
  N(\bdelt^{\le n})^{\op} \hookrightarrow N(\bdelt)^{\op},
\]

\noindent and similarly for the augmented case.

(Goerss and Jardine \cite[IV.3.2]{Goerss-Jardine} denote our
$\bdelt ^{\leq n}$ by $\bdelt _{n}$ and call a contravariant
$\Set $-valued functor on functor on it an {\bf $n$-truncated
  simplicial set}.  They call our $\skel_{n}$ the {\bf $n$-truncation
  functor}.  These notions are also discussed 30 years earlier by Mike
Artin and Barry Mazur in \cite[\S1]{Artin-Mazur}.)

For an object $X$ in $\qmcC_{\bdelt }$,
\begin{displaymath}
  (i_n^* X)_m = \colim{[m] \to [k]} X_k
  \qquad \aand
  (i_n^! X)_m = \lim{[k] \to [m]} X_k,
\end{displaymath}
 
\noindent
where the indicated morphisms in both cases are in $\bdelt $ with
$k\leq n$.  For $m\leq n$, each object is $X_{n}$.  For $m>n$, each
map $[k]\to [m]$ factors though $[n]$, so $(i_n^! X)_m$ is a certain
subobject of the $\binom{m+1}{n+1}$-fold product of $X_{n}$.
There are maps
\begin{displaymath}
  (i_n^! X)_m \to X_{n}^{\binom{m+1}{n+1}}
  \qquad \aand
 \coprod _{\binom{m+1}{n+1}}X_{n}\to  (i_n^* X)_m,
\end{displaymath}
 
\noindent which are respectively 
a monomorphism (meaning a $(-1)$-truncated morphism)
 and an effective epimorphism as in \cref{def-effective-epi}. 

The {\bf $n$-coskeleton functor} $\coskel_n
: {\qmcC}_{\bdelt} \to {\qmcC}_{\bdelt}$ is the composition
\begin{displaymath}
\xymatrix
@R=4mm
@C=15mm
{
{\qmcC_{\bdelt}}\ar[r]^(.5){\skel_{n}}
  &{\qmcC_{\bdelt_{\leq n}}}\ar[r]^(.5){i^{!}_{n}}
    &{\qmcC_{\bdelt}.}
}
\end{displaymath}
\end{defin}
In particular we have
\begin{numequation}\label{eq-coskel+1}
\begin{split}
  (\coskel_{n-1}(X)  )_{n}
 & = \left\{
    (x_0,\dots,x_{n}) \in (X_{n-1})^{n+1}
   \;\middle|\;\right.\\
 &\phantom{(X_{n})^{n}}
   \left.
        d_i(x_j) = d_{j-1}(x_i)
        \text{ for all }0\leq  i<j\leq n
   \right\} \\
  &=:M_n(X)\qquad \mbox{the {\bf $n$th matching object} of $X$,} 
\end{split}
\end{numequation}

\noindent where $d_{i}:X_{n-1}\to X_{n-2}$ is the $i$th face map.

\begin{ex}\label{eq-M0}
  In the augmented case for small $n$, we have
  \begin{align*}
  M_{0}(X)&= (\coskel_{-1}(X))_{0}
  =X_{-1}\\
M _{1}(X)   &  = (\coskel_{0}(X))_{1}  
  = \left\{
    (x_{0},x_{1}) \in (X_{0})^{2}
   \;\middle|        d_{0}(x_{1}) = d_{0}(x_{0})   \right\} \\                          & = X_{0}\times _{X_{-1}}X_{0}\\
M _{2}(X)   &  = (\coskel_{1}(X))_{2}\\
            &    = \left\{
    (x_0,  x_{1},x_{2}) \in (X_{1})^{3}
   \;\middle|\;\right.\\
          &            \qquad 
   \left.
      d_{0}(x_1)=d_0(x_0),\, d_0(x_2)=d_1(x_0),\, d_1(x_2)=d_1(x_1)       
   \right\}.
\end{align*}

\end{ex}

We have an adjunction
\begin{numequation}\label{eq-skel-rn}
\begin{split}
\xymatrix
@R=4mm
@C=20mm
{
  {\qmcC_{\bdelt }}
  \ar@<1.2ex>[r]^(.5){\skel_{n}}_(.5){\perp  }
  &{\qmcC_{\bdelt^{\leq n} }}\ar@<1.2ex>[l]^(.5){i^{!}_{n}}
}
\end{split}
\end{numequation}

\noindent with a unit
\begin{numequation}\label{eq-coskel}
\begin{split}
\eta_{X,n}:X_{\bullet}\to i^{!}_{n}\skel_{n}X_{\bullet}=:\coskel_{n}X_{\bullet}
\end{split}
\end{numequation}
 
\noindent for each object $X_{\bullet}$ in $\qmcC_{\bdelt }$.  The
target of $\skel_{-1}$ is the trivial category
$\qmcC_{\bdelt^{\leq -1} }$.  The image of its one object under
$i^{!}_{-1}$ is a constant functor of $\bdelt $ whose value is a terminal
object in $\qmcC$, which exists because $\qmcC$ is presentable as in
\cref{def-presentable}.

The unit map $(\eta_{X,n-1})_{n}:X_{n}\to M_{n}(X)$ of \cref{eq-coskel} sends an
$n$-simplex $K$ in $X_{n}$ to its $(n+1)$-tuple of $(n-1)$-faces.  It
is required to be onto in \cref{def-HTT-6.5.3.2}.

\begin{ex}\label{ex-C=Set}
  {\bf The case $\qmcC=\Set $.}
  For a simplicial set $X$ with geometric realization $|X|$, the
  adjunctions of \cref{eq-skel-Post} and \cref{eq-skel-rn} imply that
  \begin{displaymath}
|\coskel_{n}X|\simeq P^{n}|X|. 
\end{displaymath}

For $n=m-1$, we have
\begin{numequation}\label{eq-coskel+1-set}
\begin{split}
  (\coskel_{n-1}(X)  )_{n}
 & = \Map_{\qmcC_{\bdelt }}(\partial \bdelt^n, X_{\bullet})  
    \simeq   \Omega ^{n-1}X .
\end{split}
\end{numequation}

\end{ex}

\subsubsection{Localization}\label{sec-localization}
Localization functors are discussed in \cite[\S5.2.7]{Lurie:HTT}. 

\begin{defin}\label{defin-lozn}
\cite[Definition 5.2.7.2]{Lurie:HTT} 
A functor $f : {\qmcC} \to {\qmcD}$ between {\qcats} is a
{\bf  localization} if $f$ has a fully faithful right adjoint.
\end{defin}

The following definition specializes to one of Bousfield in the
ordinary category of spaces.

\begin{defin}\label{defin-S-local}
  \cite[Definition 5.5.4.1]{Lurie:HTT} Let ${\qmcC}$ be an
  $\infty$-category and let $S$ be a collection of morphisms of
  ${\qmcC}$.  We say that an object $Z \in {\qmcC}$ is {\bf $S$-local}
  if, for every morphism $s \colon X \to Y$ belonging to $S$,
  composition with $s$ induces a homotopy equivalence.
\[
  \Map_{{\qmcC}}(Y, Z) \;\longrightarrow\; \Map_{{\qmcC}}(X, Z).
\]

A morphism $f \colon X \to Y$ of ${\qmcC}$ is an {\bf $S$-equivalence}
if, for every $S$-local object $Z$, composition with $f$ induces a homotopy
equivalence
\[
  \Map_{{\qmcC}}(Y, Z) \;\longrightarrow\; \Map_{{\qmcC}}(X, Z).
\]
\end{defin}

\begin{defin}\label{defin-top-lozn}
  {\bf Topological localizations.} \cite[Definition 6.2.1.4]{Lurie:HTT}
  Let ${\qmcC}$ be a presentable {\qcat} and let $S$ be a strongly
  saturated class of morphisms of ${\qmcC}$; see \cite[Definition
  5.5.4.5]{Lurie:HTT}. We will say that $S$ is {\bf topological} if
  the following conditions are satisfied:
\begin{enumerate}[label={(\roman*)},itemindent=0em]
  \item There exists a subclass $S_0 \subseteq S$ consisting of monomorphisms
  such that $S_0$ generates $S$ as a strongly saturated class of morphisms.
  \item Given a pullback diagram
  \[
  \begin{tikzcd}
  X' \arrow[r,"f'"] \arrow[d] & X \arrow[d,"f"] \\
  Y' \arrow[r] & Y
  \end{tikzcd}
  \]
  in ${\qmcC}$ for which $f$ belongs to $S$, the morphism $f'$ also
  belongs to $S$.
\end{enumerate}

We will say that a localization $L \colon {\qmcC} \to {\qmcD}$ is
{\bf topological} if the collection $S$ of all morphisms $f \colon X \to Y$
in ${\qmcC}$ such that $Lf$ is an equivalence is topological.
\end{defin}

\subsection{Groupoids and the {\Cech} nerve}\label[appendix]{sec-groupoids-cech}

The material in this subsection is taken from
\cite[\S6.1.2]{Lurie:HTT} unless otherwise indicated.

\begin{defin}\label{def-gpoid-cat}\cite[\S6.1.2]{Lurie:HTT}.
Let $\mathcal{C}$ be an ordinary category which admits finite limits.  
A {\bf groupoid object} of $\mathcal{C}$ is a functor
\[
F \colon \mathcal{C} \longrightarrow \mathbf{Gpd}
\]
from $\mathcal{C}$ to the category $\mathbf{Gpd}$ of small groupoids, which
has the following properties:
\begin{enumerate}[label={(\roman*)},itemindent=0em]
  \item There exists an object $X_{0} \in \mathcal{C}$ and a natural (in $C$)
  identification of $\operatorname{Hom}_{\mathcal{C}}(C, X_{0})$ with the set
  of objects of the groupoid $F(C)$ for each $C \in \mathcal{C}$.

\item There exists an object $X_{1} \in \mathcal{C}$ and a natural
  identification of\linebreak
  $\operatorname{Hom}_{\mathcal{C}}(C, X_{1})$ with the set of
  morphisms in the groupoid $F(C)$ for each $C \in \mathcal{C}$.
\end{enumerate}
\end{defin}

For example, a groupoid object  in  $\Set$ is simply a groupoid.

In Lurie's words,
\begin{quote}
Giving a groupoid object of a category $\mathcal{C}$ is equivalent to giving
a pair of objects $X_{0} \in \mathcal{C}$ (the “object classifier”) and
$X_{1} \in \mathcal{C}$ (the “morphism classifier”), together with a
collection of maps which relate $X_{0}$ to $X_{1}$ and satisfy appropriate
identities, imitating the usual axiomatics of category theory.

These identities can be very efficiently encoded using the formalism of
simplicial objects. For every $n \ge 0$, let $[n]$ denote the category
associated to the linearly ordered set $\{0,\dots,n\}$, and consider the
functor
\[
F_{n} \colon \mathcal{C} \longrightarrow \mathbf{Set}
\]
defined by
\[
F_{n}(C) = \operatorname{Hom}_{\mathbf{Cat}}([n], F(C)).
\]

By assumption, $F_{0}$ and $F_{1}$ are representable by objects
$X_{0}, X_{1} \in \mathcal{C}$. Since $\mathcal{C}$ is stable under finite
limits, it follows that
\[
  F_{n} = F_{1} \times_{F_{0}} \cdots \times_{F_{0}} F_{1}
  \qquad\mbox{with $n$ factors}  \]
is representable by an object
\[
X_{n} = X_{1} \times_{X_{0}} \cdots \times_{X_{0}} X_{1}.
\]
The objects $X_{n}$ assemble into a simplicial object $X_{\bullet}$ of
$\mathcal{C}$. We can think of this construction as a generalization of the
process which associates to every groupoid $D$ its nerve $N(D)$ (a simplicial
set). Moreover, as in the classical case, the association
$F \mapsto X_{\bullet}$ is fully faithful. In other words, we can identify
groupoid objects of $\mathcal{C}$ with the corresponding simplicial objects.

Of course, not every simplicial object $X_{\bullet}$ of $\mathcal{C}$ arises
via this construction. This is true if and only if certain additional
conditions are met; for instance, the diagram
\[
\begin{tikzcd}
X_{2} \arrow[r,"d_{2}"] \arrow[d,"d_{0}"'] &
X_{1} \arrow[d,"d_{0}"] \\
X_{1} \arrow[r,"d_{1}"'] &
X_{0}
\end{tikzcd}
\]
must be Cartesian.
\end{quote}

We now describe a similar construction in {\qcats}.

\begin{defin}\label{def-gpoid-qcat}
  \cite[Definition 6.1.2.7]{Lurie:HTT} A {\bf groupoid object in an
    {\qcat} $\qmcC$} is a simplicial object
  $V_{\bullet}:\mathrm{N}(\bdelt)^{\op}\to \qmcC$ satisfying any of eight {\eqt}
  conditions listed in \cite[Proposition 6.1.2.6]{Lurie:HTT}.
\end{defin}

Of Lurie's conditions,  the last and easiest to verify is the following.
For every $n \ge 0$ and every partition 
\([n] = S \cup S'\) such that \(S \cap S'\)
consists of a single element \(s\), the diagram
\[
\begin{tikzcd}
V_{n}=V([n]) \arrow[r] \arrow[d] & V(S)=V_{|S|-1} \arrow[d] \\
V_{|S'|-1}=V(S') \arrow[r] & V(\{s\})=V_{0}
\end{tikzcd}
\]
is a pullback square in $\qmcC$, where the morphisms are induced by
the evident inclusions of subsets of $[n]$.

Let $\bdelt^{\leq n}\subseteq \bdelt$ and
$\bdelt_{+}^{\leq n}\subseteq \bdelt_{+}$ denote the full
subcategories of \cref{def-coskeleton}.  Hence 
$\bdelt_{+}^{\leq 0}\subseteq \bdelt_{+}$
and $\bdelt_{+}^{\leq 1}\subseteq \bdelt_{+}$ are the categories
\begin{displaymath}
\xymatrix
@R=4mm
@C=10mm
{
{[-1]}\ar[r]
  &{[0]}
}
\qquad \aand
\xymatrix
@R=4mm
@C=10mm
{
{[-1]}\ar[r]^(.5){}
&{[0]}\ar@<-1ex>[r]^(.5){}
      \ar@<1ex>[r]^(.5){}
      &{[1].}\ar[l]_(.5){}
}
\end{displaymath}

\noindent 
$\mathrm{N}(\bdelt_{+}^{\le 0})^{\op}$ is the simplicial 1-simplex
$\bdelt^{1}$ of \cite[Definition 2.21]{Rav:gjmcyc},
which is self-dual.  This means a $\qmcC$-valued functor
$V$ on it is {\eqt } to a morphism $u:V_{0}\to V_{-1}$ in $\qmcC$.

\begin{prop}\label{prop-6.1.2.11}
  \cite[Proposition 6.1.2.11]{Lurie:HTT}
  Let $\qmcC$ be an $\infty$-category and let
\[
  V_{\bullet}
  \colon \mathrm{N}(\bdelt_{+})^{\op} \to \qmcC
\]
be an augmented simplicial object of $\qmcC$ as in \cref{def-6.1.2.2}.

The following conditions are equivalent:
\begin{enumerate}[label={(\roman*)},itemindent=0em]
\item The augmented simplicial object $V_{\bullet}$ is a right Kan
  extension of its restriction to
  $\mathrm{N}(\bdelt_{+}^{\le 0})^{\op}$.  The latter is determined by
  the morphism $u:V_{0}\to V_{-1}$.
    
  \item The underlying simplicial object $V_\bullet$ is a groupoid object of
  $\qmcC$ as in \cref{def-gpoid-qcat}, and the diagram
  \[
    V_{\bullet}\vert \mathrm{N}(\bdelt_{+}^{\le 1}) ^{\op}
  \]
  has a subdiagram
   \[
  \begin{tikzcd}
    V_1 \arrow[r,"d_{1}"] \arrow[d,"d_{0}"'] & V_0 \arrow[d, "u"] \\
    V_0 \arrow[r,"u"] & V_{-1}
  \end{tikzcd}
   \]
   which is a pullback square in $\qmcC$, where the $d_{i}$s are the
   face maps of \cite[Definition 2.19]{Rav:gjmcyc}.
\end{enumerate}
\end{prop}

This means that in the homotopy category $h\qmcC$ we have a groupoid
object as in \cref{def-gpoid-cat} in which $X_{n}=V_{n}$ for each
$n\geq 0$.

\begin{defin}\label{def-Cech-nerve}
An augmented simplicial object
\[
V_{\bullet} \in \qmcC_{\bdelt _{+}}
\]
for an  {\qcat} $\qmcC$ is a {\bf { \Cech} nerve} if it satisfies the
equivalent conditions of \cref{prop-6.1.2.11}.  In this case,
$V_{\bullet}$ is determined up to equivalence by the map
\[
u \colon V_{0} \to V_{-1},
\]
and we will also say that $V_{\bullet}$ is the {\bf  {\Cech} nerve of
$u$}.
\end{defin}

\begin{ex}\label{ex-classic-Cech}
{\bf The classical {\Cech} nerve}. \cite{Alexandroff1927}
  Let 
 \begin{displaymath}
\msU    = \left\{U_{\alpha}\subseteq X \right\} ,
\end{displaymath}

\noindent 
be an open covering of a space $X$.  Let $V_{\bullet}$  be the simplicial
space with
\begin{align*}
  V_{n}
  &:=\coprod _{\alpha_{0},\dots,\alpha_{n}}
    U_{\alpha_{0}}\cap\cdots \cap U_{\alpha_{n}}\\
  & \phantom{:}= V_{0}\times _{X}V_{0}\times_{X} \cdots \times_{X}V_{0}
    \qquad \mbox{with $n+1$ factors.} 
\end{align*}

\noindent The nondegenerate subspace $V'_{n}\subseteq V_{n}$ for $n>0$
is the disjoint union of $(n+1)$-fold intersections in which no two
adjacent indices $\alpha_{i}$ are the same.

The relevant map $u$ for \cref{def-Cech-nerve} is
\begin{align*}
  V_{0}
  &:=\coprod _{\alpha}U_{\alpha}\to X = :V_{-1} .
\end{align*}
\end{ex}

\subsection{Sheaves and presheaves}\label[appendix]{sec-sheaves-presheaves}

Recall that a presheaf $\mcalF$ on an ordinary category $\mcC $ is a
contravariant functor with values in the category of sets or some
variant of it.  It is a sheaf if it converts certain colimits in the domain
category to limits (or homotopy limits) in the codomain.

  \begin{ex}\label{ex-poset-open}
{\bf Sheaves on the poset category of open subsets.}
Suppose
the domain of the presheaf $\mcalF$ is the poset category $\mcU(X) $ of
open subsets of a topological space $X$. Then the colimit ({\ie} pushout)
of the diagram
\begin{displaymath}
U_{i} \leftarrow U_{i}\cap U_{j} \rightarrow U_{j}
\end{displaymath}
 
\noindent is $U_{i} \cup U_{j}$, and the sheaf condition is that
${\mcalF}(U_{i} \cup U_{j})$ is the limit (pullback) of
\begin{displaymath}
  {\mcalF}(U_{i}) \rightarrow {\mcalF}(U_{i}\cap U_{j})
                  \leftarrow {\mcalF}(U_{j}) .
\end{displaymath}

{\Eqvt}ly ${\mcalF}(X)$ is the equalizer of
\[
{\mcalF}(X) \longrightarrow \prod_{\alpha} {\mcalF}(U_{\alpha})
\;\rightrightarrows\;
\prod_{\alpha,\beta} {\mcalF}(U_{\alpha} \cap U_{\beta})
\]
whenever \(\{U_{\alpha}\}\) is an open cover of \(X\). ${\mcalF}(X)$
is also the limit of the cosimplicial diagram ${\mcalF}(V_{\bullet})$
for $V_{\bullet}$ as in \cref{ex-classic-Cech}. 
\end{ex}

We will look at this again in \cref{ex-revisited}.

The following is the subject of  \cite[\S5.1]{Lurie:HTT}. 

\begin{defin}\label{defin-presheaves}
  For a small {\qcat} $\qmcC$ the {\bf {\qcat} of presheaves}
  $\mcP (\qmcC)$ is that of space valued contravariant functors
  ${\qmcC}^{\op}\to {\qmcS}$.
\end{defin}

\begin{defin}\label{defin-Yoneda-embedding}
  For a small \(\infty\)-category \({\qmcC}\), the {\bf Yoneda
    embedding} is the functor
\begin{displaymath}
\yo:{\qmcC}\to \mcP(\qmcC)   
\end{displaymath}
 
\noindent that sends an object $C$ to the functor
$\Map_{{\qmcC}}(-, C)$.  We call $\yo(C)$ the {\bf Yoneda presheaf} of
the object $C$.
\end{defin}

The symbol $\yo$ above is the Japanese hiragana character ``yo,'' the
first syllable of Yoneda’s name.

The {\qcatal} definition of a sheaf is more complicated than the one
in ordinary category theory.  Before giving it in
\cref{defin-sheaves}, we must introduce $\infty$-topoi in
\cref{sec-qtops} and Grothendieck topologies in \cref{sec-Groth-top}.

\subsection{$\infty$-topoi}\label[appendix]{sec-qtops}
In classical category theory a topos is a category that ``looks like''
the category of sets.  A good reference is the book
\cite{MacLane-Moerdijk} by Moerdijk and Saunders Mac Lane. (You may
notice that we are spelling Mac Lane's name in two different ways.
This is because he did so himself in \cite{SESM2} and
\cite{MacLane-Moerdijk}.)  See also the discussion in
\cite[\S6.1]{Lurie:HTT}.

For any category $\mcC$ one has the presheaf category $\mcP (\mcC)$ of
contravariant $\Set $-valued functors on $\mcC  $.  It is a short step
away from being a topos, and every topos is derived from a presheaf
category by a certain construction.  Such constructions are related to
Grothendieck topologies; see \cref{def-Grothendieck-topology}.  

In \cite[page~xii]{Lurie:HTT} Lurie says
\begin{quote} 
  Roughly speaking, an $\infty$-topos is an {\qcat} which ``looks
  like'' the $\infty$-category of [spaces $\qmcS$]. We will show that
  this intuition is justified in the sense that it is possible to
  reconstruct a large portion of classical homotopy theory inside an
  arbitrary $\infty$-topos. In other words, an $\infty$-topos is a
  world in which one can ``do'' homotopy theory (much as an ordinary
  topos can be regarded as a world in which one can ``do'' other types
  of mathematics).
\end{quote}

\begin{defin}\label{defin-infinity-topos}
  \cite[Definition 6.1.0.4]{Lurie:HTT} An {\qcat} ${{\qmcX}}$ is an
  {\bf $\infty$-topos} if there exists a small {\qcat} ${\qmcC}$ with
  $\qmcS$-valued presheaf category $\mcP(\qmcC)$ as in
  \cref{defin-presheaves}, and an accessible left exact localization
  functor (see
  \cref{defin-accessible-functor,defin-left-exact,defin-lozn})
  $\mcP ({\qmcC}) \to {{\qmcX}}$.

  It is {\bf quasi-compact} \cite[Definition A.2.0.12]{Lurie:SAG} if every
covering of ${\qmcX}$ has a finite subcovering: that is, for every
effective epimorphism $\coprod_{i \in I} U_i \to 1$ in ${\qmcX}$
(where $1$ is the final object of ${\qmcX}$), there exists a finite
subset $I_0 \subseteq I$ such that the map
\[
\coprod_{i \in I_0} U_i \longrightarrow 1
\]
is also an effective epimorphism. We say that an object
$X \in {\qmcX}$ is quasi-compact if the $\infty$-topos
${\qmcX}_{/X}$ is quasi-compact.
\end{defin}

In Lurie's words,
\begin{quote}
the class of \(\infty\)-topoi is defined to be the smallest collection of
\(\infty\)-categories which contains ${\qmcS}$ and is stable under
certain constructions (left exact localizations and the formation of
functor categories).
\end{quote}

The terminology in the following is explained in \cite[\S6.1]{Lurie:HTT}.

\begin{theorem}\label{theorem-topos-property}
\cite[Theorem 6.1.0.6]{Lurie:HTT} 
Let \({\qmcX}\) be an \(\infty\)-category. The following conditions are
equivalent:
\begin{enumerate}
  \item[(i)] The \(\infty\)-category \({\qmcX}\) is an \(\infty\)-topos.
  \item[(ii)] The \(\infty\)-category \({\qmcX}\) is presentable as in
    \cref{def-presentable}, and for every small simplicial set \(K\)
    and every natural transformation \(\alpha : p \to q\) of
    diagrams \(p, q : K^{\triangleright} \to {\qmcX}\), the
    following condition is satisfied: If \(q\) is a colimit diagram
    and \(\alpha|_{K}\) is a Cartesian transformation, then \(p\) is a
    colimit diagram if and only if \(\alpha\) is a Cartesian
    transformation.
  
  \item[(iii)] The \(\infty\)-category \({\qmcX}\) satisfies the following
  \(\infty\)-categorical analogues of Giraud's axioms:
  \begin{enumerate}
    \item \({\qmcX}\) is presentable.
    \item Colimits in \({\qmcX}\) are universal.
    \item Coproducts in \({\qmcX}\) are disjoint.
    \item Every groupoid object of \({\qmcX}\) is effective.
  \end{enumerate}
\end{enumerate}
\end{theorem}

Our next job is to define effective epimorphisms in an $\infty$-topos,
which we will do in \cref{def-effective-epi}.  The following
discussion from \cite[\S6.1.1]{Lurie:HTT} about effective epimorphisms
in an ordinary category is helpful.

\begin{quote}
Recall that if $X$ is an object in an (ordinary) category $\mcC$, 
then an {\bf equivalence relation} on $X$ is an object $R$ of $\mcC$ 
equipped with a map 
\[
p \colon R \longrightarrow X \times X
\]
such that for any object $S$, the induced map
\[
\operatorname{Hom}_{\mcC}(S,R) \longrightarrow 
\operatorname{Hom}_{\mcC}(S,X) \times \operatorname{Hom}_{\mcC}(S,X)
\]
exhibits $\operatorname{Hom}_{\mcC}(S,R)$ as an equivalence relation 
on $\operatorname{Hom}_{\mcC}(S,X)$.

If $\mcC$ admits finite limits, then it is easy to construct 
equivalence relations in $\mcC$: given any map 
$f \colon X \to Y$ in $\mcC$, the induced map
\[
X \times_Y X \longrightarrow X \times X
\]
is an equivalence relation on $X$.

If the category $\mcC$ admits finite colimits, then one can attempt 
to invert this process: given an equivalence relation $R$ on $X$, one can 
form the coequalizer of the two projections $R \rightrightarrows X$ to obtain 
an object which we denote by $X/R$. In the category of sets, one can recover 
$R$ as the fiber product $X \times_{X/R} X$. In general, this need not occur: 
one always has 
\[
R \subseteq X \times_{X/R} X,
\]
but the inclusion may be strict (as subobjects of $X \times X$). If equality 
holds, then $R$ is said to be an {\bf effective equivalence relation}, and 
the map $X \to X/R$ is said to be an {\bf effective epimorphism}.

\end{quote}

\begin{defin}\label{def-simp-res}
A {\bf simplicial resolution in an $\infty$-topos $\qmcX$} is an
augmented simplicial object (as in \cref{def-6.1.2.2})
\[
U^{+}_{\bullet} : N(\bdelt_{+})^{op} \longrightarrow \qmcX
\]
in which $U^{+}_{-1}$ is  a colimit of its underlying simplicial object 
\[
U_{\bullet} = U^{+}_{\bullet}\big|_{N(\bdelt)^{op}}.
\]

\end{defin}

\begin{ex}\label{ex-Cech-surjection}
  For a surjection of sets $u:V_{0}\to V_{-1}$, in
  the {\Cech} nerve $V_{\bullet}$ (see \cref{ex-classic-Cech}) we have
  \begin{align*}
V_{n}
    & = V_{0}\times _{V_{-1}}\times \cdots  \times _{V_{-1}} V_{0}
    \qquad\mbox{with $n+1$ factors}  \\
    & =\left\{(v_{0},\dots,v_{n})\in V_{0}^{n+1}:
      u(v_{0})=\cdots = u(v_{n})   \right\} \\
    & = \coprod_{x\in V_{-1} }(u^{-1}(x) )^{n+1}.
  \end{align*}

  \noindent This is a simplicial resolution of $V_{-1}$.
\end{ex}

This should not be confused with the simplicial object of
\cref{def-gpoid-qcat}, which is a groupoid object and is controlled by
two different maps $V_{1}\to V_{0}$.

\begin{prop}\label{prop-HTT-6.2.3.5}
  \cite[Corollary 6.2.3.5]{Lurie:HTT} Let $f : U \to X$ be a morphism
  in an $\infty$-topos $\qmcX$.  The following conditions are
  equivalent:
\begin{enumerate}[label={(\roman*)},itemindent=0em]
\item \label{prop-HTT-6.2.3.5i}
  Viewing $f$ as an object of the $\infty$-category $\qmcX_{/X}$,
  the truncation $\tau_{\le -1}(f)$ is a final object of $\qmcX_{/X}$.
\item \label{prop-HTT-6.2.3.5ii}
  The Čech nerve $\check{C}(f)$ of \cref{def-Cech-nerve} is a
    simplicial resolution of $X$.
\end{enumerate}
\end{prop}

\begin{defin}\label{def-effective-epi}
  An {\bf effective epimorphism} is a morphism in an $\infty$-topos
  satisfying the conditions of \cref{prop-HTT-6.2.3.5}.  An {\bf
    effective monomorphism} is one that is $(-1)$-truncated as in
\cref{def-truncated morphism}.
\end{defin}

Note that any morphism to a final object satisfies these conditions and
is therefore an effective epimorphism.

In the category $\Set $, which is the simplest example of an
$\infty$-topos, every morphism $f:U\to X$ is 0-truncated as in
\cref{def-truncated morphism}. Its $(-1)$-truncation, in which the
preimage of each element of $X$ is replaced by a singleton if it is
nonempty, is the inclusion of its image into $X$.  The map $f$ is onto
iff this is the identity map and a final object in $\Set_{/X} $, hence
condition \cref{prop-HTT-6.2.3.5i} above.  The identification of
{\Cech} nerve of such an $f$ as a simplicial resolution is the subject of
\cref{ex-Cech-surjection}, hence condition \cref{prop-HTT-6.2.3.5ii}.

\subsection{Grothendieck topologies and
  $\infty$-sheaves}\label[appendix]{sec-Groth-top}

\begin{defin}\label{defin-sieves}
{\bf Sieves.}
  \cite[Definition 6.2.2.1]{Lurie:HTT}  A
  {\bf sieve on an {\qcat} ${\qmcC}$} is a full subcategory
  ${\qmcC}^{(0)} \subseteq {\qmcC}$ having the property that if\linebreak
  $f : C \to D$ is a morphism in ${\qmcC}$ and $D$ belongs
  to ${\qmcC}^{(0)}$, then $C$ also belongs to ${\qmcC}^{(0)}$. In
  other words the sub-{\qcat} ${\qmcC}^{(0)} $ is closed under
  precomposition with morphisms in $\qmcC$.

  A {\bf sieve on  an object $C \in {\qmcC}$} is a
  sieve on the {\qcat} ${\qmcC}_{/C}$ of morphisms to $C$. Given a
  morphism $f : D \to C$ and a sieve ${\qmcC}^{(0)}_{/C}$ on $C$, we
  let $f^{*}{\qmcC}^{(0)}_{/C}$ denote the unique sieve on $D$ such
  that $f^{*}{\qmcC}^{(0)}_{/C} \subseteq {\qmcC}_{/D}$ and
  ${\qmcC}^{(0)}_{/C}$ determine the same sieve on ${\qmcC}_{/f}$.
\end{defin}

Hence a sieve ${\qmcC}^{(0)}$ is the full subcategory spanned by some
collection of objects along with all objects mapping to them.  A sieve
${\qmcC}^{(0)}_{/C}$ on an object $C$ consists of a collection of
morphisms to $C$ (which need not include the identity) along with all
morphisms to their domains.

Observe that if $F: {\qmcC} \to {\qmcD}$ is a functor between 
$\infty$-categories and ${\qmcD}^{(0)} \subseteq {\qmcD}$ is a sieve on 
${\qmcD}$, then 
\[
  F^{-1}{\qmcD}^{(0)} = {\qmcC} \times_{\qmcD} {\qmcD}^{(0)}
\]
is a sieve on ${\qmcC}$. Moreover, if $F$ is an equivalence, then 
$F^{-1}$ induces a bijection between sieves on ${\qmcD}$ and sieves on 
${\qmcC}$.

\begin{defin}\label{def-Grothendieck-topology}
  \cite[more of Definition
  6.2.2.1]{Lurie:HTT} A {\bf Grothendieck topology $\tau$} on an {\qcat}
  ${\qmcC}$ consists of a specification, for each object $C$ of
  ${\qmcC}$, of a collection of sieves on $C$ which we refer to as
  {\bf covering sieves}. These collections are required to satisfy the
  following properties:

\begin{enumerate}
\item [(i)] For every object $C$ of ${\qmcC}$, the sieve
 $ {\qmcC}_{/C}$
  is a covering sieve.

  \item [(ii)]If $f : C \to D$ is a morphism in ${\qmcC}$ and 
  ${\qmcC}^{(0)}_{/D}$ is a covering sieve on $D$, then 
  $f^{*}{\qmcC}^{(0)}_{/D}$ is a covering sieve on $C$.

  \item [(iii)] Let $C$ be an object of ${\qmcC}$, let ${\qmcC}^{(0)}_{/C}$ be a 
  covering sieve on $C$, and let ${\qmcC}^{(1)}_{/C}$ be an arbitrary sieve 
  on $C$. Suppose that for each morphism $f : D \to C$ belonging to 
  ${\qmcC}^{(0)}_{/C}$, the pullback $f^{*}{\qmcC}^{(1)}_{/C}$ is a 
  covering sieve on $D$. Then ${\qmcC}^{(1)}_{/C}$ is a covering sieve on 
  $C$.
\end{enumerate}

Such a topology is {\bf  finitary} \cite[Definition A.3.1.1]{Lurie:SAG} if
for every object $C \in {\qmcC}$ and every covering sieve
${\qmcC}^{(0)}_{/C}$ on $C$, there exists a finite collection of
morphisms
\[
\{\, C_i \to C \mid 1 \le i \le n \,\} \subseteq {\qmcC}^{(0)}_{/C}
\]
which generate a covering sieve of $C$ (in other words, the smallest sieve
${\qmcC}^{(1)}_{/C}$ containing each $C_i \to C$ is also a covering
sieve on $C$).

An {\bf $\infty$-site} $({\qmcC}, \tau)$ is an {\qcat} equipped with a
Grothendieck topology.
\end{defin}

\begin{remark}\label{remark-site}
  One can define a Grothendieck topology $\tau$ on an ordinary
  category $\mcC $ in the same way.  Such a pair $(\mcC, \tau )$ is a
  called a {\bf site}.
\end{remark}

\begin{ex}\label{ex-6.2.2.2}
Any {\qcat} ${\qmcC}$ may be equipped with the {\bf  trivial Grothendieck
topology} in which a sieve ${\qmcC}^{(0)}_{/C}$ on an object $C$ of 
${\qmcC}$ is covering if and only if 
\[
{\qmcC}^{(0)}_{/C} = {\qmcC}_{/C}.
\]
In other words, the covering sieve for each object $C$ consists of {\em  all}
morphisms to it.
\end{ex}

A presheaf $\mcalF  \in \mcP(\qmcC) $ defines a collection of objects on
which it is nonempty, and we denote the corresponding sieve
by ${\qmcC}^{(0)}({\mcalF})$.  Conversely, given a sieve
\({\qmcC}^{(0)} \subseteq {\qmcC}\), there is a unique map of simplicial sets
\(f : {\qmcC} \to \Delta^{1}\) such that \({\qmcC}^{(0)}\) is the
preimage of \(\{0\}\).  This construction determines a bijection
between sieves on \({\qmcC}\) and functors
\(f : {\qmcC} \to \Delta^{1}\), and we may identify \(\Delta^{1}\)
with the full subcategory of \(\qmcS^{op}\) spanned by the
objects $\emptyset $ and $\Delta^{0}$ in the category of Kan
complexes.  Since every \((-1)\)-truncated Kan complex is equivalent
to either \(\emptyset\) or \(\Delta^{0}\), we conclude:

\begin{lem}\label{lem-6.2.2.4}
  \cite[Lemma 6.2.2.4]{Lurie:HTT}
For every small \(\infty\)-category \({\qmcC}\), the construction
\[
{\mcalF} \longmapsto {\qmcC}^{(0)}({\mcalF})
\]
determines a bijection between the set of equivalence classes of
\((-1)\)-truncated objects (\cref{defin-truncated-object}) of
\(\mathcal{P}({\qmcC})\) and the set of all sieves on \({\qmcC}\).
\end{lem}

\begin{remark}\label{remark-minus1-truncation}
 The $(-1)$-truncation of a presheaf
  $\mcalF \in \mcP({\qmcC})$ is defined by
  \begin{displaymath}
    (\tau_{\leq -1}\mcalF)    (C)=\mycases{
 \emptyset     
       &\mbox{for }\mcalF (C):= \emptyset\\
       \Delta^{0}
       &\mbox{otherwise.}
}
\end{displaymath}
\end{remark}

Following \cite[\S6.2.2]{Lurie:HTT}, we now introduce a relative
version of the above construction.

Let \(C \in {\qmcC}\) be an object and let \(i : {\mcalF}_{C} \to \yo(C)\)
be a monomorphism in \(\mathcal{P}({\qmcC})\), meaning a natural
transformation of presheaves inducing a monomorphism ({\ie}
$(-1)$-truncated map) of Kan complexes
${\mcalF}_{C}(D)\to\Map_{{\qmcC}}(D, C)$ for each object $D \in {\qmcC}$.
We will replace the the set of \((-1)\)-truncated presheaves in
\cref{lem-6.2.2.4} by the set of such monomorphisms. This means we
need a replacement for the set of sieves on the right.

Let \({\qmcC}_{/C}({\mcalF}_{C})\) denote the full
subcategory of \({\qmcC}_{/C}\) spanned by those objects
\(f : D \to C\) of \({\qmcC}_{/C}\) such that there exists a
commutative triangle
\[
\begin{tikzcd}
& {\mcalF}_{C} \arrow[dr,"i"] &\\
  \yo(D) \arrow[rr,"\yo(f)"'] \arrow[ur] & & \yo(C) 
\end{tikzcd}
\]
in \(\mathcal{P}({\qmcC})\).  Applying the three functors in the
diagram above to an object $X$ in ${\qmcC}$ gives a diagram of spaces
\begin{displaymath}
\xymatrix
@R=7mm
@C=4mm
{
  &{\mcalF_{C}(X) }\ar@{^{(}->}[dr]^(.5){i_{X}}\\
{\Map_{\qmcC}(X,D)}  \ar[rr]_(.5){f_{*}}
    \ar@{-->}[ur]^(.5){}
  &{}
    &{\Map_{\qmcC}(X,C).}
}
\end{displaymath}
 
\noindent The morphism $f$ must be chosen to that the indicated
lifting exists and is natural in $X$.  As usual, both diagrams need
commute only up to homotopy.

It is easy to see that \({\qmcC}_{/C}({\mcalF}_{C})\) is a sieve on
$C$ and that equivalent subobjects of \(\yo(C)\) lead to the same
sieve.

\begin{prop}\label{prop-6.2.2.5}
  \cite[Proposition 6.2.2.5]{Lurie:HTT}
Let \({\qmcC}\) be a small \(\infty\)-category containing an object \(C\).
The construction described above yields a bijection
\[
(i : {\mcalF}_{C} \to \yo(C)) \longmapsto {\qmcC}_{/C}({\mcalF}_{C})
\]
from the set of monomorphisms to $\yo(C)$ to the set of sieves on $C$.
\end{prop}

\begin{defin}\label{defin-sheaves}
  {\bf {\qsheaves}.}
  \cite[Definition 6.2.2.6]{Lurie:HTT} 
Let $({\qmcC}, \tau)$ be an $\infty$-site and let $S$ be the collection
of all monomorphisms $i:{\mcalF}_{C} \to \yo(C)$ which correspond (as in
\cref{prop-6.2.2.5}) to covering sieves
\begin{displaymath}
{\qmcC}^{(0)}_{/C} \subseteq {\qmcC}_{/C}
\end{displaymath}
 
\noindent for all objects $C$ in ${\qmcC}$.  An object
$\mcalF \in \mathcal{P}({\qmcC})$ (\cref{defin-presheaves}) is an {\bf
  {\qsheaf}} or {\bf  sheaf} if it is $S$-local.  We let
$\operatorname{Shv}({\qmcC}, \tau)$ denote the full subcategory of
$\mathcal{P}({\qmcC})$ spanned by the $S$-local objects.  We will
often omit $\tau$ from the notation.

For a presentable category $\qmcD$ (\cref{def-presentable}) we let
\[
  \Shv(\qmcC; \qmcD) := \Shv(\qmcC) \otimes \qmcD
  \in \qPrL
\]
denote the category of $\qmcD$-valued sheaves on $\qmcC$. This can
alternatively be defined as the full subcategory of presheaves that satisfy
descent, namely the sheaf condition (see \cite[Remark~1.3.1.6]{Lurie:SAG}).
\end{defin}

Lurie omits the $\tau$ from his notation for this subcategory.

$\mcalF$ is $S$-local if for each monomorphism
$i:{\mcalF}_{C} \to \yo(C)$ corresponding to a covering sieve, the map
\begin{displaymath}
i^{*}:\Map_{\mathcal{P}({\qmcC})}(\yo(C), \mcalF)\to
\Map_{\mathcal{P}({\qmcC})}(\mcalF_{C}, \mcalF)
\end{displaymath}
 
\noindent is an equivalence.

\begin{ex}\label{ex-revisited}
  {\bf Sheaves on the poset category of  a topological space revisited.}
  We will show how this definition plays out in
  \cref{ex-poset-open}.  Let ${\mcC}$ be the ordinary category
  $\mcU(X) $. We give it the Grothendieck topology in which the
  covering sieves on $U$ are those sieves
  $\left\{U_{\alpha}\subseteq U \right\} $ for which
  $U=\bigcup_{\alpha}U_{\alpha} $. 

  The Yoneda presheaf $\yo(U)$ of \cref{defin-Yoneda-embedding}
  assigns to each open subset $U'$ its set of embeddings into $U$.
  This set is a singleton if $U'\subseteq U$ and the empty set otherwise,
  so $\yo(U)$ is $(-1)$-truncated.

  There is a presheaf monomorphism $i:\mcalF_{U}\to \yo(U) $ if
  $\mcalF_{U}$ is supported by, and is a singleton on, a collection of
  open subsets $U_{\alpha} \subseteq U$ that is closed under
  inclusion. Hence each such presheaf $\mcalF_{U}$ is also
  $(-1)$-truncated.  $\mcalF_{U}$ corresponds as in
  \cref{prop-6.2.2.5} to a covering sieve if the union (colimit) of the
  $U_{\alpha}$ is $U$ itself.  This means that
  \begin{displaymath}
\mcalF_{U} = \colim{\alpha} \yo(U_{\alpha}).
\end{displaymath}
 
Let $S$ be the set of all such inclusions $\mcalF_{U}\to \yo(U)$ for all
open subsets $U$.  {\em What does it mean for a presheaf $\mcalF$ to
  be $S$-local?}  It means that for each such $i$, the map
\begin{displaymath}
i^{*}:\Map_{\mathcal{P}(\mcU(X) )}(\yo(U), \mcalF)\to
\Map_{\mathcal{P}(\mcU(X) )}(\mcalF_{U}, \mcalF)
\end{displaymath}

\noindent is an equivalence, meaning a bijection of sets.  The Yoneda
lemma identifies the domain with $\mcalF(U)$, since a morphism in the
presheaf category is a natural transformation of $\Set$-valued
functors on $\mcU(X)$.  Hence the set on the right must be the same.
Thus we must have
\begin{align*}  
  \mcalF(\colim{\alpha}U_{\alpha})
  & = \mcalF(U) = \Map_{\mathcal{P}(\mcU(X))}(\mcalF_{U}, \mcalF)\\
  & = \Map_{\mathcal{P}(\mcU(X))}(\colim{\alpha} \yo(U_{\alpha})
        , \mcalF)\\
  & = \lim{\alpha}\Map_{\mathcal{P}(\mcU(X))}( \yo(U_{\alpha})
    , \mcalF)\\
  & = \lim{\alpha} \mcalF(U_{\alpha})
       \qquad \mbox{also by the Yoneda lemma.}
\end{align*}

\noindent Thus being $S$-local means the presheaf $\mcalF $ converts
colimits of open subsets (such as pushouts) to limits. {\em This is
  the classical sheaf condition.}
\end{ex}

\begin{prop}\label{prop-top-loc-groth-top}
{\bf Topological localizations and Grothendieck
topologies.} \cite[Proposition 6.2.2.17]{Lurie:HTT} 
Let \({\qmcC}\) be a small \(\infty\)-category. Then Grothendieck
topologies on \({\qmcC}\) are in bijective correspondence with
equivalence classes of topological localizations
(\cref{defin-top-lozn}) of the presheaf \(\infty\)-category
\(\mathcal{P}({\qmcC})\).
\end{prop}

The following takes Lurie four pages to prove.

\begin{prop}\label{prop-source-topoi} 
  {\bf A source of $\infty$-topoi.}  \cite[Proposition 6.2.2.7]{Lurie:HTT}
  For an $\infty$-site $({\qmcC}, \tau)$,
  $\operatorname{Shv}({\qmcC}, \tau)$ is a topological localization
  (\cref{defin-top-lozn}) of $\mathcal{P}({\qmcC})$. In particular,
  $\operatorname{Shv}({\qmcC}, \tau)$ is an $\infty$-topos.
\end{prop}

\begin{defin}\cite[Definition 6.3.1.1]{Lurie:HTT} 
Let ${\qmcX}$ and ${\qmcY}$ be {\qtops}.
A {\bf geometric morphism} from ${\qmcX}$ to ${\qmcY}$ is a functor
\[
  F_{*} : {\qmcX} \to {\qmcY}
\]
which admits a left adjoint $F^{*}$ that is left exact as in
\cref{defin-left-exact}.
\end{defin}

Such functors are studied in \cite[\S6.3.1]{Lurie:HTT}, where they are
typically denoted by $f_{*}$ and $f^{*}$.

\begin{prop}\label{prop-pushforward-leftexact}\cite[Proposition 6.3.1.9]{Lurie:HTT}.
Let
\[
  F_* \colon {\qmcX} \to {\qmcY}
\]
be a geometric morphism between $\infty$-topoi having a left adjoint
\[
  F^* \colon {\qmcY} \to {\qmcX}.
\]
Then $F^*$ and $F_*$ carry $m$-truncated objects
(\cref{defin-truncated-object}) to $m$-truncated objects and
$m$-truncated morphisms to $m$-truncated morphisms, for any integer
$m \geq -2$. Moreover, there is a canonical  equivalence of functors
\[
  \xymatrix
@R=7mm
@C=10mm
{
{\qmcY }\ar[r]^(.5){F^{*}}
        \ar[d]_(.5){\tau_{\leq m}}
  &{\qmcX}\ar[d]^(.5){\tau_{\leq m}}\\
 {\tau_{\leq m}\qmcY }\ar[r]^(.5){F^{*}}
   &{\tau_{\leq m}{\qmcX}}
}
\]
where $\tau_{\leq m}$ is the truncation functor of
\cref{defin-truncated-object}.
\end{prop}

\subsection{All things hyper: hypercompletion, hypercoverings,
  hyperdescent and hypersheaves }\label[appendix]{sec-hypercompletion}

\begin{defin}\cite[\S6.5]{Lurie:HTT} 
  \label{defin-hypercomplete-topos}
  An $\infty$-topos ${\qmcX}$ (\cref{defin-infinity-topos}) is {\bf
    hypercomplete} if every $\infty$-connective morphism
  (\cref{defin-truncated-object}) in it is an equivalence.
\end{defin}

Recall that a morphism $f:X\to Y$ in ${\qmcX}$ is $\infty$-connective
iff for each object $W$, the induced map
\begin{displaymath}
f_{*}:\Map_{\qmcX} (W, X) \to \Map_{\qmcX} (W, Y)
\end{displaymath}
 
\noindent
is a weak equivalence of Kan complexes.  Thus each
such map has an inverse, but in general it need not be induced by a
map from $Y$ to $X$.  Hence the condition above is nontrivial.

\begin{defin}\cite[\S6.5]{Lurie:HTT}
  \label{defin-hypercomplete}
  Let $S$ be the set of $\infty$-connective morphisms in an {\qtop}
  ${\qmcX}$.  An object $Z$ is {\bf hypercomplete} if it is {\bf
    $S$-local} as in \cref{defin-S-local}.  The {\bf hypercompletion}
  ${\whqmcX} \subseteq {\qmcX}$ (which Lurie denotes by
  $\mcX^{\wedge}$) is the full subcategory of such objects.  A
  morphism $f:X\to Y$ is {\bf hypercomplete} if it is hypercomplete as
  in object in the {\qtop} ${\qmcX}_{/Y}$.
\end{defin}

${\whqmcX}$ is known to be hypercomplete in the sense of
\cref{defin-hypercomplete-topos} by \cite[Lemma 6.5.2.12]{Lurie:HTT}.
It is known to contain the full subcategory $\tau_{\leq m}{\qmcX}$ of
$m$-truncated objects (as in \cref{defin-truncated-object}) by
\cite[Lemma 6.5.2.9]{Lurie:HTT}.  The {\qcat} of functors to it from a
hypercomplete {\qtop} ${\qmcY}$ is known to be isomorphic to that of
functors from  ${\qmcY}$  to ${\qmcX}$ itself by \cite[Proposition
6.5.2.13]{Lurie:HTT}.

\begin{defin}\label{def-HTT-6.5.3.2}
  \cite[Definition 6.5.3.2]{Lurie:HTT} and \cite[Definition
  (8.4)]{Artin-Mazur}.  Let ${\qmcX}$ be an $\infty$-topos. A
  simplicial object $V_{\bullet} \in {\qmcX}_{\bdelt}$ is a {\bf
    hypercovering} (sometimes called a {\bf  hypercover}) of ${\qmcX}$ if for
  each $n> 0$, the unit map
\[
  (\eta_{V_{\bullet},n-1})_{n}:V_n
        \longrightarrow M_{n}(V)=\bigl(\coskel_{n-1} V_{\bullet}\bigr)_n
  \qquad \mbox{as in \cref{eq-coskel}  } 
\]
is an effective epimorphism as in \cref{def-effective-epi}. Here
$\coskel_{n-1}$ is the functor of \cref{def-coskeleton}, and the
codomain is the $n$th matching object.

We say that $V_{\bullet}$ is an {\bf effective hypercovering} of
${\qmcX}$ if the colimit of $V_{\bullet}$ is a final object of
${\qmcX}$.

An augmented simplicial object $V_{\bullet} \in {\qmcX}_{\bdelt_{+}}$
with $V_{-1}=X$ is an {\bf augmented hypercovering}, or a {\bf
  hypercover of $X$} if the associated simplicial
object is a hypercovering in ${\qmcX}_{/X}$.
\end{defin}

For an object $X$ in $\qmcX$, the {\qcat} $\qmcX _{/X}$ is again an
{\qtop}, by \cite[Proposition 6.3.5.1]{Lurie:HTT}, in which the
identity map on $X$ is a final object.  An effective hypercovering in
$\qmcX _{/X}$ is equivalent to a hypercovering in $\qmcX$ with colimit
$X$.

For each $n>0$, the map $ V_{n} \to V_{0}^{n+1}$ that sends an $n$-face to
its $(n+1)$-tuple of 0-faces is an effective epimorphism as in
\cref{def-effective-epi}.

The {\Cech} nerve $V_{\bullet}$ of \cref{ex-classic-Cech} is known to
be a hypercovering in which each map $\eta_{V,n}:V_{n}\to M_{n}(V)$ is
an isomorphism; see \cite[Enlightenment (8.5)(b)]{Artin-Mazur}.  A
finite skeleton of $V_{\bullet}$ is usually not a hypercovering.

\begin{defin}\label{def-hyperdescent}
  A presheaf $\mcalF$ on a site $(\qmcC,\tau)$ satisfies
  {\bf  hyperdescent} if for every augmented hypercovering
\[
  U_{\bullet} \longrightarrow X,
\]
the canonical map
\[
  \lim{\bdelt} \mcalF(U_{\bullet}) \longrightarrow \; 
  \mcalF(X)\; 
\]
is an equivalence.
\end{defin}

\begin{theorem}\cite[Theorem 6.5.3.12]{Lurie:HTT}
Let ${\qmcX}$ be an {\qtop}. The following conditions are
equivalent:
\begin{enumerate}
    \item[(i)] For every $X \in {\qmcX}$, every hypercovering 
    $U_{\bullet}$ of ${\qmcX}_{/X}$ is effective.
    \item[(ii)] The {\qtop} ${\qmcX}$ is hypercomplete.
\end{enumerate}
\end{theorem}

\begin{defin}\label{defin-hypersheaf}
  A {\bf hypersheaf} on an $\infty$-site $(\qmcC, \tau)$ with values
  in $\qmcS$ is a sheaf as in \cref{defin-sheaves} which is
  hypercomplete (as in \cref{defin-hypercomplete}) as an object in the
  presheaf category $\mcP(\qmcC)$.  We denote the category of such by
  $\Shv^{\mathrm{hyp}} (\qmcC)$.  The {\bf category of hypersheaves on
    $\qmcC$ with values in a presentable {\qcat} $\qmcD$} is
  \begin{displaymath}
    \Shv^{\mathrm{hyp}} (\qmcC; \qmcD)
    := \Shv^{\mathrm{hyp}} (\qmcC)\otimes \qmcD  \in \qPrL,
  \end{displaymath}

  \noindent where the tensor product of presentable {\qcats} above is
  as in \cite[\S4.8.1]{Lurie:HA}.
\end{defin}

This can alternatively be defined as the full subcategory of
presheaves that satisfy hyperdescent as in \cref{def-hyperdescent},
the analogue of the sheaf condition for hypercoverings. This follows
from \cite[Corollary~6.5.3.13]{Lurie:HTT} together with the formula
for the Lurie tensor product \cite[Proposition~4.8.1.17]{Lurie:HA}.

\begin{defin}\label{defin-hypersheafification}\cite[page~567]{BMCSY23} 
  {\bf Hypersheafification} $(-)^{\mathrm{hyp}}$ is the left adjoint
  of the inclusion
  \begin{displaymath}
\Shv^{\mathrm{hyp}} (\qmcC; \qmcD) \to \Shv (\qmcC; \qmcD)
\end{displaymath}
 
\noindent for $\qmcC$ and $\qmcD$ as above. 
\end{defin}

\begin{prop}\label{prop-BMCSY-5.1}
  {\bf Deligne completeness for hypersheaves.} \cite[Proposition 5.1]{BMCSY23},
  \cite[Theorem A.4.0.5]{Lurie:SAG}, and \cite[Proposition
  VI.9.0]{SGA4-II}.  Let $\qmcX$ be an {\qtop} which is locally
  coherent and hypercomplete. Then $\qmcX$ has enough points. In other
  words, given a morphism $\alpha : X \to Y \in \qmcX$ which is not an
  equivalence, there exists a geometric morphism
  $f^{*} : \qmcX \to \qmcS$ such that $f^{*}(\alpha)$ is not an
  equivalence.
  
\end{prop}


The following is proved in the discussion in \cite[pages 567--568]{BMCSY23}.

\begin{prop}\label{prop-BMCSY-5.1+}
{\bf Sheaves on  continuous G-sets.}
Let $G$ be a profinite group and $\qmcC$ a compactly generated
presentable {\qcat}.  Then
\begin{align*}
\Shv(\Fin_{G}; \qmcC)
 & \simeq
  \colim{U \trianglelefteq G}
  \qmcC^{B(G/U)}
  \in \qPrL,
\end{align*}

\noindent where $U$ ranges over the open normal subgroups of $G$.

When $\qmcC=\qmcS$ we also have
\begin{displaymath}
\Shv(\Fin_{G}; \qmcS)
 \simeq
  \lim{U } 
  \qmcS^{B(G/U)}
  \simeq \lim{U } {\color{cyan}B(G/U) }.
\end{displaymath}
 
\noindent This filtered limit is the {\qcat} associated with a path
connected space $BG$, which has a canonical basepoint leading to a
{\bf stalk} given by
\begin{numequation}\label{eq-stalk}
\begin{split}
  \mcalF_{e}
  := \colim{U \trianglelefteq G} \mcalF(G/U).
\end{split}
\end{numequation}

Moreover each such space valued sheaf is a hypersheaf.
\end{prop}

\begin{defin}\label{def-BMCSY-5.2}
  \cite[Definition 5.2]{BMCSY23}
  Let $G$ be a profinite group, let $\qmcC \in \CAlg(\qPrL)$ be
  $\infty$-semiadditive as in \cref{def-semiadd}, and let
  $\mcR \in \CAlg\bigl(\Shv(\Fin_{G}; \qmcC)\bigr)$.  We say that
  $\mcR$ is a {\bf continuous $G$-Galois extension of $\mcR(G/G)$} if
  for every open normal subgroup $U \trianglelefteq G$, the object
\begin{displaymath}
\mcR(G/U) \in \CAlg(\qmcC)^{B(G/U)}
\end{displaymath}
 
\noindent
  is a faithful $G$-Galois
  extension of $\mcR(G/G)$.
\end{defin}

\begin{prop}\label{prop-BMCSY-5.3}
  \cite[Proposition 5.3]{BMCSY23}
  Let $\mcalF \colon \Fin_{G}^{\op} \to \CAlg(\qmcC)$ be a finite
  product preserving functor such that $\mcalF(G/U)$ is a faithful
  $G/U$-Galois extension of $\mcalF(G/G)$ for every open normal
  subgroup $U \trianglelefteq G$. Then $\mcalF$ satisfies the sheaf
  condition, and hence is a continuous $G$-Galois extension.
\end{prop}

\begin{prop}\label{prop-BMCSY-5.4}
  \cite[Proposition~5.4]{BMCSY23} 
  Let $G$ be a profinite group, let $\qmcC \in \CAlg(\qPrL)$ be
  semiadditive, and let $\mcR$ be a continuous $G$--Galois
  extension. Then there is a symmetric monoidal equivalence
  \begin{displaymath}
\xymatrix
@R=4mm
@C=20mm
{
{\qmcC}\ar@<1.2ex>[r]^(.3){-\otimes \mcR}_(.3){\perp }
&{\Mod _{\mcR}(\Shv (\Fin_{G};\qmcC)).}
        \ar@<1.2ex>[l]^(.7){(-)(G/G)}
}
\end{displaymath}
\end{prop}

\begin{prop}\cite[Proposition~5.6]{BMCSY23}
  Let $G$ be a profinite group of finite virtual $p$-cohomological
  dimension, and let $\qmcC \in \CAlg(\qPrLst)$ be $p$-local. Let
  $\mcR$ be a continuous $G$-Galois extension with stalk
  $R := \mcR_{e}$ as in \cref{eq-stalk}, and let $L_{R}$ denote
  Bousfield localization with respect to $R$ in $\qmcC$. For every
  $M \in \Mod_{\mcR}(\Shv(\Fin_{G}; \qmcC))$, the presheaf $L_{R}M$ is
  a hypersheaf, and the map
\[
  M \longrightarrow L_{R}M
\]
exhibits the target as the hypersheafification of the source.
\end{prop}

Following \cite[Notation~5.7]{BMCSY23}, for
$X \in \Shv(\Fin_{G};\qmcC)$, we shall abuse notation and also denote
the stalk $e^{*}X$ by $X$, and for every open $U \leq G$, denote by
$X^{hU}$ the value of $X$ at $G/U \in \Fin_{G}$.

\subsection{$F$-descent and $F$-covers}\label[appendix]{sec-F-descent}

\begin{defin}\label{def-F-descent}
  {\bf $F$-descent and $F$-covers.} \cite[Definition 3.1.1]{LZ17}
  and \cite[Definition 6.6]{BHLS}.  Let $\qmcC$ be an {\qcat}
  admitting pullbacks,\linebreak $F \colon \qmcC^{\op} \to \qmcD$ a
  functor, and $u \colon V_{0} \to V_{-1}$ a morphism in
  $\qmcC$.  We say that {\bf $u$ is of $F$-descent} if
\[
  F \circ (V_{\bullet})^{\op} \colon \mathrm{N}(\bdelt_{+})
       \longrightarrow \qmcD
\]
is a limit diagram in $\qmcD$, where
\[
  V_{\bullet} \colon \mathrm{N}(\bdelt_{+})^{\op}
       \longrightarrow \qmcC
\]
is the {\Cech} nerve of $u$ as in \cref{def-Cech-nerve}.  We say that
{\bf $u$ is an $F$-cover} if every pullback of $u$ in $\qmcC$ is of
$F$-descent.
\end{defin}

We illustrate with three examples.

\begin{ex}\label{ex-ordinary-sheaves}
{\bf Ordinary sheaves on a topological space.}
Let $\qmcC$ be the poset category $\mcU(X) $ of open subsets of a
topological space $X$ as in \cref{ex-poset-open}, in which pullbacks
are intersections. Let $u$ be the map of \cref{ex-classic-Cech}
associated with an open covering $\left\{U_{\alpha} \right\} $ of $X$,
and let $F$ be an ordinary sheaf on $\mcU(X)$.  Then $u$ is of
$F$-descent.

The pullback of $u$ along an open inclusion $A\to X$ is the
corresponding map for the open cover of $A$ by its intersections with
the $U_{\alpha}$s.  This is also of $F$-descent, so $u$ is of
universal $F$-descent.
\end{ex}

\begin{ex}\label{ex-cech-nerve}
  The {\Cech} nerve for a map $f:X\to Y$ of spaces or spectra is the
  simplicial object $V_{\bullet}$ given by
\begin{displaymath}
\xymatrix
@R=-2mm
@C=15mm
{
{V_{0}}\ar@{=}[ddd]^(.5){}
 &{V_{1}}\ar@{=}[ddd]^(.5){}
   &{V_{2}}\ar@{=}[ddd]^(.5){}\\
   {\phantom{V}}\\
   {\phantom{V}}\\
{X}\ar[r]^(.35){}
 &{X\times _{Y}X}\ar@<.5ex>[r]^(.5){}\ar@<-.5ex>[r]^(.5){}
                \ar@<.5ex>[l]^(.5){} \ar@<-.5ex>[l]^(.5){}
   &{X\times _{Y}X\times _{Y}X}
                 \ar@<1ex>[r]^(.5){}\ar[r]^(.5){}\ar@<-1ex>[r]^(.5){}
                 \ar@<1ex>[l]^(.5){}\ar[l]^(.5){}\ar@<-1ex>[l]^(.5){}
      &{\cdots } \ar@<1.5ex>[l]^(.5){}\ar@<.5ex>[l]^(.5){}
                 \ar@<-.5ex>[l]^(.5){}\ar@<-1.5ex>[l]^(.5){}\\
{x_{0}}\ar@{|->}[r]^(.5){}
  &{(x_{0},x_{0})}\\     
{x_{0}}
  &  &{(x_{0},x_{0},x_{1})}\\
  &{(x_{0},x_{1})}\ar@{|->}[ur]^(.5){}\ar@{|->}[dr]^(.5){}
                  \ar@{|->}[ul]^(.5){}\ar@{|->}[dl]^(.5){}\\                  
{x_{1}}
  &{(x_{1},x_{2})}
     &{(x_{0},x_{1},x_{1})}\\
  &{(x_{0},x_{2})}
     &{(x_{0},x_{1},x_{2}),}\ar@{|->}[ul]^(.5){}\ar@{|->}[l]^(.5){}
                            \ar@{|->}[dl]^(.5){}  \\
  &{(x_{0},x_{1})}
}
\end{displaymath}
 
\noindent where the coordinates $x_{i}\in X $  all have the
same image under $f$.

Applying a contravariant functor $F$ leads to  maps
\begin{numequation}\label{eq-F-lim-F}
\begin{split}
  F(Y)
  \to \lim{\bdelt }F(V_{\bullet})
  \to F(X), 
\end{split}
\end{numequation}
 
\noindent whose composite is $F(f)$, where the limit is the
totalization of the cosimplicial object $F(V_{\bullet})$.  The first
map is an equivalence if $f$ is of $F$-descent.  If the same is true
for any pullback of $f$, then $f$ is an $F$-cover.
\end{ex}

\begin{ex}\label{ex-cyclo-comp}
  {\bf Cyclotomically completed algebraic $K$-theory.}
  If the functor of interest (such as algebraic $K$-theory) is
  covariant, then we have to start in the opposite category of its
  domain.  Let
\begin{displaymath}
\qmcC = \mathrm{CAlg} \left({\qSp}_{T(n)}\right)^{\op}.
\end{displaymath}
 
\noindent Pushouts in $\mathrm{CAlg}\!\left({\qSp}_{T(n)}\right)$ are
smash products over $\mS_{T(n)}$, so the opposite category has
pullbacks as required in \cref{def-F-descent}.

Let \(F\) be the functor $\rK_{\Cyc(n+1)}$ of \cref{def-KT-KK} (which has
values in $\qSp_{T(n+1)}$), and let $f$ be opposite of the map
\[
  \mS_{T(n)} \longrightarrow \mS_{T(n)}[\omega^{(n)}_{p^\infty}]
\]

\noindent Then \cref{eq-F-lim-F} becomes
\begin{displaymath}
  \rK_{\Cyc(n+1)}(\mS_{T(n)})
  \to \lim{\bdelt } \rK_{\Cyc(n+1)}(V_{\bullet})
  \to \rK_{\Cyc(n+1)}(\mS_{T(n)}[\omega^{(n)}_{p^\infty}])
\end{displaymath}
 
\noindent where $V_{\bullet}$ is a cosimplicial spectrum built from
$\mS_{T(n)}[\omega^{(n)}_{p^\infty}]$.  The first map is an
equivalence by \cref{cor-6.19}.  \cref{prop-BHLS6.24} is a similar
statement about the map
$f_k \colon R^{h(p^{k} \Z)} \longrightarrow R$, where 
$R = L_{T(n)} \BPn$ equipped with the $\Z$-action of
  \cref{thm-Theorem-5.4}.
\end{ex}

\subsection{Finite sets acted on by a profinite group}
\label[appendix]{sec-finite-profinite}

\begin{defin} \label{defin-TG} {\bf The Grothendieck site of
    continuous finite $G$-sets} \cite[\S III.9]{MacLane-Moerdijk} and
  \cite[Definition 4.1]{CM21}.  For a profinite group $G$, $\Fin_{G}$
  (denoted by $\mcT_{G}$ in \cite{CM21} and by $\msT_{G}$ in
  \cite{BMCSY23}) is the Grothendieck site defined as follows:
\begin{enumerate}
\item The underlying category of $\Fin_{G}$ is the category of
  continuous finite $G$-sets.  Continuity means that each point is
  fixed by an open (meaning of finite index) subgroup of $G$.
    \item A family of maps $\{ S_i \to S \}_{i \in I}$ forms a covering
    sieve if it is jointly surjective.
  \end{enumerate}

  Given an $\infty$-category $\qmcD$ with all limits, we let
  $\mathrm{Sh}(\Fin_{G},\qmcD)$ denote the $\infty$-category of
  $\qmcD$-valued sheaves on $\Fin_{G}$, as usual. We also write
  $\mcP_{\amalg }  (\Fin_{G},\qmcD)$ for the $\infty$-category of
  presheaves on $\Fin_{G}$ with values in $\qmcD$ which carry finite
  coproducts in $\Fin_{G}$ to finite products; equivalently, these are
  $\qmcD$-valued presheaves on the orbit category of $G$.

  If $\qmcD$ is also presentable as in \cref{def-presentable}, we
  define the {\bf stalk} of a sheaf $F$ by
  \begin{displaymath}
F_{e}:=\ colim_{H\subseteq G}F(G/H) ,
\end{displaymath}
 
\noindent where the colimit is over all open subgroups $H$.
\end{defin}

This Grothendieck topology is finitary as in
\cref{def-Grothendieck-topology}.  The category of sheaves of sets on
$\Fin_{G} $ is the category of continuous (discrete) $G$-sets.

\begin{prop}\label{prop-sheaf-cond-FinG}
  {\bf The sheaf condition for presehaves on $\Fin_{G}$.}
  \cite[Prop. A.3.3.1]{Lurie:SAG}. 
Let $G$ be a profinite group, and let $\Fin_{G}$ be the site as above.
A presheaf $F$ on $\Fin_{G}$ with values in an $\infty$-category
$\qmcD$ with limits is a sheaf if and only if:
\begin{enumerate}
  \item For $X, Y \in \Fin_{G}$, the natural map induces an equivalence
  \[
    F(X \sqcup Y) \simeq F(X) \times F(Y).
  \]
  That is, $F \in \mcP_{\amalg }  (\Fin_{G},\qmcD)$ is coproduct-preserving.
\item For every surjective map of $G$-sets $u:T \twoheadrightarrow S$,
   $F(X)$ is the limit of the {\Cech} nerve of $u$ as
  in \cref{def-Cech-nerve}.
\end{enumerate}
\end{prop}

Details of the following can be found in \cite[\S4.1]{CM21}.  For a
finite group $G$, $\Fin_{G}$ is the category of finite $G$-sets.  A
$\qmcD $-valued sheaf (for $\qmcD$ as above) on it is {\eqvt} to
$\qmcD $-valued functor on $BG$, which amounts to an object in $\qmcD$
equipped with a $G$-action.

A profinite group $G$ is the limit of its finite quotient groups $G/N$
for open normal subgroups $N$.  Its $\qmcD$-valued sheaf category is
\begin{displaymath}
  \mathrm{Sh}(\Fin_{G},\qmcD)
  \simeq \lim{N\subseteq G}\Fun(B(G/N), \qmcD). 
\end{displaymath}
 
\noindent where for $N' \leq N$, the functor
$\Fun\bigl(B(G/N'), \qmcD\bigr) \longrightarrow \Fun\bigl(B(G/N),
\qmcD\bigr)$ is given by $(\cdot)^{h(N/N')}$.

\begin{prop}\label{prop-CM21-4.12}
{\bf Sheafification.} \cite[Proposition 4.12]{CM21} 
Let $\qmcD$ be a presentable {\qcat} (\cref{def-presentable}) and let
$G$ be a profinite group. Suppose that for every open normal subgroup
$N \leq G$ and each subgroup $K \leq G/N$, the limit functor
\[
(\,\cdot\,)^{hK} : \Fun(BK, \qmcD) \longrightarrow \qmcD
\]
commutes with filtered colimits of $K$-objects in $\qmcD$. Let
$F \in \Psh(\Fin_{G}, \qmcD)$ be a product-preserving presheaf on
$\Fin_{G}$. Then the sheafification $F^{\mathrm{sh}}$ of $F$ is given
by the formula
\[
F^{\mathrm{sh}}(G/H)
  = \mathop{\mathrm{colim}}_{H' \leq H}
     F(G/H')^{h(H/H')},
\]
as $H' \leq H$ ranges over all open normal subgroups.
\end{prop}

\section{Operads}
\label[appendix]{sec-BV-tensor}

\subsection{The early work of May and Stasheff.}\label[appendix]{sec-May-Stasheff} Operads
are originally defined by May in \cite{May:loop}.  Light introductions
to the topic are given by Eva Belmont \cite{Belmont-quick} and by Jim
Stasheff \cite{Stasheff-what}.  More comprehensive treatments are
given by Martin Markl, Steven Shnider and Stasheff in \cite{MSS02}, by
Murray Bremner and Vladimir Dotsenko in \cite{Bremner-Dotsenko}, and
by Jean-Louis Loday and Bruno Vallette in \cite{Loday-Vallette}.

\begin{defin}\label{def-operad}
\cite[Definitions 1.1 and 3.12]{May:loop}

\begin{enumerate}[label={(\roman*)},itemindent=0em]
\item \label{def-operadi}
A {\bf non-$\Sigma$ (or nonsymmetric) operad} $(\mcO,\gamma ) $ is a
collection of spaces $\left\{\mcO (j):j\geq 0 \right\}$ where $\mcO
(0)$ is a single point, $\mcO (1)$ has a special point $1_{\mcO}$
related to the identity map, and for each $j\geq 0$ and each $j$-tuple
$(k_{1}, \dotsc , k_{j})$ of nonnegative integers, there are structure
maps
\begin{numequation}\label{eq-gamma-K}
\begin{split}
\gamma_{(k_{1},\dotsc ,k_{j})} :
\mcO(j)\times \mcO(k_{1})\times\dotsb \times \mcO(k_{j})
    \to \mcO(k_{1} + \dotsb  + k_{j})
\end{split}
\end{numequation}%

\noindent for which the associativity diagram
\cref{eq-operad-commutes} below commutes.  We will sometimes omit the
subscript $K$ on the structure map and the structure map $\gamma $ in
the operad.  The integer $j$ is the {\bf arity}, and a point in $\mcO
(j)$ is an {\bf operation of arity $j$}. The collection $\left\{\mcO
(j):j\geq 0 \right\}$ is called an {\bf operadic sequence}.

\item \label{def-operadii} $\mcO$ is simply an {\bf operad} (sometimes
called a {\bf symmetric operad}) if in addition each space $\mcO (j)$
comes equipped with an action of the symmetric group on $j$ letters,
$\Sigma_{j}$ satisfying the equivariance condition of
\cref{eq-symm-action} below.  The collection $\left\{\mcO (j):j\geq 0
\right\}$ is called a {\bf symmetric sequence}.  In \cite{MSS02} and
\cite{Loday-Vallette} it is called a $\Sigma$-module and an
$\SS$-module respectively.

\item \label{def-operadiii} An operad $\mcO $ is {\bf $\Sigma $-free} if
each $\mcO(j)$ for $j>0$ is acted on freely by $\Sigma_{j }$.

\item \label{def-operadiv} It is {\bf discrete} if each $\mcO(j)$ is
discrete.  For any operad $\mcO$, $\pi_{0}\mcO$ is a discrete operad,
and the evident map $\epsilon :\mcO\to \pi_{0}\mcO$ is the {\bf
augmentation} of $\mcO$.

\item \label{def-operadv} 
The {\bf $m$th truncation $\mcO_{\leq m}$ of
$\mcO$} for $m>0$ is given by
\begin{displaymath}
\mcO_{\leq m} (j)=\mycases{    
\mcO(j)
       &\mbox{for }0\leq j\leq m\\
\emptyset 
       &\mbox{for }j>m.
}
\end{displaymath}


\end{enumerate}
\end{defin}

We now spell out May's associativity condition.
Given a sequence of positive integers $K= (k_{1},\dotsc
,k_{j})$, let
\begin{numequation}\label{eq-K-notation}
\begin{split}
|K|:=j,\, ||K||:=k_{1}+\dotsb +k_{j}\mbox{ and } 
\mcO (K):=\mcO(k_{1})\times\dotsb \times \mcO(k_{j}),
\end{split}
\end{numequation}%

\noindent so the structure map of \cref{eq-gamma-K} is $\gamma_{K}
:\mcO (|K|)\times \mcO (K)\to \mcO (||K||)$.

Now suppose we replace each $k_{i}$ by a sequence of positive
integers
\begin{displaymath}
M_{i} = (m_{i,1},\dotsc m_{i,k_{i}}),
\qquad \mbox{with } \widetilde{M}= (M_{1},\dotsc ,M_{j}),
\end{displaymath}

\noindent 
and we define
\begin{displaymath}
|\widetilde{M}| := |K|,\qquad 
||\widetilde{M}||:=\sum_{1\leq i\leq j}||M_{i}||
\qquad \aand \mcO (\widetilde{M}):=\prod_{1\leq i\leq j}\mcO (M_{i}).
\end{displaymath}

\noindent It follows that
\begin{align*}
\mcO (|\widetilde{M}|)
 & = \mcO(k_{1})\times\dotsb \times \mcO(k_{j})\\
\mcO (K)\times \mcO (\widetilde{M})
 & = (\mcO(k_{1})\times\dotsb \times \mcO(k_{j}) )
       \times (\mcO (M_{1})\times \dotsb \times \mcO (M_{j}))  \\
 & =  (\mcO(k_{1})\times \mcO (M_{1}))\times \dotsb \times 
         (\mcO(k_{j})\times \mcO (M_{j})),  
\end{align*}

\noindent from which we have the map $\gamma
'_{\widetilde{M}}:=\gamma_{M_{1}}\times \dotsb \times \gamma_{M_{j}} $
to
\begin{displaymath}
\mcO (||M_{1}||)\times \dotsb \times \mcO (||M_{j}||).
\end{displaymath}

\noindent We also have a map 
\begin{displaymath}
\gamma ''_{\widetilde{M}}:=\gamma_{(||M_{1}||,\dotsc ,||M_{j}||)}
:\mcO (j)\times \mcO (||M_{1}||)\times \dotsb \times \mcO (||M_{j}||)
 \to \mcO(||\widetilde{M}||).
\end{displaymath}

Then we have a diagram
\begin{numequation}\label{eq-operad-commutes}
\begin{split}
\xymatrix
@R=8mm
@C=15mm
{
{\mcO (|K|)\times \mcO (K)\times \mcO (\widetilde{M})}
      \ar[d]_(.5){\mcO (|K|)\times \gamma '_{\widetilde{M}} }
      \ar[r]^(.55){\gamma_{K}\times\mcO (\widetilde{M}) }
  &{\mcO (||K||)\times \mcO (\widetilde{M})}
       \ar[d]^(.5){\gamma_{\widetilde{M}}}\\
{\mcO (|K|)\times \mcO (||M_{1}||)\times \dotsb \times \mcO (||M_{j}||)}
       \ar[r]^(.65){\gamma''_{\widetilde{M}}}
    &{\mcO (||\widetilde{M}||),}
}
\end{split}
\end{numequation}%

\noindent {\em which is required to commute}.

May's equivariance condition on the structure map $\gamma$ in the
symmetric case is as follows. There is a right action of the symmetric
group $\Sigma_j$ on $\mcO(j)$ such that the following formulas hold
for all $c \in \mcO(k)$, $d_s \in \mcO(j_s)$, $\sigma \in \Sigma_k$,
and $\tau_s \in \Sigma_{j_s}$:
\begin{numequation}\label{eq-symm-action}
\begin{split}
\gamma(c\sigma;\, d_1,\dots,d_k)
  &=
\gamma\bigl(c;\, d_{\sigma^{-1}(1)},\dots,d_{\sigma^{-1}(k)}\bigr)\,
      \sigma(j_1,\dots,j_k)\\
\mbox{and}\quad  
\gamma(c;\, d_1\tau_1,\dots,d_k\tau_k)
  &=
\gamma(c;\, d_1,\dots,d_k)\,
      (\tau_1 \oplus \cdots \oplus \tau_k),
\end{split}
\end{numequation}%

\noindent where $\sigma(j_1,\dots,j_k)$ denotes the permutation of $j
= j_1+\cdots+j_k$ letters that permutes the $k$ blocks of sizes
$j_1,\dots,j_k$ according to $\sigma$, and where $\tau_1 \oplus \cdots
\oplus \tau_k$ is the block sum permutation in $\Sigma_j$.  More
precisely, $\sigma(j_1,\dots,j_k) \in \Sigma_{j_1+\cdots+j_k}$ is the
block permutation sending the $s$th block of size $j_s$ to the
$\sigma(s)$th block, and
\[
\tau_1 \oplus \cdots \oplus \tau_k
  \;\in\; \Sigma_{j_1+\cdots+j_k}
\]
is the block sum permutation acting as $\tau_s$ on the $s$th block.

\bigskip
The following example motivates the associativity condition of
\cref{eq-operad-commutes}.

\begin{defin}\label{def-endo-op}
\cite[Definition 1.2]{May:loop} For an object $X$ in a topologically
enriched symmetric monoidal category $\mcC $ (such as that of pointed
spaces), the {\bf endomorphism operad} $\End_{X}$ has
\begin{displaymath}
\End_{X} (j):=\mcC (X^{j}, X),
\end{displaymath}

\noindent the space of maps from $X^{j}$ to $X$.  Here
the action of $\Sigma_{j}$ is induced by its action on $X^{j}$ by
permuting coordinates.  Then given maps $g_{i}:X^{k_{i}}\to X$ for
$1\leq i\leq j$, and $f:X^{j}\to X$, we define 
$\gamma_{K} (f, g_{1}, \dotsc ,g_{j})$
to be the composite
\begin{displaymath}
\xymatrix
@R=8mm
@C=10mm
{
{X^{k_{1}}\times X^{k_{2}}\times \dotsb \times X^{k_{j}}
      =X^{k_{1}+\dotsb +k_{j}}}
      \ar[d]_(.5){(g_{1}, g_{2},\dotsc ,g_{j})}\\
{X\times X\times \dotsb \times X=X^{j}}
      \ar[d]_(.5){f}\\
{X.}
}
\end{displaymath}

Let
\[
\circ _i \colon \mcC(X^n, X) \times \mcC(X^m, X)
\longrightarrow \mcC\!\bigl(X^{\,n+m-1}, X\bigr)
\]
be given, for \(1 \leq  i \leq  n\), by
\begin{align*}
\lefteqn{(f \,\circ_i\, g)(x_1, \dots, x_{m+n-1})}\qquad\qquad\\
 & = f\bigl(x_1, \dots, x_{i-1},\, g(x_i, \dots, x_{i+m-1}),\,
x_{i+m}, \dots, x_{n+m-1}\bigr).
\end{align*}
\end{defin}

The reader can verify that the diagram corresponding to
\cref{eq-operad-commutes} commutes for this operad.

\begin{defin}\label{def-operad morphism}
\cite[\S1]{May:loop} 
A {\bf morphism of operads }
\[
\phi : (\mcO ,\gamma )\to (\mcP, \delta)
\]
is a sequence of $\Sigma_j$-equivariant maps
\[
\phi_j : \mcO(j) \to \mcP(j)
\]
such that $\phi_1(1_{\mcO}) = 1_{\mcP }$ and the following diagram
commutes, with notation as in \cref{eq-K-notation}:

\[
\begin{tikzcd}
\mcO(j) \times \mcO(k_1) \times \cdots \times \mcO(k_{j}) 
  \arrow[r, "\gamma_{K}"] 
  \arrow[d, "\phi_j \times \phi_{k_1} \times \cdots \times \phi_{k_{j}}"'] 
& \mcO(||K||) \arrow[d, "\phi_j"] \\
\mcP(j) \times \mcP(k_1) \times \cdots \times \mcP(k_j) 
  \arrow[r, "\delta_{K}"] 
& \mcP(||K||).
\end{tikzcd}
\]
\end{defin}

\begin{defin}\label{def-O-alg}
Given an operad $\mcO$, an object $X$ in $\mcC $ as in
\cref{def-endo-op} is an {\bf $\mcO$-algebra}, or {\bf $\mcO$ acts on
$X$}, if there is a operad morphism (\cref{def-operad morphism})
$\theta :\mcO \to \End_{X} $.  Such a map assigns to each
point in $\mcO (j)$ a map $X^{j}\to X$.  The point 
$1_{\mcO}\in \mcO(1)$ gets sent to the identity map on $X$.  The
$\Sigma_{j}$-{\eqvr} map $\mcO (j)\to \Map_{*} (X^{j}, X)$ is adjoint
to structure maps
\begin{numequation}\label{eq-theta-O}
\begin{split}
\xymatrix
@R=4mm
@C=10mm
{
{\mcO (j)\times_{\Sigma_{j}}X^{j}}
     \ar[r]^(.7){\theta^{\mcO,X}_{j}}
  &{X,}
}
\quad \mbox{and}\quad  
\xymatrix
@R=4mm
@C=10mm
{
{X^{j}}\times_{\Sigma_{j}}\mcO (j)
     \ar[r]^(.7){\theta^{X,\mcO}_{j}}
  &{X.}
}
\end{split}
\end{numequation}%

\noindent For $p\in \mcO (j)$ and $x_{i}\in X$ for $1\leq i\leq j$, we
will sometimes write
\begin{numequation}\label{eq-pxj}
\begin{split}
p (x_{1},\dotsc ,x_{j}):=\theta^{\mcO,X}_{j} (p,x_{1},\dotsc x_{j}).
\end{split}
\end{numequation}%

\noindent These maps define the structure of $X$ as an $\mcO$-algebra.
We will denote the category of such algebras by $\mcO[\mcC]$.

We can make a similar definition in the nonsymmetric case by omitting
the symmetric group actions in \cref{eq-theta-O}.
\end{defin}

\begin{defin}\label{def-operation}
{\bf $\mcO$-operations.}  Let $X$ be an $\mcO$-algebra
as in \cref{def-O-alg}.  An {\bf $\mcO$-operation on $X$} is a map $a
:X^{n}\to X^{m}$ for $n>m$ obtained as follows.  Let 
\begin{displaymath}
n=n_{1}+n_{2}+\dotsb +n_{m}\qquad \mbox{with $n_{i}>0$ for $1\leq i\leq m$}, 
\end{displaymath}

\noindent and let $a_{i}:X^{n_{i}}\to X$ be a map in 
$\theta (\mcO (n_{i}))\subseteq \mcC (X^{n_{i}},X)$.  Then 
\begin{displaymath}
a :=\prod_{1\leq i\leq m} \left(a_{i}:X^{n_{i}}\to X \right).
\end{displaymath}
\end{defin}

\begin{defin}\label{def-interchange}
  Suppose $X$ is both a $\mcP$-algebra and a $\mcQ$-algebra for
  operads $\mcP$ and $\mcQ$.  Then the two structures {\bf
    interchange} if for each $k,\ell >0$, the following diagram
  commutes.
\begin{numequation}\label{eq-interchange}
\begin{split}
\xymatrix
@R=8mm
@C=2mm
{
{(\mcP (k)\times_{\Sigma_{k}}X^{k\ell  })\times_{\Sigma_{\ell }}\mcQ (\ell )}
    \ar[d]_(.5){\theta^{\mcP ,X^{\ell }}_{k} 
                 \times_{\Sigma_{\ell }}\mcQ (\ell )} 
    \ar@{=}[rr]_(.5){}
  & &{\mcP (k)\times_{\Sigma_{k}} (X^{k\ell }
              \times_{\Sigma_{\ell }}\mcQ (\ell ))}
    \ar[d]^(.5){\mcP (k)\times_{\Sigma_{k}}\theta^{X^{k },\mcQ}_{\ell } } \\
{X^{\ell}\times_{\Sigma_{\ell}}\mcQ (\ell)}\ar[r]_(.65){\theta_{\ell}^{X,\mcQ}}
  &{X}
    &{\mcP (k)\times_{\Sigma_{k}}X^{k}}\ar[l]^(.65){\theta_{k}^{\mcP,X}}
}
\end{split}
\end{numequation}%

\noindent The structure maps $\theta $ are those of
\cref{eq-theta-O}. The left and right actions of $\Sigma_{k}$ and
$\Sigma_{\ell }$ respectively on $X^{kl}=X^{\langle k \rangle\times
\langle \ell \rangle}$ are induced by their actions on the indexing
set
\begin{displaymath}
\langle k \rangle\times \langle \ell
\rangle := \left\{1,2,\dotsc k \right\}\times \left\{1,2,\dotsc \ell  \right\}.
\end{displaymath}

We can make a similar definition for nonsymmetric operads by dropping
the action of either or both symmetric groups.
\end{defin}

\begin{defin}\label{def-MN}
\cite[Definition 3.1]{May:loop} {\bf Operads for topological monoids
and commutative topological monoids.}
\begin{enumerate}[label={(\roman*)},itemindent=0em]

\item \label{def-MNi} Let $\mcM $, also known as $\Assoc$, be the
discrete operad with $\mcM (j)=\Sigma_{j}$ with the evident structure
maps.  Hence an $\mcM$-space is a topological monoid.

\item \label{def-MNii} Let $\mcN$, also known as $\Comm$,  be the
discrete operad with $\mcN (j)=\pt$. with the evident structure
maps. Hence an $\mcN$-space is a commutative topological monoid. $\mcN
$ is a terminal object in the category of operads, {\ie } any operad
$\mcO$ admits a unique operad morphism to it.  It follows that
commutative topological monoid is also an $\mcO$-space for any operad
$\mcO$.
\end{enumerate}
\end{defin}

\begin{defin}\label{def-Stash-op}
\cite{Stasheff-assoc}  The {\bf Stasheff non-$\Sigma $ operad $\mcK
$}, also known as $\AA_{\infty }$, has $\mcK (1)=\pt$ and $\mcK (k)$
for $k\geq 2$ is a certain contractible $(k-2)$-dimensional polytope
called an {\bf associahedron}.  $\mcK (2)$ is a point, $\mcK (3)$ is a
line segment, $\mcK (4)$ is a pentagon, and $\mcK (5)$ is an {\em
enneahedron} 
(image  from Wikipedia) 
with nine faces 
(three disjoint quadrilaterals and six
pentagons) and fourteen vertices.  Each pentagonal face is a copy of
$\mcK (4)$.  
\begin{center}
\includegraphics[height=5cm]{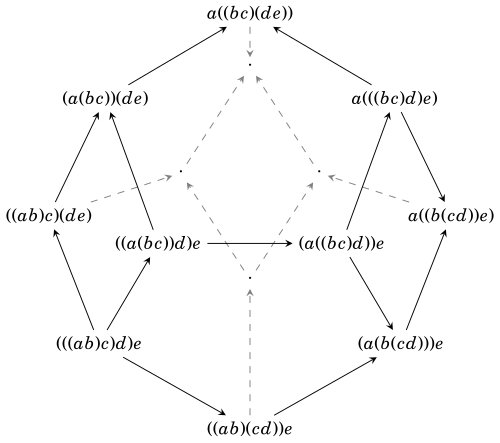}
\includegraphics[height=5cm]{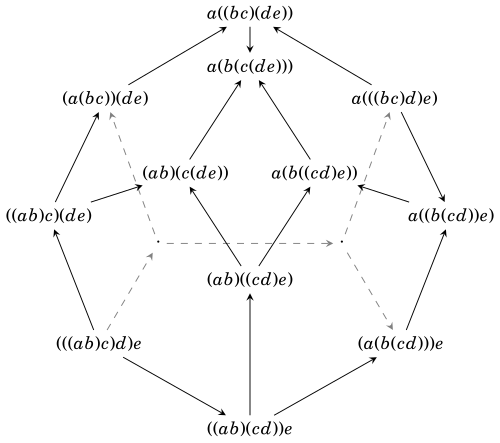}
\end{center}

\noindent A space is a $\mcK$-algebra if it has a multiplication which
is associative up to all higher homotopies.  Each edge in $\mcK (5)$
corresponds to a homotopy between the two indicated ways to multiply
five points.

The {\bf $m$th Stasheff non-$\Sigma$ operad $\AA_{m} $} is the
truncation $\mcK_{\leq m}$ as in \cref{def-operad}\cref{def-operadv}.

Hence every pointed space is an $\AA_{1}$-algebra. An
$\AA_{2}$-algebra is a pointed space equipped with a unital multiplication
with no associativity condition, and an $\AA_{3}$-algebra is a
homotopy associative $H$-space with no higher homotopy associativity.
\end{defin}

More information about associahedra can be found in \cite{CSZ} and
\cite{Loday04}.

\begin{ex}\label{ex-interchange-A2}
{\bf The interchange condition for $\mcQ=\AA_{2}$.}  Since $\AA_{2}
(\ell )$ is a point for $\ell \leq 2$ and empty for $\ell >2$, We need
consider the diagram of \cref{eq-interchange} only for $\ell =1$ and
2, and we see that it commutes trivially for $\ell =1$. For $\ell =2$,
it reads
\begin{displaymath}
\xymatrix
@R=8mm
@C=2mm
{
{(\mcP (k)\times_{\Sigma_{k}}X^{2k  })\times\AA_{2} (2 )}
    \ar[d]_(.5){\theta^{\mcP ,X^{2 }}_{k} 
                 \times\AA_{2} (2 )} 
    \ar@{=}[rr]_(.5){}
  & &{\mcP (k)\times_{\Sigma_{k}} (X^{2k }
              \times\AA_{2} (2 ))}
    \ar[d]^(.5){\mcP (k)\times_{\Sigma_{k}}\theta^{X^{k },\AA_{2}}_{2 } } \\
{X^{2}\times\AA_{2} (2)}\ar[r]_(.65){\theta_{2}^{X,\AA_{2}}}
  &{X}
    &{\mcP (k)\times_{\Sigma_{k}}X^{k}}\ar[l]^(.65){\theta_{k}^{\mcP,X}}
}
\end{displaymath}

\noindent Since $\AA_{2}$ is not symmetric, there is no
$\Sigma_{2}$-action as in \cref{eq-interchange}.  The diagram means
that for $p\in \mcP (k)$ and $x_{i}\in X$ for $1\leq i\leq 2k$, we
have
\begin{displaymath}
p (x_{1},\dotsc ,x_{k})\cdot p (x_{k+1},\dotsc ,x_{2k})
= p (x_{1}\cdot x_{k+1},\dotsc ,x_{k}\cdot x_{2k}),
\end{displaymath}

\noindent with notation as in \cref{eq-pxj}, where $(-\cdot -)$ is the
monoidal structure associated with $\AA_{2}$.
\end{ex}

\begin{defin}\label{def-little-cubes}
\cite[Definition 4.1]{May:loop} The {\bf little $n$-cubes operad}
$\EE_{n}$, also known as $\mcC_{n}$, has $\EE_{n} (j)$ being the space
of $j$-tuples of rectilinear embeddings of $I^{n}$ into $I^{n}$, the
$n$-cube $[0,1]^{n}$, such that the $j$ images of the interior are
disjoint.  More precisely, a point in $\EE_{n} (j)$ is a collection of
maps $e_{i}:I^{n}\to I^{n}$ for $1\leq i\leq j$ of the form
\begin{displaymath}
e_{i} (x_{1},\dotsc .x_{n})
 := (a_{i,1}+b_{i,1}x_{1},\dotsc, a_{i,n}+b_{i,n}x_{n}),
\end{displaymath}

\noindent where the coefficients $a_{i,k}$ and $b_{i,k}$ satisfy
\begin{displaymath}
a_{i,k}\geq 0,\qquad  b_{i,k}>0,
\qquad 
 a_{i,k}+b_{i,k}\leq 1,
\end{displaymath}

\noindent and the images of the interiors of the $I^{n}$s under the maps
$e_{k}$ for $1\leq k\leq j$ are disjoint.
\end{defin}

This operad is of interest because it acts on $\Omega^{n}X$ for any
pointed space $X$.  The space $\EE_{n} (j)$ has a free action of
$\Sigma_{j}$.  $\EE_{1} (j)$ is homotopy equivalent to $\Sigma_{j}$,
and $\EE_{n} (j)$ is $(n-2)$-connected for $n\geq 2$.

It is known that an $\EE_{1}$-structure is {\eqt} to an
$\AA_{\infty }$-structure; see \cite[Example 5.1.0.7]{Lurie:HA}. 

\begin{defin}\label{defin-En-restriction}
Every $\EE_{n+1}$-algebra is also an $\EE_{n}$-algebra.
\end{defin}

\proof
For each map 
\begin{displaymath}
e=e_{1}\amalg \dotsb \amalg e_{k}:\coprod_{1\leq i\leq k}I^{n}
\to I^{n} 
\end{displaymath}

\noindent as in \cref{def-little-cubes}, we can form a diagram
\begin{displaymath}
\xymatrix
@R=6mm
@C=10mm
{
{\displaystyle{\coprod_{1\leq i\leq j}I^{n}}}
     \ar[r]^(.58){e}
     \ar[d]^(.5){}
  &{I^{n}}\ar[d]^(.5){}\\
{\displaystyle{\coprod_{1\leq i\leq j}I^{n+1}}}
      \ar[r]^(.6){\widetilde{e}}
  &{I^{n+1},}
}
\end{displaymath}

\noindent where the vertical maps send each $n$-cube to the face of
the corresponding $(n+1)$-cube with $x_{n+1}=0$, and the last
coordinate of each map $\widetilde{e}_{i}$ is $x_{n+1}$. Thus we get
a map $\EE_{n} (j)\to \EE_{n+1} (j)$ and the result follows.  \qed

\begin{defin}\label{def-E-inf}
The {\bf $\EE_{\infty }$-operad} has
\begin{displaymath}
\EE_{\infty } (j) :=   \colim{n}\EE_{n} (j),
\end{displaymath}

\noindent (which is a contractible free $\Sigma_{j}$-space) where the
maps $\EE_{n} (j)\to \EE_{n+1} (j)$ are those constructed above.
\end{defin}

\subsection{Some generalizations}\label[appendix]{sec-generalizations} 

Returning to \cref{def-operad}, the space $\mcO (1)$ can be regarded
as the space of morphisms in a topological category with a single
object.  The special point in it is the identity morphism and
\cref{eq-operad-commutes} says that composition of morphisms is
associative.  Thus an operad as defined by May is a one object
topological category with additional structure.

This suggests three ways of generalizing \cref{def-operad} so as to
accommodate more categories as special cases:
\begin{enumerate}[label={(\roman*)},itemindent=0em]
\item The spaces $\mcO (j)$ of \cref{def-operad} could be replaced by
objects in a bicomplete closed symmetric monoidal category.  This was
studied by Max Kelly in \cite{Kelly-operads}, but he did not treat the
Boardman-Vogt tensor product there.  It is also treated in \cite[Part
II, Chapter 1]{MSS02}.

\item The point $\mcO (0)$ could be replaced by a set $C$ of {\bf
colors}.  We require $\mcO (0)$ to be a set no matter which symmetric
monoidal category the $\mcO (j)$ are allowed to belong to. The latter
are replaced by spaces (or whatever)
\begin{displaymath}
\mcO (x_{1},\dotsc ,x_{j};y) \qquad \mbox{for }x_{i},y\in C,
\end{displaymath}

\noindent with the associativity diagram \cref{eq-operad-commutes}
suitably modified.  The resulting object is a {\bf colored operad},
also known as a {\bf symmetric multicategory}.  The term
``multicategory'' refers to the fact that morphisms are generalized to
functions of several variables. See Donald Yau's \cite{Yau}.

\item Associated to a colored operad $\mcO$ is a small category
$\mcC_{\mcO}$ with object set $\mcO (0)$ and morphism sets $\mcO
(x;y)$ for $x,y\in \mcO (0)$.  The definition can be modified so that
$\mcC_{\mcO}$ gets replaced by an {\qcat}.  Lurie calls such an object 
an {\bf $\infty $-operad} in \cite[Definition 2.1.1.10]{Lurie:HA}.

\end{enumerate}

{\em We will only need the first of these.}

\subsection{Trees}\label[appendix]{sec-trees}

\begin{defin}\label{def-tree}
\cite[\S1.5]{MSS02},
A {\bf tree} is a finite connected contractible graph. We will modify the
standard convention that all edges in a graph have two
adjacent vertices and delete the vertices with only one adjacent
edge. This means that some edges will have only one adjacent vertex
and we call these edges {\bf  external edges}. The edges which are
adjacent to two vertices will be called {\bf internal edges}. Occasionally
it will be convenient to use the standard convention with two vertices
adjacent to every edge, in which case we call a vertex adjacent to
just one edge an {\bf external vertex}.  The remaining vertices will be
called {\bf internal vertices}.  

All trees are assumed to have at least one edge; the tree with just
one edge (and no vertices) is called the {\bf trivial tree}. A rooted
tree is a tree with a distinguished external edge, called the root.
The remaining external edges are called leaves. An external vertex
adjacent to a leaf will be called a {\bf leaf vertex}, and the
external vertex adjacent to the root, the {\bf root vertex}. A rooted
tree has a natural orientation with each edge oriented in the
direction of the vertex closest to the root. The root edge is oriented
toward the root vertex, but in the case of the trivial tree, this is
ambiguous so we have to choose an orientation. In any rooted tree,
every vertex is adjacent to a single outgoing edge.  

The {\bf valence} of an internal vertex is the number of incoming
edges.  A tree with no vertices of valence one is called {\bf
reduced}. In such a tree the {\bf distance from the root} of an
internal vertex is one more than the number of internal vertices
between it and the root.

A {\bf corolla} is a tree with no internal edges and a single internal
vertex of valence at least 2. We denote the $n$-leafed corolla by
$T_{n}$.  A {\bf binary tree} is one in which each internal vertex has
valence 2. A {\bf planar tree} is a tree equipped with an embedding in
the plane.  A tree that is not so equipped is said to be {\bf
nonplanar}, even though it can be embedded in the plane, unlike a
nonplanar graph.

Unless otherwise indicated, we will assume all trees are reduced,
rooted and nonplanar.  \end{defin}

\begin{defin}\label{def-tree-operad}
{\bf The tree operad.} \cite[\S1.5]{MSS02} Let $Tree(n)$ for $n>0$ be
the set of isomorphism classes of reduced nonplanar trees with one
root and $n$ leaves, with distinct labels, usually the integers 1
through $n$.  The symmetric group acts by permuting the labels. The
sequence
\[
Tree = \{Tree(n):n \geq 1\}
\]
forms an operad in the category of sets.
Given trees $S \in Tree(k)$ and $T \in Tree(j)$, 
for each $1 \leq i \leq k$, let 
\begin{numequation}\label{eq-circ-i}
\begin{split}
S \circ_i T
\end{split}
\end{numequation}%

\noindent
be the tree obtained by grafting the root of $T$ to the leaf of $S$
labeled $i$.  

There is a structure map $\gamma_{K}$ as in
\cref{eq-gamma-K}, 
\begin{numequation}\label{eq-gamma-K-tree}
\begin{split}
\gamma_{K} :
Tree(j)\times Tree(k_{1})\times\dotsb \times Tree(k_{j})
    \to Tree(k_{1} + \dotsb  + k_{j})
\end{split}
\end{numequation}%

\noindent in which, for $1\leq i\leq j$, the root of the tree
with $k_{i}$ leaves is grafted onto the $i$th leaf of the
first tree.
\end{defin}

The reader should consult \cite[\S1.4 and \S2.2]{Boardman-Vogt} for
more information, including the definitions of the terms {\bf twig},
{\bf stump}, {\bf cherry}, {\bf cherry tree}, {\bf fully grown cherry
tree}, {\bf planted cherry tree} (with no reference to George
Washington), and {\bf copse}.

\begin{remark}\label{rem-how-to-draw-trees}
{\bf How to draw trees.} Linguists draw trees with the root at
the top, while botanists draw them with the root at the bottom. We
will follow the botanical convention.
\end{remark}

The grafting process of \cref{eq-circ-i} is illustrated in
\cref{eq-grafting}. In each tree the leaves (top vertices) are
numerically labeled, and the root is the bottom edge. The two trees on
the left each have a single internal vertex. Their {\bf valences}, the
number of edges coming in from above, are 5 and 3.  The tree on the
right has two, with the same valences as the corresponding internal
vertices on the left.
\begin{numequation}\label{eq-grafting}
\begin{split}
T_{5}\circ_{4}T_{3}=
\vcenter{\xymatrix
@R=3mm
@C=2mm
{
{}\\
{1}\ar@{-}[drr]^(.5){}
  &{2}\ar@{-}[dr]^(.5){}
    &{4}\ar@{-}[d]^(.5){}
      &{4}\ar@{-}[dl]^(.5){}
        &{5}\ar@{-}[dll]^(.5){}\\
  & &*=0{\bullet}\ar@{-}[d]^(.5){}\\
  & &*=0{}
}}\circ_{4}
\vcenter{\xymatrix
@R=3mm
@C=2mm
{
{6}\ar@{-}[dr]^(.5){}
  &{7}\ar@{-}[d]^(.5){}
    &{8}\ar@{-}[dl]^(.5){}\\
  &*=0{\bullet}\ar@{-}[d]^(.5){}\\
  &{}
}}:=
\vcenter{\xymatrix
@R=3mm
@C=2mm
{
{}
  &{}
    &{6}\ar@{-}[dr]^(.5){}
      &{7}\ar@{-}[d]^(.5){}
        &{8}\ar@{-}[dl]^(.5){}\\
{1}\ar@{-}[drr]^(.5){}
  &{2}\ar@{-}[dr]^(.5){}
    &{3}\ar@{-}[d]^(.5){}
      &*=0{\bullet}\ar@{-}[dl]^(.5){}
        &{5}\ar@{-}[dll]^(.5){}  \\
  & &*=0{\bullet}\ar@{-}[d]^(.5){}\\
  & &{}
}}
\end{split}
\end{numequation}%

An instance of the structure map of \cref{eq-gamma-K-tree},
\begin{displaymath}
\gamma_{(2,3)}: Tree(2)\times (Tree(2)\times Tree(3))\to Tree(5)
\end{displaymath}

\noindent is shown here.  A similar picture is on 
\cite[page~14]{Boardman-Vogt}.

\begin{displaymath}
\vcenter{\xymatrix
@R=3mm
@C=2mm
{
{1}\ar@{-}[dr]^(.5){}
  &{}
    &{2}\ar@{-}[dl]^(.5){}\\
  &*=0{\bullet}\ar@{-}[d]^(.5){}\\
  &{}
}}
\circ
\left(\vcenter{
\xymatrix
@R=3mm
@C=1mm
{
{3}\ar@{-}[dr]^(.5){}
  &{}
    &{4}\ar@{-}[dl]^(.5){}\\
  &*=0{\bullet}\ar@{-}[d]^(.5){}\\
  &{}
}}, \vcenter{
\xymatrix
@R=3mm
@C=1mm
{
{5}\ar@{-}[dr]^(.5){}
  &{6}\ar@{-}[d]^(.5){}
    &{7}\ar@{-}[dl]^(.5){}\\
  &*=0{\bullet}\ar@{-}[d]^(.5){}\\
  &{}
}
} \right)
\vcenter{\xymatrix
@R=3mm
@C=1mm
{
&&{3}\ar@{-}[dr]^(.5){}
  &{4}\ar@{-}[d]^(.5){}
    &{5}\ar@{-}[dr]^(.5){}
      &{6}\ar@{-}[d]^(.5){}
        &{7}\ar@{-}[dl]^(.5){}\\
\ar@{|->}[rr]^(.5){} 
&& &*=0{\bullet}\ar@{-}[dr]^(.5){}  
    & &*=0{\bullet}\ar@{-}[dl]^(.5){}\\
&&  & &*=0{\bullet}\ar@{-}[d]^(.5){}\\
&&  & & &{}   
}} = (T_{2}\circ_{2}T_{2})\circ_{3}T_{3}
\end{displaymath}

\begin{ex}\label{ex-small-trees}
{\bf Some isomorphism classes of trees with few leaves.}  

There is one 2-leafed tree, $T_{2}$.  

There are two 3-leafed trees, $T_{3}$ and $T_{2}\circ_{1}T_{2}\cong
T_{2}\circ_{2}T_{2}$ shown here.
\begin{numequation}\label{eq-3-leaves}
\begin{split}
\vcenter{\xymatrix
@R=3mm
@C=2mm
{
{1}\ar@{-}[dr]^(.5){}
  &{2}\ar@{-}[d]^(.5){}
    &{3}\ar@{-}[dl]^(.5){}\\
  &*=0{\bullet}\ar@{-}[d]^(.5){}\\
  &{}
}}
\mbox{ and } 
\vcenter{\xymatrix
@R=2mm
@C=1mm
{
{1}\ar@{-}[dr]^(.5){}
  & &{2}\ar@{-}[dl]^(.5){}\\
  &*=0{\bullet}\ar@{-}[dr]^(.5){}
    & &{3}\ar@{-}[dl]^(.5){}\\
  & &*=0{\bullet}\ar@{-}[d]^(.5){}\\
  & &{}
}}
\cong 
\vcenter{\xymatrix
@R=2mm
@C=1mm
{
  &{2}\ar@{-}[dr]^(.5){}
    & &{3}\ar@{-}[dl]^(.5){}\\
{1}\ar@{-}[dr]^(.5){}
  & &*=0{\bullet}\ar@{-}[dl]^(.5){}\\
  &*=0{\bullet}\ar@{-}[d]^(.5){}\\
  &{}
}}
\end{split}
\end{numequation}%

There are five 4-leafed trees, $T_{4}$, $T_{2}\circ_{1}T_{3}$,
$(T_{2}\circ_{1}T_{2})\circ_{3}T_{2}$, 
$(T_{2}\circ_{1}T_{2})\circ_{1}T_{2}$ and $T_{3}\circ_{1}T_{2}$.
\begin{numequation}\label{eq-4-leaves}
\begin{split}
\vcenter{\xymatrix
@R=3mm
@C=0mm
{
{1}\ar@{-}[dr]^(.5){}
  &{2}\ar@{-}[d]^(.5){}
     &{3}\ar@{-}[dl]^(.5){}
        &{4}\ar@{-}[dll]^(.5){}\\
  &*=0{\bullet}\ar@{-}[d]^(.5){}\\
  &{}
}}\qquad 
\vcenter{\xymatrix
@R=2mm
@C=0mm
{
{1}\ar@{-}[dr]^(.5){}
  &{2}\ar@{-}[d]^(.5){}
    &{3}\ar@{-}[dl]^(.5){}\\
  &*=0{\bullet}\ar@{-}[d]^(.5){}
    &{4}\ar@{-}[dl]^(.5){}\\
  &*=0{\bullet}\ar@{-}[d]^(.5){}\\
  &{}
}}\qquad 
\vcenter{\xymatrix
@R=3mm
@C=0mm
{
{1}\ar@{-}[dr]^(.5){}
  &{2}\ar@{-}[d]^(.5){}
    &{3}\ar@{-}[d]^(.5){}
      &{4}\ar@{-}[dl]^(.5){}\\
  &*=0{\bullet}\ar@{-}[d]^(.5){}
    &*=0{\bullet}\ar@{-}[dl]^(.5){}\\
  &*=0{\bullet}\ar@{-}[d]^(.5){}\\
  &{}
}}\qquad 
\vcenter{\xymatrix
@R=2mm
@C=0mm
{
{1}\ar@{-}[d]^(.5){}
  &{2}\ar@{-}[dl]^(.5){}\\
*=0{\bullet}\ar@{-}[d]^(.5){}
  &{3}\ar@{-}[dl]^(.5){}\\
*=0{\bullet}\ar@{-}[d]^(.5){}
  &{4}\ar@{-}[dl]^(.5){}\\
*=0{\bullet}\ar@{-}[d]^(.5){}\\
{}
}}\qquad 
\vcenter{\xymatrix
@R=2mm
@C=0mm
{
{1}\ar@{-}[d]^(.5){}
  &{2}\ar@{-}[dl]^(.5){}\\
*=0{\bullet}\ar@{-}[dr]^(.5){}
  &{3}\ar@{-}[d]^(.5){}
    &{4}\ar@{-}[dl]^(.5){}\\
  &*=0{\bullet}\ar@{-}[d]^(.5){}\\
  &{}
}}
\end{split}
\end{numequation}%

\end{ex}

\subsection{Free operads and coproduct operads}\label[appendix]{sec-free-coprod}

\begin{defin}\label{def-free-operad}
{\bf The free operad on a symmetric sequence.} \cite[\S II.1.9]{MSS02}
and \cite[Theorem 5.5.1]{Loday-Vallette}.  For any symmetric sequence
$\mcS$ as in \cref{def-operad}\cref{def-operadii}, the free operad
$\mathbb{F}(\mcS)$ has
\[
\mathbb{F}(\mcS)(n)
\;\cong\;
\coprod_{[T]\in Tree(n)}
\left(
\prod_{v \in V(T)} \mcS(\operatorname{val}(v))
\right)\Big/\!\!\operatorname{Aut}(T),
\]
where:
\begin{itemize}
    \item $Tree (n)$ is as in \cref{def-tree-operad},

    \item $V(T)$ is the set of internal vertices of $T$,

    \item the valence $\operatorname{val}(v)$ is the number of
incoming edges at the internal vertex $v$,

    \item $\operatorname{Aut}(T)$ is the group of automorphisms of $T$
preserving the root and leaf labels.
\end{itemize}
\end{defin}

For $\mcS = \mcP \amalg \mcQ$, we have
\begin{numequation}\label{eq-free-operad}
\begin{split}
\mathbb{F}(\mcP \amalg \mcQ)(n)
& \;\cong\;
\coprod_{[T]}
\left(
\prod_{v \in V(T)} \bigl(\mcP(\operatorname{val}(v))
            \amalg \mcQ(\operatorname{val}(v))\bigr)
\right)\Big/\!\!\operatorname{Aut}(T).
\end{split}
\end{numequation}%

\noindent where the coproduct is over all $n$-leafed trees $T$.
Thus each internal vertex of valence $k$ in $T$ is labelled by a point in
$\mcP$ or $\mcQ$ of arity $k$.

\begin{defin}\label{def-coproduct-operad}
{\bf The coproduct of two operads} $\mcP$ and $\mcQ$ (either of which
may or may not be symmetric) is given by
\[
(\mcP \amalg \mcQ) (n)
\;\cong\;
\left(
\coprod_{[T]}
\prod_{v \in V(T)}
\bigl(\mcP(\operatorname{val}(v)) \amalg \mcQ(\operatorname{val}(v))\bigr)
\right)\Big/\!\!\sim,
\]
where $\sim$ is the equivalence relation generated by replacing any
$\mcP$-only subtree by its composite in $\mcP$ (meaning the corolla
with the the same number of leaves), any $\mcQ$-only subtree by its
composite in $\mcQ$, and imposing the standard operad relations.
\end{defin}

\begin{ex}\label{ex-A1}
{\bf The coproduct with $\AA_{1}$.}
For $\mcQ=\AA_{1}$ as in
\cref{def-Stash-op}, recall that $\AA_{1} (1)$ is a single point and
$\AA_{1} (n)$ is empty for $n>1$.  Since the valence of each internal
vertex exceeds 1, the space $\mcQ(\operatorname{val}(v))$ in
\cref{eq-free-operad} is always empty.  It follows that $\mcP \amalg
\AA_{1}=\mcP $.
\end{ex}

\subsection{The {\BoV} tensor product}\label[appendix]{sec-BV} 
Now suppose we have two operads $(\mcP ,\gamma )$ and $(\mcQ, \delta
)$ as in \cref{def-operad}.  Then we have the category $\mcP
[\mcQ[\mcT]]$ whose objects are $\mcP $-algebras in the category of
$\mcQ$-spaces. It is known that there is an operad $\mcR $ such that
this category is {\eqt} to $\mcR [\mcT]$, which is also {\eqt} to
$\mcQ [\mcP[\mcT]]$, the category of $\mcQ $-algebras in the category
of $\mcP$-spaces.  An object in $\mcP [\mcQ[\mcT]]\simeq \mcQ
[\mcP[\mcT]]$ is both a $\mcP $-algebra and a $\mcQ$-algebra in which
the two structure interchange as in \cref{def-interchange}.

This leads to a symmetric monoidal structure, the {\bf Boardman-Vogt
tensor product} (BV product for short) of \cite[Definition 2.14, page
41]{Boardman-Vogt}, {\em on the category of operads}, which we will
write as
\begin{displaymath}
\mcR = \mcP \otimes_{\BV} \mcQ.
\end{displaymath}

Roughly speaking, $\mcP \otimes_{\BV} \mcQ$ is the quotient of the
coproduct $\mcP \amalg \mcQ$ of \cref{def-coproduct-operad} imposed by
the interchange requirement. 

\begin{ex}\label{ex-BV-not-homotopy-inv}
{\bf The BV product does not preserve homotopy equivalence.}  We know
that the operads $\Assoc$ of \cref{def-MN}, $\AA_{\infty }$ of
\cref{def-Stash-op} and $\EE_{1}$ of \cref{def-little-cubes} are
homotopy equivalent.  We also know that $\Assoc\otimes_{BV}\Assoc\cong
\Comm$ by an argument originally due to Beno Eckmann and Peter Hilton
\cite{Eckmann-Hilton}.  They showed that if a set is equipped with two
associative multiplication maps that interchange in the sense of
\cref{def-interchange}, then they must coincide and be commutative.
On the other hand, Gerald Dunn \cite{Dunn} shows that
$\EE_{1}\otimes_{BV} \EE_{1}\simeq \EE_{2}$, but $\EE_{2}$ is not
equivalent to $\Comm$.  In \cite{Moerdijk23} Ieke Moerdijk suggests a
solution involving dendroidal sets, which are introduced by him
and Ittay Weiss in \cite{Moerdijk-Weiss07}, but that is a story for
another day.
\end{ex}

We learned the following description of the BV product from the paper
\cite[Definition 1.2]{Dwyer-Hess} by Bill Dwyer and Kathryn Hess.
Like every published definition that we know of, theirs is written
with symmetric operads in mind, but it can be modified to include the
nonsymmetric case.

\begin{defin}\label{def-composite}
{\bf The composite of two symmetric sequences.}
\cite[Notation 1.1]{Dwyer-Hess}
For any two symmetric sequences (see \cref{def-operad}\cref{def-operadii})
\[
X = \{X(n)\}_{n \geq 0}, \quad 
Y = \{Y(n)\}_{n \geq 0}
\]

\noindent of simplicial sets, a representative of a typical element of
arity $j$ in the {\bf composition product} $X \circ Y$ of the two
sequences is denoted by 
\[
(x; y_1, \ldots, y_k; \tau),
\]

\noindent where
\begin{align*}
x & \in  X (k) &
y_{s} & \in Y (j_{s})   
\mbox{ with }  \sum_{s=1}^k j_{s} = j&
\mbox{and } \tau & \in \Sigma_{j}. 
\end{align*}

\noindent
The right action of $\nu \in \Sigma_{j}$ on such an element is given by
\[
(x; y_1, \ldots, y_k; \tau ) \cdot \nu 
= (x; y_1, \ldots, y_k; \tau\nu ).
\]




The equivalence relation on representatives of elements of $X \circ Y$
satisfies
\[
(x; y_1, \ldots, y_k; \tau_1 \oplus \cdots \oplus \tau_k) 
\sim (x; y_1 \cdot \tau_1^{-1}, \ldots, y_k \cdot \tau_k^{-1}; \mathrm{Id}),
\]

\noindent where $\tau_1 \oplus \cdots \oplus \tau_k$ is the block sum
of \cref{eq-symm-action},
 and
\[
(x \cdot \sigma^{-1}; y_1, \ldots, y_j; \mathrm{Id}) 
\sim (x; y_{\sigma(1)}, \ldots, y_{\sigma(j)}; \mathrm{Id}),
\]

\noindent for all $\sigma \in \Sigma_{j}$ and $\tau_s \in \Sigma
_{j_{s}}$ as above, where $x$ and $y_{s}$ are as above.

If $\mcP$ is an operad, it has an {\eqvr} multiplication map 
\[
\gamma  : \mcP \circ \mcP \to \mcP
\]

\noindent as in \cref{eq-gamma-K}. For $p \in \mcP(k)$ and $p_s \in
\mcP(j_{s})$ for $j_{s}$ as above,
we write
\begin{numequation}\label{eq-gamma-p}
\begin{split}
p(p_1, \ldots, p_k)
 := \gamma (p; p_1, \ldots, p_k; \mathrm{Id}) \in \mcP(j).
\end{split}
\end{numequation}%

\noindent Note that since $\gamma $ is equivariant as in
\cref{eq-symm-action}, it is specified by its values on elements of
$\mcP \circ \mcP$ with representatives of the form
\[
(p; p_1, \ldots, p_k; \mathrm{Id}).
\]

For {\bf operadic sequences} as in \cref{def-operad}\cref{def-operadi} a
typical element of arity $n$ in the composition product is denoted by
simply
\[
(x; y_1, \ldots, y_k),
\]

\noindent with no symmetric group element.
\end{defin}

\begin{defin}\label{def-BV-prod}
The {\bf Boardman-Vogt tensor product} of operads $\mcP$
and $\mcQ$ is the operad 
\[
\mcP \otimes_{BV} \mcQ
\]

\noindent that is the quotient of the coproduct $\mcP \amalg \mcQ$ of
operads (see \cref{def-coproduct-operad}) by the equivalence relation
generated by
\[
\gamma (p; \underbrace{q, \ldots, q}_{k}; \mathrm{Id})
 \;\sim\; 
\gamma (q; \underbrace{p, \ldots, p}_{\ell }; \xi_{k,\ell})
\]
for all $p \in P(k)$ and $q \in Q(\ell)$, where $\xi_{k,\ell} \in
\Sigma _{k\ell}$ is the transpose permutation that ``exchanges rows
and columns.'' That is, for each $m$ with $1\leq m\leq k\ell $, there
are unique integers $i$ and $j$ with $1 \leq i \leq k$ and $1 \leq j
\leq \ell$ for which $m = (i-1)\ell + j \leq k\ell$, and we have
\[
\xi_{k,\ell}(m) = (j-1)k + i.
\]

In the nonsymmetric case the  equivalence relation
is generated by
\begin{displaymath}
\gamma (p; \underbrace{q, \ldots, q}_{k})
 \;\sim\; 
\gamma (q; \underbrace{p, \ldots, p}_{\ell }).
\end{displaymath}
\end{defin}

The only BV product we will need is $\EE\AA_{2}:=\EE_{1}\otimes_{\BV}
\AA_{2}$, which appears in \cite[Theorem C]{BHLS} and is described in
\cref{ex-E1A2}.

\begin{ex}\label{ex-BV-unit}
{\bf The Boardman-Vogt unit.} Let $\mcQ $ be the Stasheff operad
$\AA_{1}$ of \cref{def-Stash-op}.  In \cref{ex-A1} we saw that $\mcP
\amalg \AA_{1}=\mcP $ for any operad $\mcP$.  This implies that $\mcP
\otimes_{BV} \AA_{1}=\mcP $.  A similar argument can be made for
$\AA_{1}\otimes_{BV}\mcQ$ for any $\mcQ$. Thus $\AA_{1}$ is the
Boardman-Vogt unit.
\end{ex}

\begin{ex}\label{ex-E1A2}
 The operad $\EE\AA_{2}:=\EE_{1}\otimes_{\BV}\AA_{2}$ (where
$\AA_{2}$ and $\EE_{1}$ are as in
\cref{def-Stash-op,def-little-cubes}) is of interest because it
appears in \cite[Theorem C]{BHLS}. In Section 5 of that paper, the
authors construct Adams operations on $BP\langle n \rangle$ (which is
known to have an $\EE_{3}$-structure) as $\EE\AA_{2}$-algebra
automorphisms, that being the strongest structure for which their
proof works.  It is not known to be the strongest structure preserved
by their operations.

An $\AA_{2}$-structure on a spectrum $X$ is a map
\begin{displaymath}
X\vee (X\wedge X) \to X
\end{displaymath}

\noindent which is the identity on the first summand and a unital
multiplication $\theta^{X,\AA_{2}}_{2} $ as in \cref{eq-theta-O}
on the second one.  

For an $\EE_{1}$-ring spectrum $Y$ we have the map
\begin{displaymath}
\theta^{\EE_{1},Y}_{j}
:\EE_{1} (j)_{+}\wedge_{\Sigma_{j}}Y^{\otimes j}\to Y
\end{displaymath}

\noindent of \cref{eq-theta-O}. The space $\EE_{1} (j)$ is homotopy {\eqt } to
$\Sigma_{j}$.

Then for an $\EE\AA_{2}$-spectrum $Z$, the following is required to
commute as in \cref{def-interchange}.
\begin{displaymath}
\xymatrix
@R=2mm
@C=20mm
{
{(\EE_{1} (j)_{+}\times_{\Sigma_{j}}\Sigma_{2j}/\Sigma_{j})
     \wedge_{\Sigma_{2j}}Z^{\otimes 2j}}
  \ar@{=}[d]^(.5){}\\
{\EE_{1} (j)_{+}\wedge_{\Sigma_{j}}(Z\wedge Z)^{\otimes j}}
    \ar[r]^(.625){\theta^{\EE_{1},Z\wedge Z}_{j}}
    \ar[dd]_(.5){\EE_{1} (j)_{+}\wedge_{\Sigma_{j}}
             (\theta^{Z,\AA_{2}}_{2})^{\otimes j} }
  &{Z\wedge Z}\ar[dd]^(.5){\theta^{Z,\AA_{2}}_{2} }\\
{}\\
{\EE_{1} (j)_{+}\wedge_{\Sigma_{j}}Z^{\otimes j}}
    \ar[r]^(.6){\theta^{\EE_{1},Z}_{j}}
  &{Z}
}
\end{displaymath}

Hence an $\EE\AA_{2}$-structure is slightly stronger than an
$\EE_{1}$-structure but weaker than an $\EE_{2}$-structure.  It is
known that for an $\EE\AA_{2}$-ring $R$, $\THH (R)$ has an
$\AA_{2}$-structure, meaning a unital multiplication with no
associativity or commutativity condition.
\end{ex}

\bibliography{../math}
\bibliographystyle{alpha}

\end{document}